\documentclass[11pt,reqno]{amsart}

\usepackage[T1]{fontenc}
\usepackage{lmodern}
\usepackage{microtype}
\usepackage{amsmath,amssymb,mathrsfs}
\usepackage{enumitem}
\usepackage{geometry}
\usepackage[colorlinks=true,linkcolor=blue,citecolor=blue,urlcolor=blue]{hyperref}

\hypersetup{
 pdftitle={Secular Instability and the Turning-Point Principle for Rigidly Rotating Viscous Stars},
 pdfauthor={Ming Cheng, Zhiwu Lin, and Yucong Wang},
 pdfsubject={Axisymmetric spectral stability for rotating Navier--Stokes--Poisson stars},
 pdfkeywords={rotating stars, Navier--Stokes--Poisson, turning-point principle, secular instability}
}
\allowdisplaybreaks

\newtheorem{theorem}{Theorem}[section]
\newtheorem{proposition}[theorem]{Proposition}
\newtheorem{lemma}[theorem]{Lemma}
\newtheorem{corollary}[theorem]{Corollary}
\newtheorem{example}[theorem]{Example}
\theoremstyle{remark}
\newtheorem{remark}[theorem]{Remark}
\newtheorem{hypothesis}{Hypothesis}[section]

\numberwithin{equation}{section}
\setlist[enumerate]{leftmargin=2.1em}

\newcommand{\R}{\mathbb R}
\newcommand{\C}{\mathbb C}
\newcommand{\cA}{\mathcal A}
\newcommand{\cC}{\mathcal C}
\newcommand{\Ran}{\operatorname{Ran}}
\newcommand{\Ker}{\operatorname{Ker}}
\newcommand{\Dom}{\operatorname{Dom}}
\newcommand{\Ree}{\operatorname{Re}}
\newcommand{\Imm}{\operatorname{Im}}
\newcommand{\Span}{\operatorname{span}}
\newcommand{\dev}{\operatorname{dev}}
\newcommand{\abs}[1]{\lvert #1\rvert}
\newcommand{\norm}[1]{\lVert #1\rVert}
\newcommand{\ip}[2]{\langle #1,#2\rangle}
\newcommand{\half}{\tfrac12}

\title[Secular instability of rotating viscous stars]
{Secular Instability and the Turning-Point Principle\\
for Rigidly Rotating Viscous Stars}

\author{Ming Cheng}
\address{College of Mathematics, Jilin University, Changchun 130012, China}
\email{mcheng314@jlu.edu.cn}

\author{Zhiwu Lin}
\address{School of Mathematical Sciences, Fudan University, Shanghai 200433, China}
\email{zwlin@fudan.edu.cn}

\author{Yucong Wang}
\address{School of Mathematics and Computational Science, Xiangtan University,
Xiangtan 411105, China}
\email{yucongwang666@163.com}


\begin{document}

\begin{abstract}
We study axisymmetric stability of rigidly rotating viscous stars
modeled by the free-boundary Navier--Stokes--Poisson (NSP) system.  The
unstable index of the linearized NSP generator, counted with
Riesz algebraic multiplicity, equals the Morse index of the augmented energy
at fixed mass and total angular momentum.  We prove an abstract
Kelvin--Tait--Chetaev theorem for damped gyroscopic equations with finite
negative stiffness index, requiring no compactness assumptions and no
spectral gap at zero; the stiffness kernel may be infinite-dimensional.
In the NSP application, the viscous dissipation is degenerate: its
axisymmetric kernel, generated by rigid axial translation and rigid rotation,
is projected out before applying the abstract theorem, while conservation laws
and a Routh reduction transfer the resulting instability index back to the
full NSP generator.  For slowly rotating branches with fixed total angular
momentum, viscous instability begins at the continuation of a first
nondegenerate spherical mass maximum.  By contrast, for every sufficiently
small nonzero angular momentum, the corresponding Euler--Poisson star remains
axisymmetrically spectrally stable on an interval beyond this maximum.  Thus
dissipation can destabilize a rotating star before the corresponding inviscid
star becomes unstable.  This separation of the stability thresholds results
from viscous redistribution of angular momentum, which removes the inviscid
constraints on its material distribution while preserving its total value.
\end{abstract}

\maketitle

\section{Introduction}
\label{sec:introduction}

We consider a self-gravitating viscous gas occupying
$\Omega(t)=\{x\in\mathbb R^3:\rho(t,x)>0\}$.  The density $\rho$, velocity $v$, and gravitational potential $V_\rho$
satisfy
\[
\left\{
\begin{aligned}
 \partial_t\rho+\nabla\cdot(\rho v)&=0,
       &&x\in\Omega(t),\\
 \rho(\partial_t v+v\cdot\nabla v)+\nabla P(\rho)
   &=\nabla\cdot\mathbb T(v)-\rho\nabla V_\rho,
       &&x\in\Omega(t),\\
 \Delta V_\rho&=4\pi\rho,\qquad V_\rho\to0\ (|x|\to\infty),
       &&x\in\mathbb R^3.
\end{aligned}
\right.
\]
where
\begin{equation}
 \mathbb T(v)=2\nu_s\operatorname{dev}e(v)
               +\nu_b(\nabla\cdot v)I,
 \qquad e(v)=\tfrac12(\nabla v+\nabla v^T).                  \label{eq:viscous-stress}
\end{equation}
Here $P$ is the pressure, $\operatorname{dev}A=A-\tfrac13(\operatorname{tr}A)I$, and
$\nu_s,\nu_b>0$ are the shear and bulk viscosity coefficients.  On
$\Gamma(t)=\partial\Omega(t)$,
\begin{equation}
 \rho=0,\qquad (P(\rho)I-\mathbb T(v))n=0,
 \qquad \mathcal V_\Gamma=v\cdot n,                         \label{eq:NSP-boundary}
\end{equation}
where $n$ is the outward normal and $\mathcal V_\Gamma$ the normal
velocity of the boundary.  The gravitational normalization is
$V_\rho(x)=-\int_{\mathbb R^3}\rho(y)/|x-y|\,dy$.

The corresponding Euler--Poisson system is obtained by setting
$\mathbb T=0$.  In cylindrical coordinates $(r,\theta,z)$, with
$e_r,e_\theta,e_z$ denoting the corresponding unit vectors, consider the
barotropic rotating class
\[
 v=r\omega(r)e_\theta.
\]
Then
\[
 \nabla\cdot v=0,\qquad
 e_{r\theta}(v)=e_{\theta r}(v)
 =\frac12\left(\partial_r v_\theta-\frac{v_\theta}{r}\right)
 =\frac12 r\omega'(r),
\]
where $e_{r\theta}(v)$ and $e_{\theta r}(v)$ denote the $(r,\theta)$ and
$(\theta,r)$ components of the symmetric strain tensor $e(v)$ in the
orthonormal cylindrical frame $(e_r,e_\theta,e_z)$.  Thus differential
rotation produces viscous shear.  For a smooth stationary star satisfying
\eqref{eq:NSP-boundary}, the NSP energy-dissipation identity
\eqref{eq:NSP-energy-dissipation} gives
\[
 \int_{\Omega}
 \left(2\nu_s|\dev e(v)|^2+\nu_b|\nabla\cdot v|^2\right)\,dx=0.
\]
Since $\nu_s,\nu_b>0$,
\[
 \dev e(v)=0,\qquad \nabla\cdot v=0,
\]
and hence $e(v)=0$.  In particular,
$e_{r\theta}(v)=\frac12r\omega'(r)=0$, so $\omega'(r)=0$ on each connected
component of the support.  We therefore consider rigidly rotating equilibria
$(\bar\rho(r,z),\omega r e_\theta)$, for which $\mathbb T(\bar v)=0$.
Define the enthalpy and internal-energy density by
\begin{equation}
 h(\rho)=\int_0^\rho\frac{P'(s)}s\,ds,
 \qquad \Phi(\rho)=\int_0^\rho h(s)\,ds.                    \label{eq:enthalpy}
\end{equation}
The equilibrium equation is
\begin{equation}
 h(\bar\rho)+V_{\bar\rho}-\tfrac12\omega^2r^2+c=0
 \quad\hbox{in }\Omega_{\bar\rho}:=\{\bar\rho>0\},         \label{eq:equilibrium}
\end{equation}
where $c$ is constant.  Hence a rigidly rotating star is a common
equilibrium of the viscous and inviscid systems, although the two stability
criteria need not agree.

For axisymmetric perturbations of compressible stars with a vacuum free
surface, two questions are particularly natural: is the stability of a
viscous rotating star determined by the mass curve along a family with fixed
total angular momentum, and how does dissipation affect the stability
threshold relative to the Euler--Poisson system?  We show that, for slowly
rotating fixed-angular-momentum branches, the viscous stability changes
precisely at the continuation of the first nondegenerate mass maximum.
Moreover, dissipation has a destabilizing effect: for every sufficiently
small nonzero angular momentum, there is an interval beyond the viscous
turning point on which the viscous star is unstable while the corresponding
Euler--Poisson star remains axisymmetrically spectrally stable.

The mechanism behind this separation is the different angular-momentum
constraints of the two dynamics.  In axisymmetric perfect-fluid motion, the
specific angular momentum $rv_\theta$ is transported with each particle, so
its material distribution is conserved.  Viscosity redistributes angular
momentum while preserving only its total value and dissipating energy,
thereby enlarging the class of dynamically accessible perturbations.  This
is consistent with the classical distinction between \emph{dynamical}
instability of a perfect fluid and \emph{secular} instability induced by
dissipation.  For Maclaurin spheroids, for example, viscosity shifts the
nonaxisymmetric bar-mode threshold to the earlier Jacobi bifurcation
\cite{Chandrasekhar1969,RobertsStewartson1963}.  Numerical studies of
compressible rotating stars \cite{James1964} and compressible ellipsoidal
models \cite{LaiRasioShapiro1993} likewise illustrate the dependence of the
stability threshold on the admissible perturbations and constraints; see also
\cite{Tassoul1978}.

We first prove an instability-index formula
\[
 N_+(\mathcal A_{\rm NSP})=n^-(Q_J|_{X_0}).
\]
Here $N_+$ counts the unstable eigenvalues of the unrestricted
axisymmetric generator with Riesz algebraic multiplicity, $X_0$ is the
space of mass-zero density perturbations, and $Q_J$ is the constrained
second variation of the energy after minimizing over velocities at fixed total
angular momentum $J$.  These objects are defined below.  For a slowly
rotating family with $J=\kappa$, parameterized by center density $\mu$,
the formula gives
\[
 N_+\bigl(\mathcal A_{\rm NSP}(\mu,\kappa)\bigr)
 =\begin{cases}
 0,&\partial_\mu M_\kappa(\mu)\ge0,\\
 1,&\partial_\mu M_\kappa(\mu)<0,
 \end{cases}
\]
on a fixed compact interval around the continuation of the first
nondegenerate spherical mass maximum.  Thus the viscous transition occurs
at the fixed-$J$ mass maximum.  The inviscid reduced density form satisfies
\begin{equation}
 \widetilde Q_{\rm EP}=Q_J+\Delta Q,\qquad\Delta Q\ge0,      \label{eq:EP-comparison-intro}
\end{equation}
where the additional term comes from preservation of the angular-momentum
distribution.  For every sufficiently small nonzero $\kappa$, we prove
that it postpones the inviscid transition: on a nontrivial interval just
beyond the viscous turning point, the viscous star has one unstable mode
and the inviscid star is axisymmetrically spectrally stable.  Spectral
stability here means absence of spectrum in the open right half-plane.

For rotating Euler--Poisson stars with Rayleigh-stable angular velocity,
Lin and Wang \cite{LinWang2023} obtained a stability criterion,
an unstable-index count, and an exponential trichotomy.  They proved
that fixing an angular-velocity profile need not give a turning-point
principle, whereas fixing the angular-momentum distribution gives the
appropriate inviscid turning-point law.  The viscous problem fixes only
total angular momentum.  The difference in \eqref{eq:EP-comparison-intro}
therefore compares the quadratic forms on different invariant constraint
spaces, even though the equilibrium itself is the same.  We prove
strict separation of the spectral thresholds by analyzing the equality
case of this comparison.

The situation is different without rotation.  Lin and Zeng
\cite{LinZeng2022} determined the Euler--Poisson unstable index from the
oriented mass--radius curve.  Cheng, Lin, and Wang
\cite{ChengLinWang2026} proved the corresponding viscous radial index
formula away from mass extrema, together with nonlinear stability and
instability results.  The viscous and inviscid indices then coincide.
In particular, the rotational correction in
\eqref{eq:EP-comparison-intro} vanishes.  Nonlinear instability of viscous
Lane--Emden stars with $6/5<\gamma<4/3$ was proved earlier by Jang and Tice
\cite{JangTice2013}.  Corollary~\ref{cor:nonrotating-all-mu} below extends
the nonrotating linear count to mass extrema as well.

To treat the rotating viscous problem, we prove a Kelvin--Tait--Chetaev
(KTC) theorem for damped gyroscopic equations with finite negative
stiffness index.  On a common form domain, the theorem assumes coercive
damping and bounded stiffness and gyroscopic forms.  It requires no
compact embedding or compact mass or gyroscopic preconditioner; the
stiffness may have an infinite-dimensional kernel and positive spectrum
accumulating at zero.  The comparison with earlier KTC results is given in
Section~\ref{sec:KTC}.

For the stellar equations, the dissipation itself is degenerate: its
axisymmetric kernel consists of rigid axial translation and rigid
rotation.  Conservation laws and a Routh reduction remove these two
velocities and give a coercive damping form.  The infinite-dimensional
kernel of the displacement stiffness remains and is handled by the
abstract theorem.  The reduction preserves nonzero generalized
eigenspaces, so the instability count applies to the full axisymmetric
generator.  For the turning-point and strict-comparison results we also
prove the required slow-branch regularity throughout
$4/3<\gamma_0<2$, including the physical exponent $\gamma_0=5/3$.

Relativistic turning-point principles of Sorkin and
Friedman--Ipser--Sorkin \cite{Sorkin1981,Sorkin1982,FriedmanIpserSorkin1988}
give secular criteria along sequences with fixed angular momentum.
Canonical-energy and thermodynamic formulations distinguish preservation
of an angular-momentum distribution from preservation of its integral,
with additional particle-number and entropy constraints
\cite{FriedmanSchutz1978,GreenSchiffrinWald2014,SchiffrinWald2014}.
Numerical studies distinguish turning lines from dynamical neutral lines
\cite{TakamiRezzollaYoshida2011,WeihMostRezzolla2018}.  Those criteria and
comparisons concern different models; the present result counts the
unstable spectrum of a specified Newtonian dissipative generator.
For the precise statements, define the energy, mass, and total axial
angular momentum by
\begin{align}
 E(\rho,v)&=\int_{\mathbb R^3}
 \left(\tfrac12\rho|v|^2+\Phi(\rho)+\tfrac12\rho V_\rho\right)dx,
                                                               \label{eq:energy}\\
 M(\rho)&=\int_{\mathbb R^3}\rho\,dx,
 \qquad J(\rho,v)=\int_{\mathbb R^3}\rho r v_\theta\,dx.     \label{eq:MJ}
\end{align}
Equation \eqref{eq:equilibrium} is the Euler--Lagrange equation for
$E+cM-\omega J$.  Along smooth viscous solutions, $M,J$ are conserved
and $E$ decreases.

The subscript ${\rm ax}$ denotes axisymmetric scalar functions or vector
fields: $f(R_\theta x)=f(x)$ and $u(R_\theta x)=R_\theta u(x)$,
respectively.  Put
\begin{align}
 X_{\bar\rho}&=L^2_{h'(\bar\rho),{\rm ax}}(\Omega_{\bar\rho};\mathbb C),
 &\|q\|_{X_{\bar\rho}}^2&=\int_{\Omega_{\bar\rho}}h'(\bar\rho)|q|^2dx,
 \notag\\
 X_0&=\left\{q\in X_{\bar\rho}:\int q\,dx=0\right\},
 & I(\bar\rho)&=\int r^2\bar\rho\,dx.                       \label{eq:intro-spaces}
\end{align}
Density perturbations are extended by zero outside the star; unlabelled
spatial integrals below are over $\mathbb R^3$.  Define the Hermitian
quadratic form
\begin{equation}
 Q_J[q]=\int h'(\bar\rho)|q|^2dx+\int V_q\overline q\,dx
       +\frac{\omega^2}{I(\bar\rho)}
          \left|\int r^2q\,dx\right|^2.                    \label{eq:QJ-intro}
\end{equation}
Here and below, the second variation is first computed on smooth real
perturbations whose density component is supported inside the star, and
then extended continuously as a Hermitian form to the energy space.
Its last term comes from minimizing the velocity part of the second
variation at fixed $J$.  For a Hermitian form $Q$, $n^-(Q|_Y)$ is the
maximal dimension of a subspace of $Y$ on which $Q$ is negative definite.

Let $\mathcal A_{\rm NSP}$ be the unrestricted axisymmetric generator on
$X_{\bar\rho}\times L^2_{\bar\rho,{\rm ax}}(\Omega_{\bar\rho};\mathbb C^3)$,
with domain and action defined in
\eqref{eq:full-generator-domain}--\eqref{eq:full-generator}.
For its isolated eigenvalues with finite-rank Riesz projections, denoted
by $\operatorname{spec}_{\rm iso}(\mathcal A_{\rm NSP})$, set
\begin{equation}
 N_+(\mathcal A_{\rm NSP})
 =\sum_{\substack{\lambda\in\operatorname{spec}_{\rm iso}(\mathcal A_{\rm NSP})\\
                   \Ree\lambda>0}}
    \operatorname{algmult}(\lambda;\mathcal A_{\rm NSP}),       \label{eq:unstable-count}
\end{equation}
where algebraic multiplicity is the rank of the Riesz projection.

\begin{hypothesis}
\label{hyp:EOS}
The pressure satisfies
\[
 P\in C^4((0,\infty))\cap C^1([0,\infty)),\qquad
 P(0)=0,\qquad P'(s)>0\quad(s>0),
\]
and the enthalpy \eqref{eq:enthalpy} is finite at zero.  Near vacuum,
there are $K_0,a_0>0$ and $\gamma_0\in(4/3,2)$ such that
\begin{equation}
 \frac{d^j}{ds^j}\bigl(P(s)-K_0s^{\gamma_0}\bigr)
 =O(s^{\gamma_0+a_0-j}),\qquad 0\le j\le3,\quad s\downarrow0.
                                                               \label{eq:EOS-asymptotic}
\end{equation}
\end{hypothesis}

We call $(\bar\rho,\omega r e_\theta)$ a \emph{regular rigidly rotating
equilibrium} if it satisfies \eqref{eq:equilibrium}, $\bar\rho$ is a
nonzero nonnegative axisymmetric function in $C_c(\mathbb R^3)$, and
$\Omega_{\bar\rho}$ is connected and Lipschitz.  All analytic properties
needed for the following theorem are derived from these conditions and
Hypothesis~\ref{hyp:EOS} in
Proposition~\ref{prop:automatic-admissibility}.

\begin{theorem}
\label{thm:main-index}
Let $(\bar\rho,\omega r e_\theta)$ be a regular rigidly rotating
equilibrium whose pressure law satisfies Hypothesis~\ref{hyp:EOS}.
The spectrum of $\mathcal A_{\rm NSP}$ in $\{\Ree\lambda>0\}$ consists
of finitely many eigenvalues of finite algebraic multiplicity, and
\begin{equation}
 N_+(\mathcal A_{\rm NSP})
 =n^-\!\left(Q_J\big|_{X_0}\right)
 =n^-\!\left(D^2(E+cM-\omega J)
       \big|_{\{\delta M=\delta J=0\}_{\rm ax}}\right).       \label{eq:main-index}
\end{equation}
Consequently the equilibrium has no exponentially growing axisymmetric
mode if and only if $Q_J\ge0$ on $X_0$.
\end{theorem}

We impose no mass, axial-momentum, total-angular-momentum, centering, or
$z$-parity restriction on the spectral problem.  The first four conditions
are automatic on every nonzero generalized eigenspace by the conservation
laws and the axial first-moment identity, and the Routh transformation
preserves its Jordan structure.

For the spherical family let $M_0(\mu)$ and $R_0(\mu)$ denote its mass
and radius at center density $\mu$.  Suppose its first critical point
$\mu_*$ satisfies
\begin{equation}
 \partial_\mu M_0(\mu)>0\ (0<\mu<\mu_*),\qquad
 \partial_\mu M_0(\mu_*)=0,\qquad
 \partial_\mu^2M_0(\mu_*)<0.                              \label{eq:nondegenerate-maximum}
\end{equation}
Set
\[
 \mu_e=\inf\left\{\mu>0:
   \partial_\mu\left(M_0(\mu)/R_0(\mu)\right)=0\right\},
\]
with $\inf\varnothing=+\infty$.  By
Lemma~\ref{lem:first-maximum-before-mutilde}, $\mu_*<\mu_e$.
Fix $I_\mu=[\mu_-,\mu_+]\Subset(0,\mu_e)$ containing $\mu_*$
in its interior and no other zero of $\partial_\mu M_0$.
The fixed-angular-momentum branch, with $J=\kappa$, is written
$(\rho_{\mu,\kappa},\omega(\mu,\kappa)r e_\theta)$, and
$M_\kappa(\mu)=\int\rho_{\mu,\kappa}\,dx$.

\begin{theorem}
\label{thm:turning-point}
Assume Hypothesis~\ref{hyp:EOS} and
\eqref{eq:nondegenerate-maximum}.  For every sufficiently small fixed
$\kappa$, the center-normalized slow branch of
Proposition~\ref{prop:fixed-J-branch} exists on $I_\mu$, and $M_\kappa$
has there a unique critical point $\mu_*^\kappa=\mu_*+O(\kappa^2)$,
which is a strict maximum.  For every $\mu\in I_\mu$,
\begin{equation}
 N_+\bigl(\mathcal A_{\rm NSP}(\mu,\kappa)\bigr)
 =\begin{cases}
 0,&\partial_\mu M_\kappa(\mu)\ge0,\\
 1,&\partial_\mu M_\kappa(\mu)<0.
 \end{cases}                                               \label{eq:turning-count}
\end{equation}
At $\mu=\mu_*^\kappa$,
\begin{equation}
 \Ker\!\left(Q_J(\mu_*^\kappa,\kappa)|_{X_0}\right)
 =\operatorname{span}\left\{
  \left.\partial_z\rho_{\mu,\kappa}\right|_{\mu=\mu_*^\kappa},
  \left.\partial_\mu\rho_{\mu,\kappa}\right|_{\mu=\mu_*^\kappa}
 \right\},                                                 \label{eq:main-turning-kernel}
\end{equation}
and there is no spectrum in the open right half-plane.  Thus the viscous
transition near the first spherical turning point occurs exactly at the
corresponding fixed-$J$ mass maximum.
\end{theorem}

\begin{theorem}
\label{thm:strict-separation}
Assume the hypotheses of Theorem~\ref{thm:turning-point}.  For every
sufficiently small $\kappa\ne0$, $\widetilde Q_{\rm EP}$ is nonnegative
at $\mu_*^\kappa$ and remains nonnegative, modulo axial translation, for
$\mu_*^\kappa\le\mu\le\mu_*^\kappa+\delta_\kappa$, where
$0<\delta_\kappa<\mu_+-\mu_*^\kappa$.  In particular, the local
inviscid threshold
\begin{equation}
 \widehat\mu_\kappa:=\inf\left\{\mu\in(\mu_*^\kappa,\mu_+]:
 n^-\!\left(\widetilde Q_{\rm EP}(\mu,\kappa)|_{X_0}\right)>0
 \right\}                                                  \label{eq:inviscid-threshold}
\end{equation}
satisfies $\widehat\mu_\kappa>\mu_*^\kappa$, with
$\inf\varnothing=+\infty$.  Thus, on a nontrivial parameter interval,
the viscous star has one unstable axisymmetric mode while the corresponding
inviscid star is axisymmetrically spectrally stable.
\end{theorem}

The interval length $\delta_\kappa$ may tend to zero as $\kappa\to0$.
The assertions concern spectrum in the open right half-plane and do not
imply bounded center dynamics at the turning point.

The support geometry and joint parameter regularity used in these
theorems are derived in Propositions~\ref{prop:slow-branch-full-range}
and \ref{prop:slow-vacuum-tangents}, with the moving-interface estimates
in Appendix~\ref{app:slow-branch}.  This supplies the regularity needed
here in addition to the classical slow-rotation constructions
\cite{Heilig1994,JangMakino2017,JangMakino2019,StraussWu2017}; related
physical-vacuum evolution theories include
\cite{Jang2010,LuoXinZeng2014}.  Softened equations of state with
low-density exponent $5/3$ provide examples, subject to the scalar
condition \eqref{eq:nondegenerate-maximum}; see
Example~\ref{ex:asymptotic-polytropes}.

The inviscid limit is singular at the level of the unstable index.  Fix
$\kappa\ne0$ sufficiently small and
$\mu\in(\mu_*^\kappa,\mu_*^\kappa+\delta_\kappa]$.  If both viscosities
are multiplied by $\varepsilon>0$, then
\[
 N_+\bigl(\mathcal A_\varepsilon(\mu,\kappa)\bigr)=1
 \quad(\varepsilon>0),\qquad
 N_+\bigl(\mathcal A_0(\mu,\kappa)\bigr)=0,
\]
where $\mathcal A_0$ is the Euler--Poisson generator.  At the endpoint
$\mu=\mu_*^\kappa$, both systems have no spectrum in the open right
half-plane, although their zero-frequency structures need not coincide.
At zero viscosity, the angular-momentum-distribution constraints reappear.
This change in the interior conservation laws differs from the usual no-slip
inviscid limit, where mismatching viscous and Euler boundary conditions can
produce Prandtl-type boundary layers.  The index discontinuity alone does not
determine individual eigenvalue limits or rule out finite-time convergence of
solutions.

We describe the main arguments. We first prove an abstract
Kelvin--Tait--Chetaev (KTC) theorem that isolates the spectral mechanism
needed for the viscous stellar problem. The reduced displacement equation has the
weak form
\begin{equation}
 Ju_{tt}+\mathsf D u_t+\mathsf G u_t+\mathsf K u=0,
 \label{eq:abstract-intro}
\end{equation}
where $\mathsf D$ is the coercive damping form operator, $\mathsf G$ is
skew-Hermitian, $\mathsf K$ is the symmetric stiffness operator with finite
negative index, and $J$ is the mass operator induced by the embedding of the
damping form domain into the velocity space.  To study spectral instability,
we seek normal modes
\[
 u(t)=e^{\lambda t}\widehat u,\qquad \lambda\in\mathbb C,
\]
where $\lambda$ is the temporal spectral parameter; $\Re\lambda>0$
corresponds to exponential growth.  Substitution gives the quadratic
characteristic pencil
\[
 \mathsf K+\lambda(\mathsf D+\mathsf G)+\lambda^2J.
\]
Let $R_d$ be the Riesz isomorphism induced by the damping form $d$.
Preconditioning by $R_d^{-1}$ and setting
\[
 A=R_d^{-1}\mathsf K,\qquad
 \Gamma=R_d^{-1}\mathsf G,\qquad
 M=R_d^{-1}J,
\]
gives
\[
 \mathcal P(\lambda)
 =A+\lambda(I+\Gamma)+\lambda^2M,
\]
since $R_d^{-1}\mathsf D=I$ on the damping form space.  Decomposing $A$
into its nonnegative and negative spectral parts, the latter has finite
rank because the stiffness has finite negative index.  For
$\Re\lambda\ge0$, $\lambda\ne0$, multiplication of the nonnegative part
of the pencil by $\bar\lambda$ yields a coercive form, so the pencil is
Fredholm of index zero without any compactness assumption on $M$ or
$\Gamma$.  The main issue is the possible accumulation of characteristic
values at the origin, since the nonnegative stiffness may have no spectral
gap there.  Projecting a characteristic vector onto the finite-dimensional
negative spectral space gives
\[
 \|u_-\|=O(|\lambda|)
 \qquad (\lambda\to0,\ \Re\lambda\ge0),
\]
and substitution into the scalar energy identity excludes all nonzero
characteristic values sufficiently close to the origin.  One can therefore
choose a fixed contour in the right half-plane and deform the mass and
gyroscopic terms to zero while keeping the coercive damping fixed.
The Gohberg--Sigal argument principle then shows that the total algebraic
multiplicity of the unstable characteristic values equals the negative
stiffness index.  Finally, a domain-valued root-chain correspondence
identifies the partial multiplicities of the quadratic pencil with the
Jordan structure of the first-order generator, and hence transfers the
pencil count to its unstable spectral index.

We next apply the abstract KTC theorem to the linearized NSP system.
The viscous dissipation is not coercive on the full axisymmetric velocity
space: its kernel is generated by the rigid axial translation $e_z$ and
the rigid rotation $q_{\rm rot}=re_\theta$.  We project these two
directions out and work on the corresponding velocity space
$H_\diamond$ and form domain $V_\diamond$.  Korn's inequality then makes
the viscous form coercive on $V_\diamond$, so that the reduced displacement
equation falls within the framework of the abstract theorem.  Its
stiffness form is
\[
 k[v,w]
 =\bigl\langle L_J\cC_\diamond v,\cC_\diamond w\bigr\rangle,
 \qquad
 \cC_\diamond v=-\nabla\cdot(\bar\rho v),
\]
where $L_J$ is the density Hessian, namely the operator associated with
the density quadratic form obtained from the second variation of the
augmented energy after minimizing over velocity perturbations subject
to fixed total angular momentum.  The
stiffness still has an infinite-dimensional kernel, containing in
particular weighted-divergence-free displacements, but this is allowed by
the abstract KTC theorem.

The projection of the rigid viscous modes does not impose an artificial
restriction on the unstable spectral problem.  Indeed, mass, axial
momentum, and total angular momentum are conserved, while the axial first
moment satisfies an evolution identity whose derivative is the axial
momentum.  The corresponding four bounded covectors therefore vanish on
every nonzero generalized eigenspace of the unrestricted generator.
They determine the rigid velocity components, and a domain-preserving
Routh transformation reduces the dynamics to the coercive velocity
space.  For the density part, the range of $\cC_\diamond$ is dense in
the mass-zero, axially centered space $X_{00}$.  Since the negative
subspaces are finite-dimensional, approximation of the corresponding
Gram matrices gives
\[
 n^-(\cC_\diamond^*L_J\cC_\diamond)
 =n^-(L_J|_{X_{00}})
 =n^-(Q_J|_{X_0}).
\]
The last equality restores axial translation, which is an exact zero
direction of the density Hessian.  Analytic Gaussian elimination
identifies the partial multiplicities of the reduced first-order system
with those of its quadratic pencil, while a finite-dimensional resolvent
decomposition and the Routh similarity transfer the resulting count back
to the full, unrestricted axisymmetric NSP generator.  This proves the
instability-index formula of Theorem~\ref{thm:main-index}.

The turning-point theorem requires a further analysis of the slowly
rotating equilibrium branch.  We construct this branch in the enthalpy
variable on a fixed ambient domain, rather than differentiating the
moving vacuum boundary directly.  Although second parameter derivatives
of the zero-extended density may be singular at the free surface when
$\gamma_0>3/2$, these singularities remain integrable throughout
$4/3<\gamma_0<2$.  This gives the regularity needed to pass from fixed
angular velocity to fixed total angular momentum $J=\kappa$ and to
differentiate the equilibrium equation along the fixed-$J$ branch.
The branch tangent satisfies
\[
 L_J\partial_\mu\rho_{\mu,\kappa}
 =-(\partial_\mu c_{\mu,\kappa})\,1,
 \qquad
 Q_J[\partial_\mu\rho_{\mu,\kappa}]
 =-(\partial_\mu c_{\mu,\kappa})
   \partial_\mu M_\kappa.
\]
Before the first spherical radial degeneracy, continuation of the density
spectrum shows that $L_J$ has exactly one negative direction, with axial
translation as its only kernel direction, and
$\partial_\mu c_{\mu,\kappa}>0$ for sufficiently small $\kappa$.
A one-constraint inertia calculation therefore reduces the constrained
Morse index to the sign of $\partial_\mu M_\kappa$.  Hence the viscous
unstable index changes precisely when the fixed-$J$ mass curve passes
through the continuation of the first nondegenerate spherical mass
maximum.  At the turning point the branch tangent enters the constrained
kernel in addition to the persistent axial-translation mode.

Finally, we compare the viscous constrained Hessian with the reduced
Euler--Poisson form.  For rigid rotation,
\[
 \widetilde Q_{\rm EP}=Q_J+\Delta Q,\qquad \Delta Q\ge0,
\]
where $\Delta Q$ is the additional quadratic term arising from
preservation of the material distribution of specific angular momentum
in the inviscid dynamics.  Weighted Cauchy--Schwarz gives an explicit
characterization of its equality case: in terms of the cylindrical
column mass,
\[
 m_\sigma(r)=C\,r\,m'_{\bar\rho}(r).
\]
If equality held on the turning direction for a sequence of slowly
rotating stars with $\kappa\to0$, passage to the spherical limit and
Abel inversion would identify the limiting branch tangent with the
mass-preserving homology tangent
\[
 3\rho+x\cdot\nabla\rho.
\]
The differentiated equilibrium equation would then force
$3P'(\rho)-h(\rho)$ to be constant on the density range of the star and
therefore
\[
 P(\rho)=K\rho^{4/3}.
\]
This contradicts the assumed low-density exponent $\gamma_0>4/3$.
Consequently $\Delta Q$ is strictly positive in the turning direction
for every sufficiently small nonzero $\kappa$.  Compactness and
operator-norm continuity of the Rayleigh contribution then give a
positive inviscid gap modulo axial translation on an interval beyond
the viscous turning point.  Thus the viscous star already has one
unstable axisymmetric mode there, whereas the corresponding
Euler--Poisson star remains axisymmetrically spectrally stable, proving
Theorem~\ref{thm:strict-separation}.

Several questions remain open.  These include the asymptotic behavior of
the unstable eigenvalues as the viscosity tends to zero, nonlinear
secular evolution beyond the turning point, the center dynamics at the
critical parameter, and nonaxisymmetric perturbations.  Rapidly rotating
configurations \cite{StraussWu2019} and relativistic viscous bar modes
\cite{SaijoGourgoulhon2006} provide further settings for comparison.
Radiation-driven instabilities
\cite{Chandrasekhar1970,FriedmanSchutz1975,Andersson1998,
FriedmanMorsink1998}
also remove energy but, unlike viscosity in the present model, change
the total angular momentum.  A relativistic extension would require a
causal dissipative model
\cite{IsraelStewart1979,HiscockLindblom1983,Kovtun2019,
BemficaDisconziNoronha2022}
and a corresponding spectral analysis of its additional variables.

Section~\ref{sec:KTC} proves the abstract KTC theorem.
Sections~\ref{sec:linearization} and \ref{sec:index-proof} derive the
linearized NSP system and prove Theorem~\ref{thm:main-index}.
Section~\ref{sec:equilibria} constructs the fixed-$J$ branch and analyzes
its constrained Hessian.  Sections~\ref{sec:turning-proof} and
\ref{sec:comparison} prove Theorems~\ref{thm:turning-point} and
\ref{thm:strict-separation}, respectively.  The appendices contain the
rotating-frame Lagrangian derivation, the moving-vacuum regularity
estimates, the operator-pencil facts used in the KTC argument, and the
remaining functional-analytic estimates.

\section{An abstract Kelvin--Tait--Chetaev theorem}
\label{sec:KTC}

We consider second-order equations with symmetric damping and a gyroscopic
term.  Under a coercive damping estimate and finite negative stiffness
index, we determine the dimension of the unstable spectral subspace of
the first-order generator by counting the characteristic values of its
quadratic pencil.

\subsection{The second-order equation and its generator}

Let $V$ and $H$ be complex Hilbert spaces and let
$\jmath:V\hookrightarrow H$ be a dense continuous embedding.  Inner products
are taken to be linear in the first argument.  We use the continuous
antidual $V^*$, so $\langle f,v\rangle$ is linear in $f$ and conjugate-linear
in $v$.  Let $d,k,g$ be bounded sesquilinear forms on $V$ such that
\begin{align}
 d[u,v]&=\overline{d[v,u]},
 &d[u,u]&\ge d_0\|u\|_V^2,                                      \label{eq:KTC-d}\\
 k[u,v]&=\overline{k[v,u]},
 &|k[u,v]|&\le C_k\|u\|_V\|v\|_V,                              \label{eq:KTC-k}\\
 g[u,v]&=-\overline{g[v,u]},
 &|g[u,v]|&\le C_g\|u\|_V\|v\|_V.                             \label{eq:KTC-g}
\end{align}
The stellar application satisfies the stronger bound
\begin{equation}
 |g[u,v]|\le C_g\|u\|_H\|v\|_H.                                \label{eq:KTC-g-H}
\end{equation}
Only \eqref{eq:KTC-g} is needed for the abstract theorem.  We impose
no compactness assumption on the embedding or the gyroscopic form.
We use the equivalent norm
$\|u\|_V^2=d[u,u]$ from now on.

Let $\iota:H\to V^*$ be the canonical injection,
\[
 \langle\iota h,w\rangle=(h,\jmath w)_H,
\]
and set $J=\iota\jmath$.  Denote by
$\mathsf D,\mathsf K,\mathsf G\in\mathcal L(V,V^*)$ the form operators
defined by
\[
 \langle\mathsf Du,w\rangle=d[u,w],\qquad
 \langle\mathsf Ku,w\rangle=k[u,w],\qquad
 \langle\mathsf Gu,w\rangle=g[u,w].
\]
The weak second-order equation is
\begin{equation}
Ju_{tt}+\mathsf Du_t+\mathsf Gu_t+\mathsf Ku=0.                  \label{eq:abstract-second-order}
\end{equation}
Time derivatives here may be interpreted distributionally.  When
$u_{tt}$ is $H$-valued, its image in $V^*$ is understood through
$\iota:H\to V^*$; $J=\iota\jmath$ is the restriction to $V$.
On $\mathscr X=V\times H$, equipped with
\[
 \|(u,v)\|_{\mathscr X}^2=\|u\|_V^2+\|v\|_H^2,
\]
define
\begin{align}
 \operatorname{Dom}(\mathscr A)
  &=\left\{(u,v)\in V\times V:
       \mathsf Ku+\mathsf Dv+\mathsf Gv\in\iota(H)\right\},
                                                                    \label{eq:KTC-generator-domain}\\
 \mathscr A(u,v)
  &=\left(v,-\iota^{-1}
       (\mathsf Ku+\mathsf Dv+\mathsf Gv)\right).                  \label{eq:KTC-generator}
\end{align}

\begin{lemma}
\label{lem:KTC-wellposed}
Under \eqref{eq:KTC-d}--\eqref{eq:KTC-g}, the operator
$\mathscr A$ generates a strongly continuous semigroup on $\mathscr X$.
Consequently, \eqref{eq:abstract-second-order} is well posed on
$V\times H$ in the semigroup sense.
\end{lemma}

\begin{proof}
For $(u,v)\in\operatorname{Dom}(\mathscr A)$, the definition of the
domain and the form identities give
\begin{align*}
 \Re(\mathscr A(u,v),(u,v))_{\mathscr X}
 &=\Re\{d[v,u]-k[u,v]\}-d[v,v]\\
 &\le C\|u\|_V^2-\frac12\|v\|_V^2
 \le C\|(u,v)\|_{\mathscr X}^2.
\end{align*}
Here $\Re g[v,v]=0$ was used.  Thus $\mathscr A-CI$ is dissipative for
a sufficiently large $C$.

It remains to prove the range condition.  Fix a sufficiently large real
$\lambda$ and let $(f,h)\in V\times H$.  The equation
$(\lambda-\mathscr A)(u,v)=(f,h)$ gives $v=\lambda u-f$.  Substitution in
the second component leads to the variational equation
\begin{equation}
 \ell_\lambda[u,w]
 =(h,w)_H+\lambda(f,w)_H+d[f,w]+g[f,w],\qquad w\in V,              \label{eq:KTC-resolvent-form}
\end{equation}
where
\[
 \ell_\lambda[u,w]
 =\lambda^2(u,w)_H+\lambda d[u,w]+\lambda g[u,w]+k[u,w].
\]
Since $\lambda$ is real and $g$ is skew-Hermitian,
\[
 \Re\ell_\lambda[u,u]
 \ge (\lambda-C_k)\|u\|_V^2+\lambda^2\|u\|_H^2.
\]
For $\lambda>C_k$, Lax--Milgram yields a unique $u\in V$ satisfying
\eqref{eq:KTC-resolvent-form}.  With $v=\lambda u-f$, the same
identity shows
\[
 \mathsf Ku+\mathsf Dv+\mathsf Gv=\iota(h-\lambda v)\in\iota(H).
\]
Choose $\lambda$ larger than both $C_k$ and the quasi-dissipativity
constant $C$.  Then $(u,v)\in\operatorname{Dom}(\mathscr A)$ and
$\operatorname{Ran}(\lambda-\mathscr A)=\mathscr X$.  The same argument
also gives density of the domain.  Indeed, if
$y\perp\operatorname{Dom}(\mathscr A)$ and
$x=(\lambda-\mathscr A)^{-1}y$, then
$0=\Re(y,x)\ge(\lambda-C)\|x\|_{\mathscr X}^2$, so $x=0$ and $y=0$.
Thus $\mathscr A-CI$ is maximally dissipative: indeed, the preceding range
condition is exactly
$\operatorname{Ran}((\lambda-C)I-(\mathscr A-CI))=\mathscr X$.
The Lumer--Phillips theorem, applied to this scalar shift, proves the claim.
\end{proof}

Let $n^-(k)$ denote the maximal dimension of a subspace of $V$ on which
$k$ is negative definite, and let
$\ker k=\{u\in V:k[u,w]=0\text{ for every }w\in V\}$ be its form kernel.
For an isolated eigenvalue $\lambda$ of
$\mathscr A$, let $\operatorname{algmult}(\lambda;\mathscr A)$ denote the
rank of its Riesz projection.  We write
\[
 N_+(\mathscr A)
 =\sum_{\lambda\in\sigma(\mathscr A),\ \Re\lambda>0}
   \operatorname{algmult}(\lambda;\mathscr A),
\]
whenever the sum is finite.

\begin{theorem}
\label{thm:abstract-KTC}
Assume \eqref{eq:KTC-d}--\eqref{eq:KTC-g}, the dense continuous embedding
$V\hookrightarrow H$, and
\begin{equation}
 n^-(k)<\infty.                                                   \label{eq:KTC-finite-index}
\end{equation}
Then the spectrum of $\mathscr A$ in $\mathbb C_+$ consists of finitely
many isolated eigenvalues of finite algebraic multiplicity, and
\begin{equation}
 N_+(\mathscr A)=n^-(k).                                         \label{eq:KTC-generator-count}
\end{equation}
\end{theorem}
\begin{remark}
No Fredholm, closed-range, or finite-dimensional-kernel assumption is
imposed on the stiffness at zero.  In particular, $\ker k$ may be
infinite-dimensional and the nonnegative spectrum of the preconditioned
stiffness may accumulate at zero.  Neither the embedding $V\hookrightarrow H$
nor the preconditioned mass or gyroscopic operator is required to be compact.
\end{remark}

We prove the theorem after relating the generator to its quadratic pencil.

\subsection{The generator and the quadratic pencil}

Let $R_d:V\to V^*$ be the Riesz isomorphism induced by $d$ and define
\begin{equation}
 A=R_d^{-1}\mathsf K,\qquad
 M=R_d^{-1}J,\qquad
 \Gamma=R_d^{-1}\mathsf G.                                      \label{eq:KTC-preconditioned}
\end{equation}
Then $A=A^*$, $M=M^*\ge0$, and $\Gamma^*=-\Gamma$ on $(V,d)$.  The
operators $M$ and $\Gamma$ are bounded, and
$(Mu,u)_V=\|u\|_H^2>0$ for $u\ne0$.
Moreover,
\[
 n^-(A)=n^-(k)<\infty.
\]
The preconditioned pencil is
\begin{equation}
 \mathcal P(\lambda)
 =A+\lambda(I+\Gamma)+\lambda^2M,                              \label{eq:KTC-pencil}
\end{equation}
and the corresponding form pencil is
\begin{equation}
 \mathcal L(\lambda)
 =\mathsf K+\lambda(\mathsf D+\mathsf G)+\lambda^2J
 =R_d\mathcal P(\lambda).                                      \label{eq:KTC-form-pencil}
\end{equation}

\begin{lemma}
\label{lem:KTC-generator-pencil}
For $\lambda_0\in\mathbb C_+$,
$\lambda_0\in\sigma(\mathscr A)$ if and only if $\lambda_0$ is a
characteristic value of $\mathcal P$.  Such points are isolated, and
\begin{equation}
 \operatorname{algmult}(\lambda_0;\mathscr A)
 =\operatorname{algmult}(\lambda_0;\mathcal P).                  \label{eq:KTC-multiplicity-equivalence}
\end{equation}
Thus the Gohberg--Sigal multiplicity on the right equals the ordinary
Riesz algebraic multiplicity of the generator.
More precisely, for every $\lambda\in\mathbb C_+$,
\begin{equation}
 \lambda\in\rho(\mathscr A)
 \quad\Longleftrightarrow\quad
 \mathcal L(\lambda):V\to V^*\ \hbox{is invertible}
 \quad\Longleftrightarrow\quad
 \mathcal P(\lambda):V\to V\ \hbox{is invertible}.               \label{eq:KTC-invertibility-equivalence}
\end{equation}
\end{lemma}

\begin{proof}
Let
\[
 A_+=A\mathbf1_{[0,\infty)}(A),\qquad
 A_-=A\mathbf1_{(-\infty,0)}(A).
\]
By the spectral theorem, the negative spectral space
$V_-:=\mathbf1_{(-\infty,0)}(A)V$ has dimension
$\dim V_-=n^-(A)<\infty$: the quadratic form is negative definite on
$V_-$ and nonnegative on its orthogonal complement.  Hence $A_-$ has
finite rank, $A_+\ge0$, and $A|_{V_-}$ is invertible if $V_-\ne\{0\}$.
For
$a=\Re\lambda\ge0$, $\lambda\ne0$, set
\[
 B(\lambda)=A_++\lambda(I+\Gamma)+\lambda^2M.
\]
Skew-adjointness of $\Gamma$ and nonnegativity of $M$ give
\[
 \Re(\bar\lambda B(\lambda)u,u)_V
 =a(A_+u,u)_V+|\lambda|^2\|u\|_V^2
    +a|\lambda|^2(Mu,u)_V
 \ge |\lambda|^2\|u\|_V^2.
\]
The bounded, possibly non-Hermitian form
$(u,v)\mapsto\bar\lambda(B(\lambda)u,v)_V$ therefore satisfies the
coercivity hypothesis of the complex Lax--Milgram theorem.  For each
$f\in V$, that theorem gives a unique $u\in V$ with
$\bar\lambda(B(\lambda)u,v)_V=\bar\lambda(f,v)_V$ for every $v\in V$.
Thus $B(\lambda)$ is both injective and surjective, and the estimate gives
$\|B(\lambda)^{-1}\|\le|\lambda|^{-1}$.  Consequently,
$\mathcal P(\lambda)=B(\lambda)+A_-$ is a finite-rank perturbation of an
invertible operator, hence Fredholm of index zero throughout
$\{\Re\lambda\ge0,\lambda\ne0\}$.  Finite rank of the negative
stiffness, rather than compactness of $M$ or $\Gamma$, supplies the
Fredholm property.  The pencil is
invertible for every sufficiently large positive real $\lambda$, because
\[
 \Re(\mathcal P(\lambda)u,u)_V
 \ge(\lambda-\|A\|)\|u\|_V^2.
\]
The analytic Fredholm theorem therefore implies that the characteristic
values in $\mathbb C_+$ are isolated and have finite Gohberg--Sigal
multiplicity.

If $\mathcal L(\lambda)$ is invertible and $(f,h)\in V\times H$, then
\begin{equation}
 (\lambda-\mathscr A)^{-1}(f,h)
 =\left(u,\lambda u-f\right),
 \quad
 u=\mathcal L(\lambda)^{-1}
 \bigl[\iota h+\lambda Jf+\mathsf Df+\mathsf Gf\bigr].           \label{eq:KTC-resolvent-formula}
\end{equation}
Conversely, if $\mathcal L(\lambda)u=0$ and $v=\lambda u$, then
$(u,v)\in\operatorname{Dom}(\mathscr A)$ and
$(\lambda-\mathscr A)(u,v)=0$.  Thus invertibility of
$\lambda-\mathscr A$ implies injectivity, and hence invertibility, of the
index-zero operator $\mathcal L(\lambda)$.  This proves the spectral
equivalence.  Formula \eqref{eq:KTC-resolvent-formula} also shows that
the generator resolvent is meromorphic with finite-rank principal part at
each characteristic value.

The Jordan structures can be identified directly.  Put
$u_{-1}=u_{-2}=0$.  Since
\[
 \mathcal L'(\lambda)=\mathsf D+\mathsf G+2\lambda J,
 \qquad \mathcal L''(\lambda)=2J,
\]
and all higher derivatives vanish, a root chain
$u_0,\ldots,u_{m-1}$ of $\mathcal L$ at $\lambda_0$ satisfies
\begin{equation}
 \mathcal L(\lambda_0)u_\ell
 +\mathcal L'(\lambda_0)u_{\ell-1}
 +Ju_{\ell-2}=0,\qquad 0\le\ell<m.                             \label{eq:KTC-root-chain-expanded}
\end{equation}
Define
\[
 v_\ell=\lambda_0u_\ell+u_{\ell-1}.
\]
Equation \eqref{eq:KTC-root-chain-expanded} implies
\[
 \mathsf Ku_\ell+(\mathsf D+\mathsf G)v_\ell
 =-\lambda_0^2Ju_\ell-2\lambda_0Ju_{\ell-1}-Ju_{\ell-2}
 \in\iota(H),
\]
so $(u_\ell,v_\ell)\in\operatorname{Dom}(\mathscr A)$.  A direct
calculation gives
\[
 (\mathscr A-\lambda_0)(u_\ell,v_\ell)
 =(u_{\ell-1},v_{\ell-1}).
\]
The converse follows from the first component of this identity.  Hence
root chains of the pencil and Jordan chains of the generator are in
one-to-one correspondence, with the same lengths.  The correspondence
commutes with truncation and preserves extendability at each level;
its map on leading vectors, $u_0\mapsto(u_0,\lambda_0u_0)$, is an
isomorphism.  It therefore preserves the complete list of partial
multiplicities, not just the maximal chain length.  Since multiplication
by the fixed isomorphism $R_d$ does not alter root chains, this proves
\eqref{eq:KTC-multiplicity-equivalence}.
\end{proof}

\subsection{The pencil index}

\begin{theorem}
\label{thm:KTC-pencil}
Under the hypotheses of Theorem~\ref{thm:abstract-KTC}, the
characteristic values of \eqref{eq:KTC-pencil} in $\mathbb C_+$ form
a finite set and
\begin{equation}
 \sum_{\lambda\in\mathbb C_+}
 \operatorname{algmult}(\lambda;\mathcal P)=n^-(A).              \label{eq:KTC-pencil-count}
\end{equation}
\end{theorem}
We prove this formula by a homotopy that keeps the damping term fixed.
\begin{lemma}
\label{lem:KTC-boundary}
For $(r,s)\in[0,1]^2$, set
\begin{equation}
 \mathcal P_{r,s}(\lambda)
 =A+\lambda(I+s\Gamma)+r\lambda^2M.                              \label{eq:KTC-homotopy}
\end{equation}
The following statements hold uniformly in $(r,s)$:
\begin{enumerate}
\item every characteristic value in $\mathbb C_+$ satisfies
      $|\lambda|\le\|A\|$;
\item there is no nonzero characteristic value on $i\mathbb R$;
\item there exists $\varepsilon>0$ such that there is no characteristic
      value with $\Re\lambda\ge0$ and
      $0<|\lambda|<\varepsilon$.
\end{enumerate}
\end{lemma}

\begin{proof}
Let $\mathcal P_{r,s}(\lambda)u=0$, normalize $\|u\|_V=1$, and write
$\lambda=a+ib$.  Set
\[
 (\Gamma u,u)_V=i\gamma,
 \qquad m=(Mu,u)_V.
\]
Multiplication of the scalar equation by $\bar\lambda$ and passage to
real parts gives
\begin{equation}
 a(Au,u)_V+|\lambda|^2+ra|\lambda|^2m=0.                         \label{eq:KTC-energy}
\end{equation}
If $a>0$, this yields $|\lambda|\le\|A\|$.  If
$\lambda=ib\ne0$, the imaginary part of
$(\mathcal P_{r,s}(ib)u,u)_V=0$ is $b\|u\|_V^2=0$, a contradiction.

It remains to prove the uniform exclusion near zero.  Let
$V_-=\mathbf1_{(-\infty,0)}(A)V$ and let $P_-$ be the orthogonal
projection onto $V_-$.  If $V_-=\{0\}$, the energy identity already
excludes every characteristic value in $\mathbb C_+$, so suppose
$V_-\ne\{0\}$.  As above, $A|_{V_-}$ is invertible on this
finite-dimensional space.  Suppose that $a>0$ and $|\lambda|$ is small.
Writing $u_-=P_-u$, projection of the root equation gives
\[
 (A|_{V_-}+\lambda I)u_-
 =-s\lambda P_-\Gamma u-r\lambda^2P_-Mu.
\]
The inverse of $A|_{V_-}+\lambda I$ is uniformly bounded for small
$|\lambda|$, and hence
\begin{equation}
 \|u_-\|_V\le C|\lambda|.                                      \label{eq:KTC-negative-small}
\end{equation}
Since $A\ge0$ on $V_-^\perp$,
\begin{equation}
 (Au,u)_V\ge-\|A\|\|u_-\|_V^2\ge-C|\lambda|^2.                 \label{eq:KTC-A-lower}
\end{equation}
Substitution into \eqref{eq:KTC-energy} gives
\[
 0=a(Au,u)_V+|\lambda|^2+ra|\lambda|^2m
   \ge |\lambda|^2(1-Ca)>0
\]
for sufficiently small $|\lambda|$, a contradiction.  The case $a=0$
was already excluded.  All constants above are uniform for
$(r,s)\in[0,1]^2$.  Only $A|_{V_-}$ was inverted: no spectral gap,
closed range, or Fredholm property of $A$ on $V_-^\perp$ at zero is used.
\end{proof}

\begin{proof}[Proof of Theorem~\ref{thm:KTC-pencil}]
The preceding Fredholm argument applies uniformly to
\[
 B_{r,s}(\lambda)=A_++\lambda(I+s\Gamma)+r\lambda^2M,
 \qquad \mathcal P_{r,s}(\lambda)=B_{r,s}(\lambda)+A_-.
\]
Indeed, for $a=\Re\lambda\ge0$ and $\lambda\ne0$,
\[
 \Re(\bar\lambda B_{r,s}(\lambda)u,u)_V
 =a(A_+u,u)_V+|\lambda|^2\|u\|_V^2
   +ra|\lambda|^2(Mu,u)_V
 \ge |\lambda|^2\|u\|_V^2.
\]
Thus $B_{r,s}(\lambda)$ is invertible and
$\mathcal P_{r,s}(\lambda)$ is Fredholm of index zero there.
Choose the
$\varepsilon$ from Lemma~\ref{lem:KTC-boundary}, fix
$0<\rho<\varepsilon$, take $R>\|A\|$, and set
\[
 \mathscr O_{\rho,R}
 =\{\lambda\in\mathbb C:\Re\lambda>0,
      \ \rho<|\lambda|<R\}.
\]
On the closure of this half-annulus,
$\|B_{r,s}(\lambda)^{-1}\|\le1/\rho$, uniformly in $(r,s)$.
On a fixed compact neighborhood $K$ of the closed half-annulus, the
polynomial coefficients give
\[
 \|B_{r,s}(\lambda')-B_{r,s}(\lambda)\|
 \le C_K|\lambda'-\lambda|
\]
uniformly in $(r,s)$.  For $\lambda$ in the closed half-annulus and
$|\lambda'-\lambda|<\rho/(2C_K)$, shrinking this radius to remain in
$K$, we have
\[
 \bigl\|B_{r,s}(\lambda)^{-1}
       [B_{r,s}(\lambda')-B_{r,s}(\lambda)]\bigr\|<\tfrac12.
\]
A Neumann series makes $B_{r,s}(\lambda')$ invertible.  This gives one
common open neighborhood of the closed half-annulus for all $(r,s)$.
Consequently every $\mathcal P_{r,s}$ is analytic Fredholm of index zero
on that neighborhood.
The same lemma, together with the index-zero property, shows that every
$\mathcal P_{r,s}$ is invertible on the piecewise smooth boundary of this
fixed half-annulus.  The Gohberg--Sigal homotopy theorem therefore makes
the total algebraic multiplicity inside the contour independent of
$(r,s)$.  Since the parameter square is connected,
\[
 \sum_{\lambda\in\mathbb C_+}
 \operatorname{algmult}(\lambda;\mathcal P_{1,1})
 =
 \sum_{\lambda\in\mathbb C_+}
 \operatorname{algmult}(\lambda;\mathcal P_{0,0}).
\]
Now $\mathcal P_{0,0}(\lambda)=A+\lambda I$.  Its right-half-plane
characteristic values are $\lambda=-a_j>0$, where $a_j<0$ runs through
the negative eigenvalues of $A$, and their multiplicities are unchanged.
The sum is therefore $n^-(A)$.  Crucially, the contour avoids the origin;
no Fredholm or closed-range property of the stiffness at zero enters
the argument.
\end{proof}

\begin{proof}[Proof of Theorem~\ref{thm:abstract-KTC}]
Theorem~\ref{thm:KTC-pencil} and
Lemma~\ref{lem:KTC-generator-pencil} give
\[
 N_+(\mathscr A)
 =\sum_{\lambda\in\mathbb C_+}
   \operatorname{algmult}(\lambda;\mathcal P)
 =n^-(A)=n^-(k).
\]
The same results show that the unstable generator spectrum is finite and
consists only of isolated eigenvalues of finite algebraic multiplicity.
\end{proof}

Classical KTC and dissipation-induced instability results are discussed in
\cite{BlochEtAl1994,MaddocksOverton1995,KrechetnikovMarsden2007}, with
infinite-dimensional developments in
\cite{KrechetnikovMarsden2009,Pivovarchik1992,Shkalikov1996}.
Pivovarchik \cite[Theorem~3.3]{Pivovarchik1992} proved a
compactness-free algebraic instability count for quadratic operator pencils
$\lambda^2I+\lambda(D+iG)+L$, allowing zero in the essential stiffness
spectrum.  His theorem is formulated with a self-adjoint stiffness operator,
uniformly positive $H$-space damping, operator-domain compatibility, and
stiffness-controlled bounds on the $H$-space damping and gyroscopic operators.
By contrast, Theorem~\ref{thm:abstract-KTC} is formulated on a common form
domain: the damping form defines $V$, while the stiffness and gyroscopic forms
need only be bounded on $V\times V$, with no estimate of the associated
$H$-space damping operator in terms of the stiffness.  Thus that damping
operator may be unbounded on an infinite-dimensional stiffness kernel, and the
preconditioned mass operator need not be the identity or boundedly invertible.
These distinctions are essential for NSP, where the stiffness vanishes on the
infinite-dimensional space $\nabla\cdot(\bar\rho u)=0$ while dissipation
controls the $H^1$ norm after the rigid modes are removed.  The present form
realization also constructs the first-order generator and identifies its Riesz
algebraic multiplicities with those of the pencil.  Shkalikov's operator
formulation \cite{Shkalikov1996}, by comparison, assumes bounded invertibility
of both mass and stiffness.

\begin{remark}
\label{rem:KTC-comparison}
Theorem~2.5 of \cite{ChengLinWang2026} treats
$u_{tt}+Du_t+Lu=0$ at the operator level.  Its assumptions include uniform
positivity of $D$, self-adjointness and finite negative index of $L$,
$\ker L=\{0\}$, a positive gap on the complementary stiffness space,
compatibility of the domains of $D$ and $L$, and compactness of a shifted
$L$-energy sublevel.  The condition $\ker L=\{0\}$ is used in the damping
homotopy to prevent characteristic values from reaching the origin.

On a common form realization, Theorem~\ref{thm:abstract-KTC} extends
this result by admitting a gyroscopic term and removing stiffness
nondegeneracy, the positive gap, and the compactness assumptions.
The proofs also differ: we keep the coercive damping fixed and deform the
mass and gyroscopic terms to zero, using the finite-dimensional negative
stiffness space to exclude small unstable characteristic values.
The operator-domain assumptions in \cite{ChengLinWang2026} and our form
hypotheses are arranged differently, so this comparison concerns their
common variational class.  The radial NSP system belongs to that class,
and the following corollary completes its instability count at mass
extrema.
\end{remark}

Let $\rho_\mu$, $\mu\in\mathcal I$, be any parameter interval of the
nonrotating family considered in \cite{ChengLinWang2026,LinZeng2022}.
For this corollary, the pressure assumptions are those of the cited
nonrotating theory: $P\in C^1((0,\infty))$, $P(0)=0$, $P'>0$, and
$s^{1-\gamma_1}P'(s)\to K_1>0$ as $s\downarrow0$, for some
$\gamma_1\in(6/5,2)$.  The interval $\mathcal I$ is contained in the
parameter range on which the spherical stars have finite support.
Let $R_\mu$ be the support radius and
$M(\mu):=\int_{\mathbb{R}^3}\rho_\mu\,dx$.
Denote by
$\cA^{\rm rad}_{\rm NSP}(\mu)$ and $\cA^{\rm rad}_{\rm EP}(\mu)$ the
radial linearized generators, and by $L^{\rm disp}_\mu$ the radial
displacement stiffness denoted by $L_\mu$ in
\cite[(3.10)]{ChengLinWang2026}.

\begin{corollary}
\label{cor:nonrotating-all-mu}
For every $\mu\in\mathcal I$, including
$\partial_\mu M(\mu)=0$,
\begin{equation}
 N_+\!\left(\cA^{\rm rad}_{\rm NSP}(\mu)\right)
 =n^-\!\left(L^{\rm disp}_\mu\right)
 =N_+\!\left(\cA^{\rm rad}_{\rm EP}(\mu)\right)
 =n^-(D^0_\mu)-\mathfrak i_\mu,                           \label{eq:nonrotating-all-mu}
\end{equation}
where $D^0_\mu$ is the radial Schr\"odinger operator in the
density--potential reduction, defined in
\cite[(3.27)]{LinZeng2022} and
\begin{equation}
 \mathfrak i_\mu=
 \begin{cases}
  1,& \partial_\mu M(\mu)\,
       \partial_\mu\!\left(M(\mu)/R_\mu\right)>0
       \quad\hbox{or}\quad \partial_\mu M(\mu)=0,\\
  0,& \partial_\mu M(\mu)\,
       \partial_\mu\!\left(M(\mu)/R_\mu\right)<0
       \quad\hbox{or}\quad
       \partial_\mu\!\left(M(\mu)/R_\mu\right)=0.
 \end{cases}                                                \label{eq:nonrotating-i-mu}
\end{equation}
The two alternatives are exhaustive because the two derivatives in
\eqref{eq:nonrotating-i-mu} cannot vanish simultaneously.  In particular,
at a mass extremum,
\begin{equation}
 N_+\!\left(\cA^{\rm rad}_{\rm NSP}(\mu)\right)
 =n^-(D^0_\mu)-1.                                          \label{eq:nonrotating-mass-extremum}
\end{equation}
\end{corollary}

\begin{proof}
Use $H_\mu=L^2((0,R_\mu),\rho_\mu r^2\,dr)$ and the radial space
$V_\mu=\{u:u(r)e_r\in H^1(B_{R_\mu};\C^3)\}$, with the damping form
$d_\mu$ in \cite[(3.13)]{ChengLinWang2026}.
The estimate \cite[(3.14)]{ChengLinWang2026} makes $d_\mu$ coercive,
and $V_\mu\hookrightarrow H_\mu$ is dense and continuous.  The proofs of
\cite[Lemmas~3.4--3.5]{ChengLinWang2026} give boundedness of the
stiffness form on $V_\mu$ and finiteness of its negative index.  These
estimates do not use $\partial_\mu M\ne0$; that condition is used
separately in \cite{ChengLinWang2026} to exclude the stiffness kernel.

We first check that restricting the radial displacement form to the damping
form domain $V_\mu$ does not change its Morse index.  Indeed, put
$\cC_\mu u=-r^{-2}\partial_r(r^2\rho_\mu u)$.  Its range on
$V_\mu$ is dense in the mass-zero radial density space: the annihilator
calculation in \cite[(3.57)]{LinZeng2022} uses only compactly supported
smooth velocities and gives the mass representer $1/h'(\rho_\mu)$.
Both the viscous and inviscid displacement forms are obtained by
composing the same continuous density Hessian with $\cC_\mu$.
Approximating a basis of each finite-dimensional negative density
subspace by elements of $\cC_\mu V_\mu$ therefore gives
$n^-(k_\mu|_{V_\mu})=n^-(L^{\rm disp}_\mu)$.
Theorem~\ref{thm:abstract-KTC}, with $g=0$, yields
\begin{equation}
 N_+\!\left(\cA^{\rm rad}_{\rm NSP}(\mu)\right)
 =n^-\!\left(L^{\rm disp}_\mu\right).
                                                        \label{eq:nonrotating-KTC-step}
\end{equation}
Here the radial reduction in \cite[Section~3]{ChengLinWang2026} eliminates
the density as $\sigma=\lambda^{-1}\cC_\mu u$ for $\lambda\ne0$.
This elimination is analytic near each $\lambda_0\in\C_+$ and preserves
all partial multiplicities, by the same Gaussian elimination and
root-chain calculation as in Lemma~\ref{lem:KTC-generator-pencil}.
Thus \eqref{eq:nonrotating-KTC-step} counts Riesz algebraic multiplicities
of the radial density--velocity generator.

Finally, \cite[Theorem~1.2(ii) and Section~3.5]{LinZeng2022} gives
\[
 n^-(L^{\rm disp}_\mu)
 =N_+\!\left(\cA^{\rm rad}_{\rm EP}(\mu)\right)
 =n^-(D^0_\mu)-\mathfrak i_\mu
\]
for every $\mu$.  This proves \eqref{eq:nonrotating-all-mu}, including
\eqref{eq:nonrotating-mass-extremum}, since $\mathfrak i_\mu=1$ when
$\partial_\mu M=0$.  Possible zero eigenvalues are outside the contour
used in Theorem~\ref{thm:KTC-pencil}.
\end{proof}

Corollary~\ref{cor:nonrotating-all-mu} completes the linear spectral count
at the extremal parameter.  It does not assert asymptotic or nonlinear
stability there: the differentiated equilibrium contributes to the
zero-frequency center structure and may generate a Jordan chain.

\section{Linearized NSP and symmetry reduction}
\label{sec:linearization}

Throughout Sections~\ref{sec:linearization} and \ref{sec:index-proof}, we
fix a regular rigidly rotating equilibrium whose pressure law satisfies
Hypothesis~\ref{hyp:EOS}.

\subsection{Rotating-frame linearization}

We work in the frame rotating with angular velocity $\omega e_z$.  If $u$
is the perturbation of the relative velocity and $\sigma$ the Eulerian
density perturbation, the linearized equations are
\begin{align}
 \sigma_t&=-\nabla\cdot(\bar\rho u),                                      \label{eq:linear-mass}\\
 u_t+2\omega e_z\times u+
 \nabla\bigl(h'(\bar\rho)\sigma+V_\sigma\bigr)
 &=\bar\rho^{-1}\nabla\cdot\mathbb T(u).                                  \label{eq:linear-momentum}
\end{align}
Appendix~\ref{app:lagrangian} derives these equations from the
free-boundary Lagrangian system. In rotating coordinates, the material
variation of the centrifugal force cancels with the differentiated
equilibrium force.

We call a scalar function $f$ axisymmetric if
$f(R_\vartheta x)=f(x)$ for every rotation $R_\vartheta$ about the
$z$-axis, and a vector field $u$ axisymmetric if
$u(R_\vartheta x)=R_\vartheta u(x)$.  The subscript ``${\rm ax}$'' denotes
the corresponding closed subspace.  In particular, $H^1_{\rm ax}$ is the
closure in $H^1$ of smooth Cartesian axisymmetric vector fields.  Let
\begin{equation}
 H=L^2_{\bar\rho,\mathrm{ax}}(\Omega_{\bar\rho};\C^3),\qquad
 (u,v)_H=\int_{\Omega_{\bar\rho}}\bar\rho\,u\cdot\overline v\,dx,          \label{eq:velocity-H}
\end{equation}
and
\begin{equation}
 X=X_{\bar\rho}:=L^2_{h'(\bar\rho),\mathrm{ax}}(\Omega_{\bar\rho};\C),
 \qquad \|\sigma\|_X^2=\int_{\Omega_{\bar\rho}}
             h'(\bar\rho)|\sigma|^2\,dx.                  \label{eq:density-space}
\end{equation}
We shall use the closed density subspaces
\begin{equation}
 X_0=\left\{\sigma\in X:\int_{\Omega_{\bar\rho}}\sigma\,dx=0\right\},\qquad
 X_{00}=\left\{\sigma\in X:\int_{\Omega_{\bar\rho}}\sigma\,dx=0,\quad
                         \int_{\Omega_{\bar\rho}} z\sigma\,dx=0\right\}.       \label{eq:X00}
\end{equation}
Define
\begin{equation}
 \cC u=-\nabla\cdot(\bar\rho u),\qquad
 L_0\sigma=h'(\bar\rho)\sigma+V_\sigma\in X^*.                            \label{eq:C-L0}
\end{equation}
We identify $H$ with a subspace of $(H^1_{\rm ax})^*$ by the embedding
$\langle\iota_H f,u\rangle=(f,u)_H$.
The form adjoint $\cC^*:X^*\to(H^1_{\rm ax})^*$ is defined by
\begin{equation}
 \ip{\cC^*\phi}{u}_{(H^1)^*,H^1}
 =\ip{\phi}{\cC u}_{X^*,X}.                           \label{eq:C-adjoint}
\end{equation}
For smooth $\phi$, this pairing equals
$\int\bar\rho\nabla\phi\cdot\overline u\,dx$.
Thus $\cC^*\phi$ is the distribution $\bar\rho\nabla\phi$;
its force representative under $\iota_H$ is $\nabla\phi$ whenever this
gradient belongs to $H$. In sums in $(H^1_{\rm ax})^*$, an $H$-valued
term is understood through $\iota_H$. The operator-domain conditions
below require only that the complete momentum residual be represented
by an element of $H$.  The sesquilinear form
associated with viscous dissipation is
\begin{equation}
 d[u,v]=\int_{\Omega_{\bar\rho}}
 \left(2\nu_s\dev e(u):\overline{\dev e(v)}
 +\nu_b(\nabla\cdot u)\overline{\nabla\cdot v}\right)dx.                  \label{eq:viscous-form}
\end{equation}
We abbreviate $d[u]:=d[u,u]$.
Let $\mathsf D:H^1_{\rm ax}\to(H^1_{\rm ax})^*$ be its form map,
$\langle\mathsf Du,v\rangle=d[u,v]$.  Its nonnegative self-adjoint
realization in $H$ is denoted by $D$:
\begin{equation}
 (Du,v)_H=d[u,v],\qquad
 Du=-\bar\rho^{-1}\nabla\cdot\mathbb T(u)                                \label{eq:D-realization}
\end{equation}
in the strong interior notation.  For a smooth nonlinear solution,
integration by parts and the zero-traction boundary condition give
\begin{equation}
 \frac{dE}{dt}
 =-\int_{\Omega(t)}\left(2\nu_s|\dev e(v)|^2
             +\nu_b|\nabla\cdot v|^2\right)dx\le0.          \label{eq:NSP-energy-dissipation}
\end{equation}
The corresponding quadratic dissipation for the linearized velocity is
$d[u,u]$.  This form is nonnegative on the full velocity space and becomes
coercive only after its rigid kernel is removed below.  Precisely,
\begin{equation}
 \Dom(D)=\{u\in H^1_{\rm ax}:\exists f\in H\text{ such that }
 d[u,v]=(f,v)_H\ \forall v\in H^1_{\rm ax}\},\qquad Du=f.   \label{eq:D-domain}
\end{equation}
The traction condition is the weak conormal condition encoded by this form;
an $H^{-1/2}$ or classical trace is asserted only when additional regularity
is available.  Finally
\begin{equation}
 G_0u=2\omega e_z\times u
 =-2\omega u_\theta e_r+2\omega u_r e_\theta,\qquad G_0^*=-G_0.           \label{eq:G0}
\end{equation}
Equations \eqref{eq:linear-mass}--\eqref{eq:linear-momentum} become
\begin{equation}
 \sigma_t=\cC u,\qquad
 u_t+Du+G_0u+\cC^*L_0\sigma=0.                                             \label{eq:first-order-full}
\end{equation}
The unrestricted energy and form spaces are
\begin{equation}
 \mathscr H_{\rm full}=X\times H,\qquad
 \mathscr V_{\rm full}=X\times H^1_{\rm ax}(\Omega_{\bar\rho};\C^3).
 \label{eq:full-phase-spaces}
\end{equation}
On $\mathscr V_{\rm full}$ define
\begin{align}
 \mathfrak a_{\rm full}[(\sigma,u),(\eta,w)]
 &:=- (\cC u,\eta)_X+d[u,w]+(G_0u,w)_H
       +\ip{L_0\sigma}{\cC w}.                             \label{eq:full-sectorial-form}
\end{align}
Lemma~\ref{lem:full-generation} proves that this form is closed after a
scalar shift and represents $-\cA_{\rm NSP}$, with domain
\begin{equation}
 \Dom(\cA_{\rm NSP})=
 \{(\sigma,u)\in X\times H^1_{\rm ax}:
   \mathsf Du+G_0u+\cC^*L_0\sigma\in H\},                 \label{eq:full-generator-domain}
\end{equation}
where the sum is first interpreted in $(H^1_{\rm ax})^*$.  Thus
\begin{equation}
 \cA_{\rm NSP}(\sigma,u)
 =(\cC u,-\mathsf Du-G_0u-\cC^*L_0\sigma).                   \label{eq:full-generator}
\end{equation}
All spectral statements below refer to this operator realization.

\subsection{The kernel of viscosity}

Put
\begin{equation}
 q_{\rm rot}=r e_\theta,\qquad
 M=\norm{e_z}_H^2=\int_{\Omega_{\bar\rho}}\bar\rho dx,\qquad
 I=\norm{q_{\rm rot}}_H^2=\int_{\Omega_{\bar\rho}}\bar\rho r^2dx.                     \label{eq:rigid-vectors}
\end{equation}

\begin{lemma}
\label{lem:Korn}
Let $\Omega_{\bar\rho}$ be connected, bounded, axisymmetric, and Lipschitz,
and suppose that $\bar\rho>0$ almost everywhere in $\Omega_{\bar\rho}$ and
$\bar\rho\in L^\infty$.  Then
\begin{equation}
 \Ker d=\Span\{e_z,q_{\rm rot}\}.                                                      \label{eq:d-kernel}
\end{equation}
On the full axisymmetric space one has the generalized Korn estimate
\begin{equation}
 d[v]+\|v\|_H^2\ge c_{K,0}\|v\|_{H^1}^2.                 \label{eq:Korn-shifted}
\end{equation}
If
\begin{equation}
 H_\diamond=\{v\in H:(v,e_z)_H=(v,q_{\rm rot})_H=0\},\qquad
 V_\diamond=H^1_{\rm ax}(\Omega_{\bar\rho};\C^3)\cap H_\diamond,           \label{eq:diamond-spaces}
\end{equation}
then
\begin{equation}
 d[v]\ge c_K\norm v_{H^1}^2\ge c_H\norm v_H^2,\qquad v\in V_\diamond.      \label{eq:Korn-coercive}
\end{equation}
\end{lemma}

\begin{proof}
Let
\begin{equation}
 \mathcal R=\{x\mapsto a+Bx: a\in\C^3, B\in\C^{3\times3},\ B^T=-B\}
                                                               \label{eq:rigid-motion-space}
\end{equation}
be the space of infinitesimal rigid motions.  We first recall the two Korn
facts used below.  On a bounded connected Lipschitz domain,
\begin{align}
 \|w\|_{H^1}
 &\le C_\Omega\bigl(\|e(w)\|_{L^2}+\|w\|_{L^2}\bigr),
                                                               \label{eq:Korn-second}\\
 \inf_{\xi\in\mathcal R}\|w-\xi\|_{H^1}
 &\le C_\Omega\|e(w)\|_{L^2}.                             \label{eq:Korn-modulo-rigid}
\end{align}
Here $e(w)$ is the symmetric gradient defined in \eqref{eq:viscous-stress}. The kernel and the second estimate follow from
the first one as follows.  The distributional identity
\begin{equation}
 \partial_j\partial_k w_i
 =\partial_j e_{ik}(w)+\partial_k e_{ij}(w)-\partial_i e_{jk}(w)
                                                               \label{eq:rigidity-identity}
\end{equation}
shows that $e(w)=0$ implies that all second derivatives of $w$ vanish.
Connectedness gives $w(x)=a+Bx$, and $e(w)=0$ gives $B^T=-B$.
The argument applies separately to the real and imaginary parts.

To derive \eqref{eq:Korn-modulo-rigid}, let $\Pi_{\mathcal R}$ be the
$H^1$-orthogonal projection onto the finite-dimensional space $\mathcal R$.
If the estimate failed, there would be
$w_n\in\mathcal R^{\perp_{H^1}}$ such that
\begin{equation}
 \|w_n\|_{H^1}=1,\qquad \|e(w_n)\|_{L^2}\longrightarrow0.
\end{equation}
After passage to a subsequence, $w_n\rightharpoonup w$ in $H^1$ and
$w_n\to w$ in $L^2$.  Then $e(w)=0$, so $w\in\mathcal R$, while the weak
limit remains in $\mathcal R^{\perp_{H^1}}$.  Hence $w=0$.
Applying \eqref{eq:Korn-second} to $w_n-w$ now gives $w_n\to0$ in $H^1$,
a contradiction.

For an axisymmetric field, the approximating rigid motion in
\eqref{eq:Korn-modulo-rigid} may be chosen axisymmetric.  If $R_\vartheta$
is rotation about the $z$-axis, define
\begin{equation}
 (\mathcal Q\xi)(x)=\frac1{2\pi}\int_0^{2\pi}
       R_{-\vartheta}\xi(R_\vartheta x)\,d\vartheta .        \label{eq:rigid-axial-average}
\end{equation}
The averaging operator $\mathcal Q$ is a contraction on $H^1$, preserves
$\mathcal R$, and fixes every axisymmetric vector field.  Thus
$\|u-\mathcal Q\xi\|_{H^1}\le\|u-\xi\|_{H^1}$.  Its fixed subspace in
$\mathcal R$ is
\begin{equation}
 \mathcal R_{\rm ax}
 =\Span\{e_z,e_z\times x\}=\Span\{e_z,q_{\rm rot}\}.        \label{eq:axisymmetric-rigid-space}
\end{equation}
This also proves directly that every axisymmetric field with $e(u)=0$ has
the form $u=a e_z+bq_{\rm rot}$.

Since
\begin{equation}
 |e(u)|^2=|\dev e(u)|^2+\frac13|\nabla\cdot u|^2
\end{equation}
and $\nu_s,\nu_b>0$, there are $c_\nu,C_\nu>0$ such that
\begin{equation}
 c_\nu\|e(u)\|_{L^2}^2\le d[u]
 \le C_\nu\|e(u)\|_{L^2}^2.                              \label{eq:d-e-equivalence}
\end{equation}
This proves \eqref{eq:d-kernel}.

We next prove \eqref{eq:Korn-shifted}.  For $u\in H^1_{\rm ax}$, choose
$\xi=a e_z+bq_{\rm rot}\in\mathcal R_{\rm ax}$ with
\begin{equation}
 \|u-\xi\|_{H^1}\le C d[u]^{1/2}.                         \label{eq:rigid-approximant}
\end{equation}
The two rigid vectors are $H$-orthogonal and
\begin{equation}
 \|a e_z+bq_{\rm rot}\|_H^2=M|a|^2+I|b|^2.               \label{eq:weighted-rigid-norm}
\end{equation}
Thus the weighted $H$ norm is a norm on $\mathcal R_{\rm ax}$, and
finite-dimensional norm equivalence gives
\begin{align}
 \|\xi\|_{H^1}
 &\le C\|\xi\|_H
 \le C\bigl(\|u\|_H+\|u-\xi\|_H\bigr)\notag\\
 &\le C\bigl(\|u\|_H+d[u]^{1/2}\bigr).                   \label{eq:rigid-weight-control}
\end{align}
Combining this with \eqref{eq:rigid-approximant} and then squaring proves
\eqref{eq:Korn-shifted}.

Finally let $u\in V_\diamond$.  Since
$u\perp_H\mathcal R_{\rm ax}$, the same $\xi$ satisfies
\begin{equation}
 \|\xi\|_H^2=(\xi-u,\xi)_H
 \le\|\xi-u\|_H\|\xi\|_H.
\end{equation}
Hence $\|\xi\|_H\le\|u-\xi\|_H\le C d[u]^{1/2}$ and, by norm
equivalence on $\mathcal R_{\rm ax}$,
$\|\xi\|_{H^1}\le C d[u]^{1/2}$.  Together with
\eqref{eq:rigid-approximant}, this yields
$\|u\|_{H^1}^2\le C d[u]$.  The bound $\bar\rho\in L^\infty$ gives
\begin{equation}
 \|u\|_H^2\le\|\bar\rho\|_{L^\infty}\|u\|_{L^2}^2
 \le C\|u\|_{H^1}^2,
\end{equation}
and proves \eqref{eq:Korn-coercive}.

Here $H^1_{\rm ax}$ is the closure of smooth Cartesian axisymmetric vector
fields.  The relations $u_r=u_\theta=0$ at $r=0$ hold for smooth
representatives; no codimension-two $H^1$ trace on the axis is used.
\end{proof}

Removing both $e_z$ and $q_{\rm rot}$ therefore gives the closed,
coercive form required in Section~\ref{sec:KTC}.

\subsection{Conservation laws}

On the full phase space define
\begin{align}
 \mathfrak m(\sigma,u)&=\int\sigma dx,                                    \label{eq:mass-covector}\\
 \mathfrak p(\sigma,u)&=(u,e_z)_H=\int\bar\rho u_zdx,                        \label{eq:Pz-covector}\\
 \mathfrak z(\sigma,u)&=\int z\sigma dx,                                  \label{eq:Z-covector}\\
 \mathfrak j(\sigma,u)&=\omega\int r^2\sigma dx+(u,q_{\rm rot})_H.                  \label{eq:J-covector}
\end{align}
Here $\mathfrak j$ is the variation of total angular momentum and
$\mathfrak p$ is the variation of axial linear momentum.

\begin{proposition}
\label{prop:automatic-constraints}
Along every solution of \eqref{eq:first-order-full},
\begin{equation}
 \mathfrak m_t=\mathfrak p_t=\mathfrak j_t=0,\qquad
 \mathfrak z_t=\mathfrak p.                                               \label{eq:conservation-linear}
\end{equation}
If $E_\lambda$ is a generalized eigenspace of the full generator at
$\lambda\ne0$, then
\begin{equation}
 \mathfrak m=\mathfrak p=\mathfrak j=\mathfrak z=0
 \quad\hbox{on }E_\lambda.                                                \label{eq:automatic-all}
\end{equation}
In particular \eqref{eq:automatic-all} holds on the entire Riesz unstable
space, including all generalized eigenvectors.
\end{proposition}

\begin{proof}
Mass conservation follows from \eqref{eq:linear-mass}.  Let
\begin{equation}
 \tau:=\cC e_z=-\partial_z\bar\rho.
\end{equation}
Differentiating \eqref{eq:equilibrium} in $z$ gives
\begin{equation}
 L_0\tau=0.                                                              \label{eq:translation-density-kernel}
\end{equation}
Since $D e_z=0$ and $(G_0u,e_z)_H=0$,
\begin{equation}
 \langle \cC^*L_0\sigma,e_z\rangle
 =\langle L_0\sigma,\cC e_z\rangle
 =\overline{\langle L_0\tau,\sigma\rangle}=0,
\end{equation}
where the brackets are the corresponding dual pairings,
which proves $\mathfrak p_t=0$.  Integration by parts gives
$\mathfrak z_t=(u,e_z)_H=\mathfrak p$.

For angular momentum set $R(u)=\int\bar\rho r u_rdx$.  Then
\begin{equation}
 \int r^2\cC u\,dx=2R(u),\qquad
 (G_0u,q_{\rm rot})_H=2\omega R(u),\qquad Dq_{\rm rot}=\cC q_{\rm rot}=0.
\end{equation}
The two contributions cancel in $\mathfrak j_t$.

We verify directly that these identities hold on the full operator domain.
Let $U=(\sigma,u)\in\Dom(\cA_{\rm NSP})$.  Extend $\bar\rho u$ by zero.
By the equilibrium identity, $h(\bar\rho)\in W^{1,\infty}$; since
$1/h'(\bar\rho)$ is bounded, the chain rule gives
$\bar\rho\in W^{1,\infty}(\Omega_{\bar\rho})$ with zero trace.  Thus the
zero extension of $\bar\rho u$ belongs to $W^{1,1}(\mathbb R^3)$ and has
no boundary measure in its distributional divergence.  Distributional
integration by parts is therefore legitimate.
Using compactly supported smooth cutoffs equal to $1,z,r^2$ on a
neighborhood of $\overline\Omega_{\bar\rho}$ gives
\begin{equation}
 \int \cC u\,dx=0,\qquad
 \int z\cC u\,dx=(u,e_z)_H,\qquad
 \int r^2\cC u\,dx=2\int\bar\rho r u_r\,dx.                 \label{eq:domain-conservation-ibp}
\end{equation}
The momentum residual
\begin{equation}
 F=\mathsf Du+G_0u+\cC^*L_0\sigma\in H
\end{equation}
is first defined in $(H^1_{\rm ax})^*$.  Testing it against the form-domain
vectors $e_z$ and $q_{\rm rot}$ is legitimate.  The identities
$d[u,e_z]=d[u,q_{\rm rot}]=0$, $\cC e_z=\tau$, $\cC q_{\rm rot}=0$,
$L_0\tau=0$, and the two Coriolis identities above
give, together with \eqref{eq:domain-conservation-ibp},
\begin{equation}
 \mathfrak m\cA_{\rm NSP}=0,\qquad
 \mathfrak p\cA_{\rm NSP}=0,\qquad
 \mathfrak j\cA_{\rm NSP}=0,\qquad
 \mathfrak z\cA_{\rm NSP}=\mathfrak p                 \label{eq:adjoint-covector-identities}
\end{equation}
on $\Dom(\cA_{\rm NSP})$.  Thus
$\mathfrak m,\mathfrak p,\mathfrak j\in\Dom(\cA_{\rm NSP}^*)$ with
adjoint image zero, while
$\mathfrak z\in\Dom(\cA_{\rm NSP}^*)$ with adjoint image $\mathfrak p$.

A conserved bounded functional $\ell$ satisfies
$\ell\cA_{\rm NSP}=0$.  On $E_\lambda$, the restriction of
$\cA_{\rm NSP}$ is invertible, so
\begin{equation}
 \ell|_{E_\lambda}
 =(\ell\cA_{\rm NSP}|_{E_\lambda})
  (\cA_{\rm NSP}|_{E_\lambda})^{-1}=0.                                  \label{eq:covector-invertibility}
\end{equation}
This proves the first three vanishings.  Since $E_\lambda$ is invariant and
$\mathfrak z\cA_{\rm NSP}=\mathfrak p$, the last follows either from the
same root-chain induction or from
$\mathfrak z(\cA_{\rm NSP}-\lambda)=-\lambda\mathfrak z$ on $E_\lambda$.
\end{proof}

These constraints express biorthogonality to the adjoint conserved
covectors. The translation center chain is
\begin{equation}
 y_T=(\tau,0),\qquad y_B=(0,e_z),\qquad
 \cA_{\rm NSP}y_T=0,\quad\cA_{\rm NSP}y_B=y_T,                          \label{eq:translation-chain}
\end{equation}
Moreover $\int z\tau dx=M$, so $\mathfrak p=\mathfrak z=0$ is precisely the
dual separation from this center block.

\subsection{Routh reduction}

Let $P_\diamond$ be the $H$-orthogonal projection onto $H_\diamond$:
\begin{equation}
 P_\diamond u=u-\frac{(u,e_z)_H}{M}e_z
                 -\frac{(u,q_{\rm rot})_H}{I}q_{\rm rot}.              \label{eq:Pdiamond}
\end{equation}
Put
\begin{equation}
 \mathfrak r(\sigma)=\int r^2\sigma dx.                                  \label{eq:radial-moment-functional}
\end{equation}
On the invariant subspace $\Ker\mathfrak p\cap\Ker\mathfrak j$, define
\begin{equation}
 v=P_\diamond u,\qquad
 u=v-\frac{\omega}{I}\mathfrak r(\sigma)\,q_{\rm rot}.                 \label{eq:Routh-transform}
\end{equation}
This is a bounded isomorphism onto $X\times H_\diamond$.  Define
\begin{equation}
 L_J=L_0+\frac{\omega^2}{I}\mathfrak r^*\mathfrak r,\qquad
 G=P_\diamond G_0|_{H_\diamond},\qquad
 \cC_\diamond=\cC|_{V_\diamond}.                                             \label{eq:reduced-operators}
\end{equation}
Here $\cC_\diamond:V_\diamond\to X$ is a form-domain operator, and
$\langle\mathfrak r^*z,\eta\rangle=z\overline{\mathfrak r(\eta)}$.
Thus the rank-one form is
$\mathfrak r(\sigma)\overline{\mathfrak r(\eta)}$.
Let $R_X:X\to X^*$ be the Riesz isomorphism for the weighted density
inner product.  We write
\begin{equation}
 \widehat L_J:=R_X^{-1}L_J:X\to X                              \label{eq:LJ-Riesz}
\end{equation}
for the bounded self-adjoint Riesz realization.  Thus $L_J$ denotes the
form map, while spectral projections of the density Hessian refer to
$\widehat L_J$.

\begin{proposition}
\label{prop:Routh-reduction}
The transformation \eqref{eq:Routh-transform} intertwines the full system
on $\Ker\mathfrak p\cap\Ker\mathfrak j$ with
\begin{equation}
 \sigma_t=\cC_\diamond v,\qquad
 v_t+D_\diamond v+Gv+\cC_\diamond^*L_J\sigma=0,                            \label{eq:first-order-reduced}
\end{equation}
where $D_\diamond$ is the $H_\diamond$-realization of $d$ and
$\mathsf D_\diamond:V_\diamond\to V_\diamond^*$ is its form map.  In cylindrical
components,
\begin{equation}
 G
 \begin{pmatrix}v_r\\v_z\\v_\theta\end{pmatrix}
 =
 \begin{pmatrix}
 -2\omega v_\theta\\[2mm]
 0\\[1mm]
 2\omega\left(v_r-\dfrac rI\int\bar\rho r v_rdx\right)
 \end{pmatrix}.                                                          \label{eq:reduced-G-components}
\end{equation}
For every $\lambda\ne0$, the generalized eigenspace of the full generator,
which lies in $\Ker\mathfrak p\cap\Ker\mathfrak j$, is carried to the
corresponding generalized eigenspace of \eqref{eq:first-order-reduced}, with
identical Jordan block sizes.  The additional mass and centering reduction
to $X_{00}$ is made below.
\end{proposition}

\begin{proof}
Since $\cC q_{\rm rot}=0$ and $Dq_{\rm rot}=0$, the continuity equation is
$\sigma_t=\cC_\diamond v$.  Moreover,
\begin{equation}
 G_0q_{\rm rot}=2\omega e_z\times(re_\theta)=-2\omega r e_r,
 \qquad r e_r\in H_\diamond.                              \label{eq:G-rigid-rotation}
\end{equation}
If $a(\sigma)=(\omega/I)\mathfrak r(\sigma)$, then projection of the
momentum equation onto $H_\diamond$ produces the force
\begin{align}
 -a(\sigma)P_\diamond G_0q_{\rm rot}
 &=\frac{2\omega^2}{I}\mathfrak r(\sigma)r e_r =\cC_\diamond^*\left[
     \frac{\omega^2}{I}\mathfrak r(\sigma)r^2\right].       \label{eq:positive-rank-one}
\end{align}
This is exactly the rank-one part of $\cC_\diamond^*L_J\sigma$, with the
positive sign in \eqref{eq:reduced-operators}.  The term obtained by
differentiating $a(\sigma)q_{\rm rot}$ is parallel to $q_{\rm rot}$ and
is killed by $P_\diamond$.  The displayed force identity is understood
through the weighted embedding into $V_\diamond^*$.  Hence the projected equation is
\eqref{eq:first-order-reduced}.  Finally, direct evaluation of
$P_\diamond G_0v$ gives \eqref{eq:reduced-G-components}.

Every nonzero generalized eigenspace of the full generator lies in the
constraint subspace by Proposition~\ref{prop:automatic-constraints}; the
bounded intertwining map and its inverse therefore preserve its complete
Jordan structure.
\end{proof}

Differentiating the second equation in \eqref{eq:first-order-reduced} and
using the first gives
\begin{equation}
 v_{tt}+D_\diamond v_t+Gv_t+Kv=0,\qquad
 K=\cC_\diamond^*L_J\cC_\diamond.                                           \label{eq:reduced-second-order}
\end{equation}
Its form pencil $P(\lambda):V_\diamond\to V_\diamond^*$ is defined by
\begin{equation}
 \langle P(\lambda)v,w\rangle
 =\lambda^2(v,w)_H+\lambda d[v,w]+\lambda(Gv,w)_H
  +\ip{L_J\cC_\diamond v}{\cC_\diamond w}.                                  \label{eq:NSP-form-pencil}
\end{equation}
For this pencil we write
\begin{equation}
 N_+(P):=\sum_{\lambda\in\mathbb C_+}
          \operatorname{algmult}(\lambda;P),              \label{eq:NSP-pencil-count-definition}
\end{equation}
where the multiplicity is the Gohberg--Sigal multiplicity recalled in
Appendix~\ref{app:Gohberg-Sigal}.

For the spectral reduction we use the centered density phase space
\begin{equation}
 \mathscr H_{\rm red}=X_{00}\times H_\diamond,\qquad
 \mathscr V_{\rm red}=X_{00}\times V_\diamond.            \label{eq:reduced-phase-spaces}
\end{equation}
The reduced generator is the realization
\begin{align}
 \cA_{\rm red}(\sigma,v)
 &=\bigl(\cC_\diamond v,-\mathsf D_\diamond v-Gv-\cC_\diamond^*L_J\sigma\bigr),
                                                               \label{eq:reduced-generator}\\
 \Dom(\cA_{\rm red})
 &=\{(\sigma,v)\in X_{00}\times V_\diamond:
 \mathsf D_\diamond v+Gv+\cC_\diamond^*L_J\sigma\in H_\diamond\}. \label{eq:reduced-generator-domain}
\end{align}
Lemma~\ref{lem:full-generation}, with $L_0,\cC,G_0$ replaced by
$L_J,\cC_\diamond,G$, proves that this is a closed analytic-semigroup
generator.

\begin{proposition}
\label{prop:root-chain-equivalence}
For every $\lambda\in\C_+$, the reduced first-order characteristic operator
\begin{equation}
 \mathcal M(\lambda)=
 \begin{pmatrix}
 \lambda&-\cC_\diamond\\
 \cC_\diamond^*L_J&\lambda+\mathsf D_\diamond+G
 \end{pmatrix}                                                          \label{eq:first-order-characteristic}
\end{equation}
is the bounded map
\begin{equation}
 \mathcal M(\lambda):X_{00}\times V_\diamond
       \longrightarrow X_{00}\times V_\diamond^*.        \label{eq:M-spaces}
\end{equation}
It is Fredholm (respectively invertible) if and only if the form pencil
$P(\lambda)$ is Fredholm (respectively
invertible), with the same Fredholm index.  Moreover,
\begin{equation}
 \lambda\in\rho(\cA_{\rm red})\quad\Longleftrightarrow\quad
 P(\lambda)\text{ is invertible},                         \label{eq:resolvent-equivalence}
\end{equation}
and the generator and pencil have the same partial multiplicities at every
isolated point in $\C_+$.
\end{proposition}

\begin{proof}
On the spaces in \eqref{eq:M-spaces}, analytic Gaussian elimination gives
\begin{equation}
 \begin{pmatrix}I&0\\-\lambda^{-1}\cC_\diamond^*L_J&I\end{pmatrix}
 \mathcal M(\lambda)
 \begin{pmatrix}I&\lambda^{-1}\cC_\diamond\\0&I\end{pmatrix}
 =
 \begin{pmatrix}\lambda I&0\\0&\lambda^{-1}P(\lambda)\end{pmatrix},       \label{eq:Gaussian-factorization}
\end{equation}
where $P(\lambda)$ is defined by \eqref{eq:NSP-form-pencil}.  The outer and
scalar diagonal factors are analytic and invertible near every
$\lambda\in\C_+$.  Analytic equivalence proves the Fredholm, index, and
invertibility assertions and preserves all partial multiplicities.  In
$\C_+$ the pencil has Fredholm index zero by
Theorem~\ref{thm:KTC-pencil}.  The variational resolvent equation
$(\lambda-\cA_{\rm red})(\sigma,v)=(f,g)$ is precisely
$\mathcal M(\lambda)(\sigma,v)=(f,g)$; the representation theorem and
\eqref{eq:reduced-generator-domain} therefore give
\eqref{eq:resolvent-equivalence}.  More explicitly, a variational solution
belongs to $\Dom(\cA_{\rm red})$ because its second residual is the
prescribed $H_\diamond$ datum.  Conversely, injectivity on the operator
domain gives injectivity of the form map, and Fredholm index zero gives
surjectivity.  We spell out the domain point in the multiplicity assertion.
Let $\mathcal J_{\rm red}:\mathscr H_{\rm red}\to
\mathscr V_{\rm red}^*$ be the canonical embedding,
\begin{equation}
 \langle\mathcal J_{\rm red}(\sigma,v),(\eta,w)\rangle
 = (\sigma,\eta)_X+(v,w)_H.                              \label{eq:reduced-canonical-embedding}
\end{equation}
Identifying $X_{00}$ with its anti-dual by the density Riesz map, we have
$\mathcal M'(\lambda)=\mathcal J_{\rm red}$, and all higher derivatives
vanish. Thus a root chain $Y_0,\ldots,Y_{m-1}$, with $Y_{-1}=0$,
satisfies
\begin{equation}
 \mathcal M(\lambda_0)Y_\ell
 +\mathcal J_{\rm red}Y_{\ell-1}=0.                       \label{eq:M-root-chain-relations}
\end{equation}
The right-hand side is represented by an element of
$\mathscr H_{\rm red}$.  Starting with $\ell=0$,
\eqref{eq:reduced-generator-domain} therefore places each $Y_\ell$ in
$\Dom(\cA_{\rm red})$ inductively, and
\eqref{eq:M-root-chain-relations} is equivalent to
\begin{equation}
 (\cA_{\rm red}-\lambda_0)Y_\ell=Y_{\ell-1}.
\end{equation}
Conversely, every generator Jordan chain satisfies these variational
identities.  Thus root chains of $\mathcal M$ and Jordan chains of
$\cA_{\rm red}$ correspond bijectively with the same lengths.  The
analytic equivalence \eqref{eq:Gaussian-factorization} transfers the entire
chain to $P(\lambda)$.
\end{proof}

\begin{proposition}
\label{prop:full-spectral-transfer}
Let
\begin{equation}
 \mathscr Y=\Ker\mathfrak m\cap\Ker\mathfrak p
             \cap\Ker\mathfrak j\cap\Ker\mathfrak z
       \subset\mathscr H_{\rm full}.                       \label{eq:Y-constraints}
\end{equation}
Then $\mathscr Y$ is a closed invariant subspace.  For every
$\lambda\ne0$,
\begin{equation}
 \lambda-\cA_{\rm NSP}\text{ is Fredholm/invertible}
 \Longleftrightarrow
 \lambda-\cA_{\rm NSP}|_{\mathscr Y}\text{ is Fredholm/invertible},
                                                               \label{eq:full-Y-resolvent}
\end{equation}
with equal Fredholm index and equal algebraic multiplicities at isolated
spectral points.  The Routh map
\begin{equation}
 (\sigma,u)\longmapsto(\sigma,P_\diamond u)                \label{eq:Y-Routh-map}
\end{equation}
is a bounded, domain-preserving similarity from
$\cA_{\rm NSP}|_{\mathscr Y}$ to $\cA_{\rm red}$.
\end{proposition}

\begin{proof}
Let $T=(\mathfrak m,\mathfrak p,\mathfrak j,\mathfrak z)$.  The identities
\eqref{eq:conservation-linear} say, on $\Dom(\cA_{\rm NSP})$,
\begin{equation}
 T\cA_{\rm NSP}=NT,\qquad
 N(a,b,c,d)=(0,0,0,b).                                  \label{eq:quotient-nilpotent}
\end{equation}
For domain data, this identity gives
$T e^{t\cA_{\rm NSP}}=e^{tN}T$ by uniqueness of the finite-dimensional
evolution equation; density extends the identity to all energy data.
Thus $\mathscr Y=\Ker T$ is semigroup invariant.
The four covectors are independent: the mass and axial first-moment
functionals are independent on the density component, while the axial and
rotational velocity moments are independent on the velocity component.
Hence $T:\mathscr H_{\rm full}\to\C^4$ is onto.  Since
$\Dom(\cA_{\rm NSP})$ is dense, $T(\Dom(\cA_{\rm NSP}))$ is dense in
$\C^4$ and, being a linear subspace of a finite-dimensional space, equals
$\C^4$.  Choose a four-dimensional subspace
$F\subset\Dom(\cA_{\rm NSP})$ on which $T$ is an isomorphism.  Relative to
$\mathscr H_{\rm full}=\mathscr Y\oplus F$, one has
\begin{equation}
 \Dom(\cA_{\rm NSP})=
 (\Dom(\cA_{\rm NSP})\cap\mathscr Y)\oplus F.             \label{eq:full-domain-splitting}
\end{equation}
Indeed, subtracting $(T|_F)^{-1}Tx$ from any
$x\in\Dom(\cA_{\rm NSP})$ gives the required element of
$\Dom(\cA_{\rm NSP})\cap\mathscr Y$.  The domain operator has the
triangular form
\begin{equation}
 \cA_{\rm NSP}=
 \begin{pmatrix}\cA_{\rm NSP}|_{\mathscr Y}&B\\0&N_F\end{pmatrix},
 \quad N_F=(T|_F)^{-1}N(T|_F),                             \label{eq:full-block-form}
\end{equation}
where $B$ is the $\mathscr Y$-component of $\cA_{\rm NSP}|_F$ and is
bounded because $F$ is finite dimensional.  Since $N_F^2=0$,
$\lambda-N_F$ is invertible for $\lambda\ne0$.  The triangular Gaussian
elimination is the explicit factorization
\begin{equation}
 \lambda-\cA_{\rm NSP}
 =\begin{pmatrix}I&-B(\lambda-N_F)^{-1}\\0&I\end{pmatrix}
  \begin{pmatrix}
   \lambda-\cA_{\rm NSP}|_{\mathscr Y}&0\\0&\lambda-N_F
  \end{pmatrix}.                                           \label{eq:full-block-resolvent-factorization}
\end{equation}
This proves \eqref{eq:full-Y-resolvent}, including equality of Fredholm
indices.  Near an isolated $\lambda\ne0$ the factors are analytic and
invertible, so they preserve partial multiplicities and the Riesz algebraic
multiplicity.

On $\mathscr Y$, $\sigma\in X_{00}$, $(u,e_z)_H=0$, and
$(u,q_{\rm rot})_H=-\omega\mathfrak r(\sigma)$. Hence
\eqref{eq:Y-Routh-map} has the bounded inverse
\[
 (\sigma,v)\longmapsto
 \left(\sigma,v-\frac\omega I\mathfrak r(\sigma)q_{\rm rot}\right).
\]
We verify its effect on operator domains. Put
$a(\sigma)=(\omega/I)\mathfrak r(\sigma)$,
$u=v-a(\sigma)q_{\rm rot}$, and define the form residuals
\[
 F_{\rm full}=\mathsf Du+G_0u+\cC^*L_0\sigma,\qquad
 F_{\rm red}=\mathsf D_\diamond v+Gv+\cC_\diamond^*L_J\sigma.
\]
Proposition~\ref{prop:Routh-reduction} gives
$F_{\rm full}|_{V_\diamond}=F_{\rm red}$. Testing against the two rigid
vectors gives
\[
 \langle F_{\rm full},e_z\rangle=0,\qquad
 \langle F_{\rm full},q_{\rm rot}\rangle=2\omega R(v),
 \qquad R(v)=\int\bar\rho r v_r\,dx.
\]
Consequently, if $F_{\rm red}$ is represented by
$f_{\rm red}\in H_\diamond$, then
\[
 F_{\rm full}
 =\iota_H\left(f_{\rm red}
       +\frac{2\omega R(v)}I q_{\rm rot}\right).
\]
Conversely, an $H$-valued full residual restricts to an
$H_\diamond$-valued reduced residual. Thus both transformations preserve
the operator domains. Finally,
$a(\cC_\diamond v)=2\omega R(v)/I$, which is exactly the rotational
component obtained by differentiating the inverse transformation.
The intertwining therefore holds for the full domain operators and proves
the asserted similarity.
\end{proof}

\section{The viscous instability index}
\label{sec:index-proof}

We verify the form bounds, finite Morse index, density of the divergence
range, and semigroup properties needed for the index formula. These
properties follow from Hypothesis~\ref{hyp:EOS} and the regular-equilibrium
conditions; no equilibrium branch is needed in this section.

\begin{proposition}
\label{prop:automatic-admissibility}
Every regular rigidly rotating equilibrium whose pressure law satisfies
Hypothesis~\ref{hyp:EOS} has the following properties.
\begin{enumerate}
\item The zero extensions of \(h(\bar\rho)\) and \(\bar\rho\) belong to
\(W^{1,\infty}(\mathbb R^3)\), \(1/h'(\bar\rho)\in L^\infty\), and
\(\tau=-\partial_z\bar\rho\in X\).  Moreover,
\(h'(\bar\rho)|\nabla\bar\rho|^2\in L^\infty\) and
\begin{equation}
 L_0\tau=L_J\tau=0,\qquad
 \int\tau\,dx=\int r^2\tau\,dx=0,\qquad
 \int z\tau\,dx=M.                                      \label{eq:automatic-translation}
\end{equation}
\item The functionals \(\mathfrak m,\mathfrak p,\mathfrak z,\mathfrak j\)
defined in \eqref{eq:mass-covector}--\eqref{eq:J-covector} are bounded on
the full energy space.
\item The Riesz realization of the gravitational potential form is compact
and self-adjoint in the density energy space, and \(L_J\) has finite Morse
index.
\item The dissipation form is closed and coercive after removal of axial
translation and rigid rotation; \(V_\diamond\hookrightarrow H_\diamond\)
is compact; and \(\cC_\diamond:V_\diamond\to X_{00}\) is bounded with dense
range.
\item The full and reduced operators generate analytic semigroups on their
respective energy spaces.
\end{enumerate}
Thus all hypotheses needed to apply the abstract KTC theorem follow from
Hypothesis~\ref{hyp:EOS} and the elementary regular-equilibrium conditions.
\end{proposition}

\begin{proof}
The equilibrium equation gives
\[
 h(\bar\rho)=-V_{\bar\rho}+\tfrac12\omega^2r^2-c
 \quad\hbox{on }\Omega_{\bar\rho}.
\]
A bounded compactly supported density has
\(\nabla V_{\bar\rho}\in L^\infty\).  Since \(h(\bar\rho)=0\) on the
boundary, its zero extension, denoted by \(W\), is therefore Lipschitz.
The pressure
asymptotics and \(\gamma_0<2\) imply
\[
 m_*:=\inf_{0<s\le\|\bar\rho\|_\infty}h'(s)>0.
\]
Indeed, \(h'(s)\to+\infty\) as \(s\downarrow0\), while \(h'\) is positive
and continuous away from zero.  Let \(\mathscr G=h^{-1}\) on the relevant compact
range and set \(\mathscr G(0)=0\).  Since
\(\mathscr G'=1/h'\le m_*^{-1}\), \(\mathscr G\) is
Lipschitz there; hence the zero extension
\(\widetilde{\bar\rho}=\mathscr G(W)\) belongs to
\(W^{1,\infty}(\mathbb R^3)\).  The Sobolev chain rule gives
\begin{equation}
 h'(\bar\rho)|\nabla\bar\rho|^2
 =\frac{|\nabla W|^2}{h'(\bar\rho)}\in L^\infty.           \label{eq:automatic-gradient}
\end{equation}
Hence \(\tau\in X\).  Distributional differentiation commutes with the
Newtonian potential:
\(\partial_zV_{\widetilde{\bar\rho}}
 =V_{\partial_z\widetilde{\bar\rho}}\).  Differentiating the equilibrium
equation in the stellar interior and using
\(\tau=-\partial_z\widetilde{\bar\rho}\) therefore gives \(L_0\tau=0\).
The two rank-one moments of \(\tau\) vanish, so also \(L_J\tau=0\).
Since \(\widetilde{\bar\rho}\in W^{1,\infty}_c\), distributional
integration by parts has no boundary measure and gives
\[
 \int\tau\,dx=0,\qquad \int r^2\tau\,dx=0,\qquad
 \int z\tau\,dx=\int\bar\rho\,dx=M.
\]
This proves item 1.  Lemma~\ref{lem:bounded-covectors} proves item 2, and
Lemma~\ref{lem:gravity-compact} proves item 3.  Lemma~\ref{lem:Korn},
Lemma~\ref{lem:C-form-bound}, and
Proposition~\ref{prop:helmholtz-inertia} prove item 4, while
Lemma~\ref{lem:full-generation} proves item 5. The kernel of the density
Hessian along the slow branch is described in
Proposition~\ref{prop:density-spectrum}.
\end{proof}

\subsection{The density--velocity form estimate}

The next estimate is the form bound required by the abstract pencil theorem.

\begin{lemma}
\label{lem:C-form-bound}
For every regular rigidly rotating equilibrium whose pressure law satisfies
Hypothesis~\ref{hyp:EOS}, and every $u\in H^1_{\rm ax}$,
\begin{equation}
 \norm{\cC u}_X^2
 \le C\left(\norm{\nabla u}_{L^2}^2+\norm u_{L^2}^2\right).               \label{eq:C-H1-bound}
\end{equation}
Consequently, for $u,v\in V_\diamond$,
\begin{equation}
 \abs{\ip{L_J\cC_\diamond u}{\cC_\diamond v}}
 \le C\,d[u]^{1/2}d[v]^{1/2}.                                           \label{eq:potential-form-bound}
\end{equation}
\end{lemma}

\begin{proof}
Since $h'(\bar\rho)=P'(\bar\rho)/\bar\rho$,
\begin{align}
 h'(\bar\rho)\abs{\nabla\cdot(\bar\rho u)}^2
 &\le 2\bar\rho P'(\bar\rho)\abs{\nabla\cdot u}^2
   +2h'(\bar\rho)\abs{\nabla\bar\rho}^2\abs u^2.                         \label{eq:C-pointwise}
\end{align}
Both coefficients on the right are bounded.  Indeed,
$\bar\rho P'(\bar\rho)$ extends continuously to zero and is bounded on
$[0,\|\bar\rho\|_\infty]$.  For the second coefficient, differentiate
the equilibrium equation:
\begin{equation}
 \nabla h(\bar\rho)=-\nabla V_{\bar\rho}+\omega^2r e_r,
 \qquad
 h'(\bar\rho)|\nabla\bar\rho|^2
 =\frac{|\nabla h(\bar\rho)|^2}{h'(\bar\rho)}.             \label{eq:vacuum-gradient-control}
\end{equation}
The compactly supported bounded density gives
$\nabla V_{\bar\rho}\in L^\infty$, while $1/h'(\bar\rho)$ is bounded and
vanishes at the vacuum boundary.  Hence the last expression is bounded.
Integration proves
\eqref{eq:C-H1-bound}.

The multiplication part of the form $L_J:X\to X^*$ is the $X$ inner
product.  The Riesz realization of the Newtonian map is bounded and compact
on $X$, while
\begin{equation}
 \abs{\int r^2\sigma dx}
 \le\left(\int\frac{r^4}{h'(\bar\rho)}dx\right)^{1/2}\norm\sigma_X.        \label{eq:Q-bounded}
\end{equation}
Thus $L_J:X\to X^*$ and $\widehat L_J:X\to X$ are bounded.  Combine this
fact with \eqref{eq:C-H1-bound} and Korn's inequality
\eqref{eq:Korn-coercive}.
\end{proof}

\subsection{Weighted Helmholtz decomposition and inertia}

Recall the centered mass-zero space $X_{00}$ from \eqref{eq:X00}.

\begin{proposition}
\label{prop:helmholtz-inertia}
The following statements hold for every regular rigidly rotating equilibrium
whose pressure law satisfies Hypothesis~\ref{hyp:EOS}.
\begin{enumerate}
\item $\cC_\diamond:V_\diamond\to X_{00}$ is bounded and has dense range.
With
\begin{equation}
 \mathcal N=\Ker \cC_\diamond,\qquad
 V_1=\mathcal N^{\perp_d},                                                 \label{eq:weighted-Helmholtz}
\end{equation}
we have the $d$-orthogonal decomposition
$V_\diamond=\mathcal N\oplus V_1$.
\item The embedding $V_\diamond\hookrightarrow H_\diamond$ is compact.
\item The Riesz operator $A$ defined by
\begin{equation}
 (Av,w)_{V_\diamond}=\ip{L_J\cC_\diamond v}{\cC_\diamond w},\qquad
 (v,w)_{V_\diamond}=d[v,w],                                               \label{eq:NSP-A}
\end{equation}
 has finite negative index.  Moreover
\begin{equation}
 n^-(A)=n^-\!\left(L_J|_{X_{00}}\right).                                 \label{eq:inertia-X00}
\end{equation}
\end{enumerate}
\end{proposition}

\begin{proof}
Boundedness was proved in Lemma~\ref{lem:C-form-bound}.  The identities
\begin{equation}
 \int \cC_\diamond vdx=0,\qquad
 \int z\cC_\diamond vdx=(v,e_z)_H=0                                         \label{eq:range-moments}
\end{equation}
give $\Ran \cC_\diamond\subset X_{00}$.  To identify its closure, we use the
annihilator.  If $\zeta\in X$ is $X$-orthogonal to
$\cC_\diamond V_\diamond$, put $\phi=h'(\bar\rho)\zeta$ and define
\begin{equation}
 \Lambda(v):=(\cC v,\zeta)_X,\qquad v\in H^1_{\rm ax}.                      \label{eq:annihilator}
\end{equation}
This is a continuous functional, by Lemma~\ref{lem:C-form-bound}, and it
vanishes on $V_\diamond$.
Define
\begin{equation}
 S:H^1_{\rm ax}\longrightarrow\C^2,\qquad
 S(v)=\bigl((v,e_z)_H,(v,q_{\rm rot})_H\bigr).
\end{equation}
Since $S(e_z)=(M,0)$ and $S(q_{\rm rot})=(0,I)$, the map $S$ is onto and
$\Ker S=V_\diamond$.  Every continuous functional on $H^1_{\rm ax}$ that
vanishes on $\Ker S$ therefore factors through $S$.  Hence there are
constants $a,b$ such that
\begin{equation}
 \Lambda(v)=(v,a e_z+bq_{\rm rot})_H,
 \qquad v\in H^1_{\rm ax}.                                                \label{eq:annihilator-multipliers}
\end{equation}
We identify this equality distributionally. By
Proposition~\ref{prop:automatic-admissibility}, the zero extension of
$\bar\rho$ belongs to $W^{1,\infty}(\mathbb R^3)$.
If $K\Subset\Omega_{\bar\rho}$ and
$\psi\in C_c^\infty(K;\mathbb C^3)$ is axisymmetric, then
$\inf_K\bar\rho>0$, $1/\bar\rho\in W^{1,\infty}(K)$, and
$v=\psi/\bar\rho\in H^1_{\rm ax}$ is admissible.  Since
$\cC v=-\nabla\cdot\psi$, \eqref{eq:annihilator-multipliers} yields
\begin{equation}
 \langle\nabla\bar\phi,\psi\rangle
 =\int_{\Omega_{\bar\rho}}
       \psi\cdot\overline{(a e_z+bq_{\rm rot})}\,dx .                   \label{eq:annihilator-distribution}
\end{equation}
The distribution in \eqref{eq:annihilator-distribution} is axisymmetric.
Rotational averaging therefore shows that testing against axisymmetric
fields is equivalent to testing against arbitrary compactly supported
fields.  Exhausting $\Omega_{\bar\rho}$ by compact subsets gives
$\nabla\phi=a e_z+bq_{\rm rot}$ in distributions.  Taking the
distributional curl gives
$b\,\nabla\times q_{\rm rot}=0$, and
$\nabla\times q_{\rm rot}=2e_z$, so $b=0$.
Consequently $\phi=a_0+a_1z$ and
\begin{equation}
 \zeta=\frac{a_0+a_1z}{h'(\bar\rho)}.
\end{equation}
These are precisely the two Riesz representers of the mass and axial
first-moment functionals.  The double-annihilator theorem proves
$\overline{\Ran \cC_\diamond}=X_{00}$.  The integrability needed here is
verified in Lemma~\ref{lem:bounded-covectors}.

We next prove the second assertion.  Let $(v_n)$ be bounded in
$V_\diamond$ with respect to $d[v]^{1/2}$.  By
\eqref{eq:Korn-coercive}, it is bounded in
$H^1(\Omega_{\bar\rho};\C^3)$.  Passing to a subsequence, the
Rellich--Kondrachov theorem gives
\begin{equation}
 v_n\rightharpoonup v\quad\hbox{in }H^1,
 \qquad v_n\longrightarrow v\quad\hbox{in }L^2.           \label{eq:Rellich-subsequence}
\end{equation}
The axisymmetric $H^1$ space is closed, so $v$ is axisymmetric.  Since
$\bar\rho\in L^\infty$,
\begin{equation}
 \|v_n-v\|_H^2
 \le \|\bar\rho\|_{L^\infty}\|v_n-v\|_{L^2}^2\longrightarrow0.
                                                               \label{eq:weighted-Rellich}
\end{equation}
The two weighted orthogonality conditions pass to this limit:
$(v,e_z)_H=(v,q_{\rm rot})_H=0$.  Thus $v\in H_\diamond$, and the selected
subsequence converges strongly in $H_\diamond$.  This proves the compact
embedding.  Only the upper bound on $\bar\rho$ is used to pass from $L^2$ to $H$.

The stiffness kernel remains infinite-dimensional.  For example,
$v=\psi(r,z)e_\theta$ belongs to $\mathcal N$ whenever $\psi$ is smooth,
supported away from the axis and the boundary, and
$\int\bar\rho r\psi\,dx=0$.  These fields form an infinite-dimensional
space and are orthogonal to both rigid vectors.

The equality of indices follows from density of the range.  If $E\subset V_\diamond$ is $A$-negative, then
$\cC_\diamond|_E$ is injective and $\cC_\diamond E$ is $L_J$-negative.  Thus
$\dim E\le n^-(L_J|_{X_{00}})$.  Conversely, approximate a basis of any
finite-dimensional $L_J$-negative subspace of $X_{00}$ by elements of
$\Ran \cC_\diamond$.  Boundedness of $L_J$ preserves negative definiteness of
the finite Gram matrix, giving the reverse inequality.
\end{proof}

\subsection{Restoring axial translation}

Recall $\tau=-\partial_z\bar\rho$.  We have
\begin{equation}
 \int\tau dx=0,\qquad \int z\tau dx=M,\qquad
 L_J\tau=0,\qquad \mathfrak r(\tau)=0.                                   \label{eq:tau-properties}
\end{equation}
Therefore every $\sigma\in X_0$ has a unique decomposition
\begin{equation}
 \sigma=\sigma_{00}+\frac1M\left(\int z\sigma dx\right)\tau,\qquad
 \sigma_{00}\in X_{00},                                                   \label{eq:X0-decomposition}
\end{equation}
and
\begin{equation}
 Q_J[\sigma]=Q_J[\sigma_{00}].                                           \label{eq:translation-zero-form}
\end{equation}
It follows that
\begin{equation}
 n^-(L_J|_{X_{00}})=n^-(L_J|_{X_0}).                                    \label{eq:X00-X0-index}
\end{equation}
Thus centering preserves the constrained Morse index.

\subsection{Proof of the main index theorem}

\begin{proof}[Proof of Theorem~\ref{thm:main-index}]
On $V_\diamond$ take $d$ as the Hilbert inner product and set
\begin{equation}
 k[v,w]=\ip{L_J\cC_\diamond v}{\cC_\diamond w},\qquad
 g[v,w]=(Gv,w)_H.                                                        \label{eq:NSP-dkg}
\end{equation}
Lemma~\ref{lem:Korn} gives coercivity of $d$; the continuous embedding
$V_\diamond\hookrightarrow H_\diamond$ and boundedness of $G$ give
bounded $M\ge0$ and skew-adjoint $\Gamma$ in
\eqref{eq:KTC-preconditioned}.  Lemma~\ref{lem:C-form-bound} gives
boundedness of $k$, and Proposition~\ref{prop:helmholtz-inertia} gives the
finite negative index of its Riesz operator.
Thus the quadratic-pencil part of
Theorem~\ref{thm:abstract-KTC}, stated separately as
Theorem~\ref{thm:KTC-pencil}, applies to
\eqref{eq:NSP-form-pencil} and yields
\begin{equation}
 N_+(P)=n^-(\cC_\diamond^*L_J\cC_\diamond)
       =n^-(L_J|_{X_{00}}).                                              \label{eq:KTC-applied}
\end{equation}

Proposition~\ref{prop:root-chain-equivalence} transfers the Fredholm
property, the absence of continuous or residual right-half-plane spectrum,
and the algebraic count to the reduced first-order generator.  The
resolvent-level reduction and similarity in
Proposition~\ref{prop:full-spectral-transfer} transfer the complete nonzero
spectrum to the full unrestricted generator.  Finally
\eqref{eq:X00-X0-index} and
Lemma~\ref{lem:second-variation} give
\begin{align}
 N_+(\cA_{\rm NSP})
 &=n^-(L_J|_{X_{00}})
  =n^-(L_J|_{X_0}) \notag\\
 &=n^-\!\left(D^2(E+cM-\omega J)
       \big|_{\{\delta M=\delta J=0\}_{\rm ax}}\right),                  \label{eq:main-proof-chain}
\end{align}
which is \eqref{eq:main-index}.
\end{proof}

\begin{corollary}
\label{cor:no-imposed-constraints}
The four linear conditions
$\mathfrak m=\mathfrak p=\mathfrak j=\mathfrak z=0$ are automatic on
every generalized unstable eigenspace.  The axial rigid velocity is
removed by $\mathfrak p=0$.  The rotational component need not be
$H$-orthogonal to $q_{\rm rot}$; rather,
$\mathfrak j=0$ determines it uniquely from the density perturbation, and the
bounded Routh transformation removes that component without changing the
unstable spectrum.  Consequently the left side of
\eqref{eq:main-proof-chain} is the unstable count on the full
axisymmetric phase space.
\end{corollary}

\section{Fixed-\texorpdfstring{$J$}{J} equilibria and constrained Hessians}
\label{sec:equilibria}

\subsection{From fixed angular velocity to fixed total angular momentum}
Let $\rho_\mu$ be the
spherical equilibrium with center density $\rho_\mu(0)=\mu$.  We normalize
each nearby rotating density by the same center density.

Throughout this section and Appendix~\ref{app:slow-branch},
\begin{equation}
 \mathscr G(y)=\begin{cases}h^{-1}(y),&y>0,\\0,&y\le0,
       \end{cases}                                         \label{eq:positive-inverse-enthalpy}
\end{equation}
on the enthalpy range reached by the branch.

Under Hypothesis~\ref{hyp:EOS}, \(q_{\mu,0}=h(\rho_\mu)\) is the unique
Lane--Emden solution
\begin{equation}
 q_{rr}+\frac2r q_r=-4\pi\mathscr G(q),\qquad
 q(0)=h(\mu),\qquad q_r(0)=0,                            \label{eq:spherical-enthalpy-ODE}
\end{equation}
up to its first zero \(R_0(\mu)\).  The low-density exponent
\(\gamma_0>4/3>6/5\) gives \(R_0(\mu)<\infty\), and
\[
 \partial_rq_{\mu,0}(R_0(\mu))
 =-\frac{M_0(\mu)}{R_0(\mu)^2}<0.
\]
Classical ODE dependence, including the differentiated remainder bounds in
\eqref{eq:EOS-asymptotic}, gives the parameter regularity and uniform
transversality on compact center-density intervals used below; see
\cite[Section~3, especially Lemmas~3.1--3.2]{LinZeng2022}.

The exterior potential equals $-M_0(\mu)/|x|$.  Evaluating
\eqref{eq:equilibrium} on the boundary therefore gives
\[
 c_{\mu,0}=\frac{M_0(\mu)}{R_0(\mu)}.
\]
We set
\begin{equation}
 \mu_e:=\inf\left\{\mu>0:
   \partial_\mu\!\left(\frac{M_0(\mu)}{R_0(\mu)}\right)=0\right\},
                                                               \label{eq:mutilde-definition}
\end{equation}
with $\mu_e=+\infty$ if the set is empty.  By
\cite[Lemma~3.7]{LinZeng2022}, $\mu_e$ is the first parameter at
which the radial block of the spherical density Hessian can acquire a
kernel.  It is not a loss-of-invertibility parameter for the
center-normalized equilibrium equation used below.  To record
the orientation at small density, rescale the spherical enthalpy equation
by its central value and the length
\(\ell_\mu=(h(\mu)/(4\pi\mu))^{1/2}\).  The normalized equation tends
to the Lane--Emden equation of exponent \(\alpha=1/(\gamma_0-1)\).
The differentiated remainders in \eqref{eq:EOS-asymptotic} imply that both
the error in its nonlinearity and its logarithmic parameter derivative
\(\mu\partial_\mu\) are \(O(\mu^{a_0})\), uniformly on bounded
enthalpy intervals.  The integral form of the radial ODE, differentiated
by \(\mu\partial_\mu\), and transversality at the limiting first zero
therefore give
\[
 R_0(\mu)=C_R\mu^{(\gamma_0-2)/2}(1+\varepsilon_R(\mu)),\qquad
 M_0(\mu)=C_M\mu^{(3\gamma_0-4)/2}(1+\varepsilon_M(\mu)),
\]
where
\[
 |\varepsilon_R|+|\varepsilon_M|
 +|\mu\partial_\mu\varepsilon_R|
 +|\mu\partial_\mu\varepsilon_M|=O(\mu^{a_0}).
\]
Thus differentiation of the ratio is justified and gives
\begin{equation}
 \frac{M_0(\mu)}{R_0(\mu)}\sim C\mu^{\gamma_0-1},\qquad
 \partial_\mu c_{\mu,0}
 \sim C(\gamma_0-1)\mu^{\gamma_0-2}>0
 \quad(\mu\downarrow0).                                  \label{eq:low-density-c-scaling}
\end{equation}
This is the standard low-density rescaling used in
\cite[Section~3]{LinZeng2022}.
Hence $\partial_\mu c_{\mu,0}>0$ on
$(0,\mu_e)$.  This is the interval used in the spherical turning
point theorem \cite{LinZeng2022}.

\begin{lemma}
\label{lem:spherical-density-spectrum}
For every $\mu\in(0,\mu_e)$, the Riesz realization
$\widehat L_0(\mu):=\widehat L_J(\mu,0)$ of the spherical density Hessian has one simple
negative eigenvalue and
\begin{equation}
 \Ker\widehat L_0(\mu)=\Span\{\partial_z\rho_\mu\}.         \label{eq:spherical-density-spectrum}
\end{equation}
The negative eigenfunction is radial, whereas the kernel is odd in $z$.
On every $I_\mu\Subset(0,\mu_e)$, the negative eigenvalue and the
positive spectrum on the orthogonal complement of these two spaces are
separated uniformly from zero.  Moreover, let $I\Subset(0,\infty)$ and let
$D$ be a ball centered at the origin containing the supports of the
corresponding spherical stars,
and let $\mathscr G$ be as in \eqref{eq:positive-inverse-enthalpy}.  Then
\begin{equation}
 \mathcal L_\mu v
 =v+V_{\mathscr G'(q_{\mu,0})v}-V_{\mathscr G'(q_{\mu,0})v}(0),              \label{eq:normalized-enthalpy-linearization}
\end{equation}
where $q_{\mu,0}=h(\mu)+V_{\rho_\mu}(0)-V_{\rho_\mu}$, is an isomorphism
on the axisymmetric, $z$-even subspace of $C^1(\overline D)$ consisting of
functions vanishing at the origin.  Its inverse is uniformly bounded on
such compact intervals $I$.
\end{lemma}

\begin{proof}
Here $\dot H^1_{\rm ax}(\mathbb R^3)$ is the homogeneous Sobolev space of
axisymmetric scalar functions, with norm
$\|\nabla\phi\|_{L^2(\mathbb R^3)}$.
On $\dot H^1_{\rm ax}(\mathbb R^3)$ consider
\[
 D_\mu=-\Delta-\frac{4\pi}{h'(\rho_\mu)},
\]
where the coefficient is extended by zero outside the star.  The
density--potential reduction gives
$n^-(\widehat L_0)=n^-(D_\mu)$ and identifies their kernels
\cite[Lemmas~3.3--3.4]{LinZeng2022}.  In spherical harmonics,
\[
 D_\mu^\ell=-\partial_r^2-\frac2r\partial_r
       +\frac{\ell(\ell+1)}{r^2}-\frac{4\pi}{h'(\rho_\mu(r))}.
\]
Differentiating the equilibrium equation gives
$D_\mu^1\partial_rV_{\rho_\mu}=0$, with
$\partial_rV_{\rho_\mu}>0$ for $r>0$.  This no-node solution satisfies
the regular $\ell=1$ boundary conditions and is therefore the ground state
of $D_\mu^1$; hence $D_\mu^1\ge0$ and its kernel is simple.  For
$\ell\ge2$,
\[
 D_\mu^\ell=D_\mu^1+\frac{\ell(\ell+1)-2}{r^2}>0
\]
in the quadratic-form sense.  Thus the only nonradial axisymmetric
kernel is axial translation, and every negative direction is radial.

For the radial operator, the normalized tangent equation yields
\begin{equation}
 \Ker D_\mu^0\ne\{0\}
 \quad\Longleftrightarrow\quad
 \partial_\mu\!\left(\frac{M_0(\mu)}{R_0(\mu)}\right)=0.  \label{eq:radial-kernel-cprime}
\end{equation}
The radial Morse index is one at small center density and is constant away
from the parameters in \eqref{eq:radial-kernel-cprime}; see
\cite[Lemmas~3.7--3.9]{LinZeng2022}.  This proves the pointwise assertions.
For the uniform statement, choose a ball $D$ containing the supports for
$\mu\in I_\mu$ and set
$b_\mu=[\mathscr G'(q_{\mu,0})]^{1/2}$, extended by zero.  Under the
isometry $\sigma\mapsto h'(\rho_\mu)^{1/2}\sigma$, followed by zero
extension, the Hessian is
\[
 I+\mathbb G_\mu,\qquad
 \mathbb G_\mu(x,y)=-\frac{b_\mu(x)b_\mu(y)}{|x-y|},
\]
with the identity on the exterior complement.  The fixed-domain spherical
dependence gives $b_\mu$ continuous in $C(\overline D)$, so
$\mathbb G_\mu$ is Hilbert--Schmidt and operator-norm continuous.  If the
positive gap collapsed along a sequence in $I_\mu$, norm convergence at a
subsequential limit would produce an additional zero mode, contrary to the
pointwise kernel description.  Compactness of $I_\mu$ therefore gives the
uniform gaps.

It remains to treat the center-normalized operator $\mathcal L_\mu$.  Put
$\alpha=1/(\gamma_0-1)>1$ and
$a_\mu=\mathscr G'(q_{\mu,0})$, extended by zero outside the star.  For every
$0<\eta<\min\{1,\alpha-1\}$, $a_\mu\in C_c^{0,\eta}$; hence
$v\mapsto V_{a_\mu v}-V_{a_\mu v}(0)$ maps bounded subsets of $C^1$ into
bounded subsets of $C^{2,\eta}$ and is compact on $C^1$.  Thus
$\mathcal L_\mu$ is Fredholm of index zero.  Suppose
$\mathcal L_\mu v=0$.  Then
\[
 v=-V_{a_\mu v}+V_{a_\mu v}(0),
\]
so the global right-hand side solves
$(-\Delta-4\pi a_\mu)v=0$.  Decompose $v$ into axisymmetric spherical
harmonics.  Reflection symmetry leaves only even degrees.  For every
$\ell\ge2$, the corresponding component is a decaying $\dot H^1$ solution
and vanishes because $D_\mu^\ell>0$ by
\cite[Lemma~3.5(ii)]{LinZeng2022}.  The radial component satisfies
$v_0(0)=0$ and, by regularity at the origin, $v_0'(0)=0$; uniqueness for
\[
 v_0''+\frac2r v_0'+4\pi a_\mu(r)v_0=0
\]
gives $v_0\equiv0$.  Indeed,
\[
 r^2v_0'(r)=-4\pi\int_0^r s^2a_\mu(s)v_0(s)\,ds,
\]
so \(|v_0'(r)|\le Cr\sup_{0\le s\le r}|v_0(s)|\).  After integration,
the supremum on the right is at most \(Cr^2\) times itself; it vanishes
for small \(r\), and ordinary ODE uniqueness continues the zero solution.
The $\ell=1$ translation modes are $z$-odd in the axisymmetric class and
are excluded.  The radial initial-value argument does not require the
global radial component to decay at infinity.  Thus $\mathcal L_\mu$ is injective and
therefore invertible.  Operator-norm continuity in $\mu$ gives a uniform
inverse bound on compact parameter intervals.
\end{proof}

\begin{lemma}
\label{lem:first-maximum-before-mutilde}
If \eqref{eq:nondegenerate-maximum} holds, then
\begin{equation}
 \left.\partial_\mu\!\left(\frac{M_0(\mu)}{R_0(\mu)}\right)
 \right|_{\mu=\mu_*}>0,
 \qquad \mu_*<\mu_e.                              \label{eq:first-max-before-mutilde}
\end{equation}
\end{lemma}

\begin{proof}
Apply \cite[Lemma~3.14]{LinZeng2022} with $\mu_0=\mu_*$.  Since
$\partial_\mu M_0(\mu)>0$ for $0<\mu<\mu_*$, it gives
\[
 \partial_\mu c_{\mu,0}
 =\partial_\mu\!\left(M_0(\mu)/R_0(\mu)\right)>0
\]
on that interval.  Hence $\partial_\mu c_{\mu_*,0}\ge0$ by continuity.
Because $\partial_\mu M_0(\mu_*)=0$,
\cite[Lemma~3.10]{LinZeng2022} implies
$\partial_\mu c_{\mu_*,0}\ne0$, and therefore
$\partial_\mu c_{\mu_*,0}>0$.  Continuity gives $\varepsilon>0$ for which
$\partial_\mu c_{\mu,0}>0$ on $(0,\mu_*+\varepsilon)$.  The definition of
$\mu_e$ now yields $\mu_*<\mu_e$.
\end{proof}

\begin{proposition}
\label{prop:slow-branch-full-range}
Assume Hypothesis~\ref{hyp:EOS} and let $I_\mu\Subset(0,\infty)$.  There exist
$s_0>0$ and a fixed ambient ball $D$.  Among solutions in a fixed
$C^1$-neighborhood of the spherical branch, there is a unique
axisymmetric, equatorially symmetric family
$q_{\mu,s}\in C^1(\overline D)$,
$(\mu,s)\in I_\mu\times[0,s_0)$, such that
\begin{equation}
 q_{\mu,s}=h(\mu)+V_{\mathscr G(q_{\mu,s})}(0)-V_{\mathscr G(q_{\mu,s})}
                 +\tfrac12sr^2,\qquad
 \rho_{\mu,s}=\mathscr G(q_{\mu,s}).                              \label{eq:fixed-omega-branch}
\end{equation}
The parameter maps satisfy
\begin{equation}
 q_{\mu,s}\in C^2_{\mu,s}(C^1(\overline D)),\qquad
 \widetilde\rho_{\mu,s}\in
 C^2_{\mu,s}(L^1(D))\cap C^1_{\mu,s}(C(\overline D)),    \label{eq:branch-parameter-regularity}
\end{equation}
where $\widetilde\rho$ denotes zero extension.  In particular,
\begin{equation}
 M(\mu,s)=\int\rho_{\mu,s}\,dx,\quad
 I(\mu,s)=\int r^2\rho_{\mu,s}\,dx,\quad
 c(\mu,s)=-h(\mu)-V_{\rho_{\mu,s}}(0)                    \label{eq:branch-scalars}
\end{equation}
are $C^2$.
Each $(\rho_{\mu,s},\sqrt{s}\,r e_\theta)$ is a regular rigidly rotating
equilibrium.  These conclusions hold throughout $4/3<\gamma_0<2$,
including $\gamma_0=5/3$.
\end{proposition}

\begin{proposition}
\label{prop:slow-vacuum-tangents}
For the family of Proposition~\ref{prop:slow-branch-full-range}, put
$\alpha=1/(\gamma_0-1)$ and $B=B_1(0)$.  After decreasing $s_0$ if necessary,
the supports form a uniformly strictly convex $C^{3,\eta}$ family for
every $0<\eta<\min\{1,\alpha-1\}$.  There are uniformly bi-Lipschitz
support maps $\Psi_{\mu,s}:B\to\Omega_{\mu,s}$, continuous in
$W^{1,\infty}$, and the defining function $\delta_B(a)=1-|a|$ such that,
in a fixed boundary collar,
\begin{equation}
 q_{\mu,s}(\Psi_{\mu,s}(a))=A_{\mu,s}(a)\delta_B(a),\quad
 0<A_*\le A_{\mu,s}\le A^*,\qquad
 \rho_{\mu,s}\asymp d_{\mu,s}^{\alpha},\quad
 \frac1{h'(\rho_{\mu,s})}\asymp d_{\mu,s}^{\alpha-1}.     \label{eq:uniform-vacuum-factor}
\end{equation}
Here \(d_{\mu,s}(x)=\operatorname{dist}
(x,\partial\Omega_{\mu,s})\).
For $p\in\{\mu,s\}$, the Eulerian tangent
$f^p_{\mu,s}=\partial_p\widetilde\rho_{\mu,s}$ satisfies
\begin{equation}
 f^p_{\mu,s}=\mathbf1_{\{q_{\mu,s}>0\}}\mathscr G'(q_{\mu,s})
                  \partial_pq_{\mu,s},\qquad
 |f^p_{\mu,s}|\le C d_{\mu,s}^{\alpha-1},                 \label{eq:Eulerian-branch-tangent}
\end{equation}
and depends continuously on the parameters in the density energy space
under the unitary pullback
\begin{equation}
 (U_{\mu,s}f)(a)=
 \left[h'(\rho_{\mu,s}(\Psi_{\mu,s}(a)))
       |\det D\Psi_{\mu,s}(a)|\right]^{1/2}
 f(\Psi_{\mu,s}(a)).                                      \label{eq:density-unitary-pullback}
\end{equation}
\end{proposition}
Both propositions are proved in Appendix~\ref{app:slow-branch}.  The second
records the boundary estimates and identifications used in the spectral
and inviscid comparisons.

We next reparametrize this family by total angular momentum.

\begin{proposition}
\label{prop:fixed-J-branch}
For the branch of Proposition~\ref{prop:slow-branch-full-range}, there is
$\kappa_0>0$
and a unique $C^2$
function $\omega=\omega(\mu,\kappa)$ on
$I_\mu\times(-\kappa_0,\kappa_0)$ such that
\begin{equation}
 \omega(\mu,\kappa)I(\rho_{\mu,\omega(\mu,\kappa)^2})=\kappa.
 \label{eq:J-scalar}
\end{equation}
Consequently
\begin{equation}
 \rho_{\mu,\kappa}:=\rho_{\mu,\omega(\mu,\kappa)^2},
 \qquad c_{\mu,\kappa}:=c_{\mu,\omega(\mu,\kappa)^2},
 \qquad v_{\mu,\kappa}=\omega(\mu,\kappa)r e_\theta
 \label{eq:fixed-J-family}
\end{equation}
has total angular momentum exactly $J=\kappa$.  We write
\begin{equation}
 \Psi_{\mu,\kappa}:=\Psi_{\mu,\omega(\mu,\kappa)^2}.       \label{eq:fixed-J-support-map}
\end{equation}
Uniformly for
$\mu\in I_\mu$,
\begin{align}
 \omega(\mu,\kappa)&=\frac{\kappa}{I(\rho_\mu)}+O(\kappa^3),
 \label{eq:omega-expansion}\\
 \|\widetilde\rho_{\mu,\kappa}-\widetilde\rho_\mu\|_{C(\overline D)}
   &=O(\kappa^2),
 \qquad
 U_{\mu,\kappa}\partial_\mu\widetilde\rho_{\mu,\kappa}
 \longrightarrow U_{\mu,0}\partial_\mu\widetilde\rho_{\mu,0}
 \quad(\kappa\to0).                                    \label{eq:rho-expansion}
\end{align}
Moreover $\omega(\mu,-\kappa)=-\omega(\mu,\kappa)$.
\end{proposition}

\begin{proof}
Write $I(\mu,s)=I(\rho_{\mu,s})$ and
\begin{equation}
 F(\mu,\omega,\kappa)=\omega I(\mu,\omega^2)-\kappa.
\end{equation}
The moment of inertia is strictly positive and continuous on the compact
interval.  Therefore there is $I_*>0$ such that
\begin{equation}
 \partial_\omega F(\mu,0,0)=I(\mu,0)\ge I_*>0.
\end{equation}
The uniform parameter-dependent implicit-function theorem yields the unique
$C^2$ solution of \eqref{eq:J-scalar}.  Since
$I(\mu,\omega^2)=I(\mu,0)+O(\omega^2)$, first $\omega=O(\kappa)$ and then
\eqref{eq:omega-expansion} follow.  Oddness follows from uniqueness.  The
fundamental theorem of calculus in $s$ and the bound on $f^s_{\mu,s}$ give
the first estimate in \eqref{eq:rho-expansion}.  At fixed $\kappa$, the
Eulerian chain rule gives
\begin{equation}
 \partial_\mu\widetilde\rho_{\mu,\kappa}
 =f^\mu_{\mu,s}+2\omega\,\partial_\mu\omega\,f^s_{\mu,s},
 \qquad s=\omega^2.                                      \label{eq:fixed-J-tangent-chain-rule}
\end{equation}
The second term tends to zero because
$\omega\partial_\mu\omega=O(\kappa^2)$.  Indeed, differentiating
\eqref{eq:J-scalar} at fixed $\kappa$ gives
\begin{equation}
 \partial_\mu\omega
 =-\frac{\omega\,\partial_\mu I(\mu,\omega^2)}
         {I(\mu,\omega^2)+2\omega^2\partial_sI(\mu,\omega^2)}
 =O(\kappa).                                             \label{eq:omega-mu-estimate}
\end{equation}
The first term converges by Proposition~\ref{prop:slow-vacuum-tangents}.
This proves the second assertion in \eqref{eq:rho-expansion}.  The same
chain rule and the $C^1$ dependence of the zero-extended density with values
in $C(\overline D)$ give uniform convergence of these Eulerian tangents on
the fixed ball $D$.
\end{proof}

Set
\begin{equation}
 M_\kappa(\mu)=\int_{\mathbb{R}^3}\rho_{\mu,\kappa}\,dx.
\end{equation}

\begin{proposition}
\label{prop:mass-maximum}
Suppose \eqref{eq:nondegenerate-maximum} holds, and let
$I_\mu=[\mu_-,\mu_+]\Subset(0,\mu_e)$ contain $\mu_*$ in its
interior and no other zero of $\partial_\mu M_0$.  For all
sufficiently small $\abs\kappa$, there is a unique critical point
$\mu_*^\kappa$ of $M_\kappa$ in $I_\mu$; it is a strict local maximum and
is the continuation of the first spherical mass maximum.  Moreover
\begin{equation}
 \mu_*^\kappa=\mu_*+O(\kappa^2),\qquad
 f_\kappa:=\left.\partial_\mu\rho_{\mu,\kappa}
             \right|_{\mu=\mu_*^\kappa}
 \longrightarrow
 f_*:=\left.\partial_\mu\rho_\mu\right|_{\mu=\mu_*}
 \label{eq:tangent-convergence}
\end{equation}
in the pulled-back density space and uniformly on the fixed ambient ball
after zero extension.
\end{proposition}

\begin{proof}
Proposition~\ref{prop:fixed-J-branch} and the scalar $C^2$ dependence of
$M(\mu,s)$ give $M_\kappa\to M_0$ in $C^2(I_\mu)$.  Indeed,
$s(\mu,\kappa)=\omega(\mu,\kappa)^2\to0$ in $C^2(I_\mu)$, since
$\omega$ is $C^2$ and $\omega(\mu,0)\equiv0$.
Apply the implicit-function theorem to
$\partial_\mu M_\kappa(\mu)=0$ at $(\mu_*,0)$.  Uniform convergence
preserves the sign of $\partial_\mu M_\kappa$ on the two compact pieces of
$I_\mu$ outside a small neighborhood of $\mu_*$, while
$\partial_\mu^2M_\kappa<0$ throughout that neighborhood gives the unique
sign change there.  For the stated displacement rate, use
\eqref{eq:omega-mu-estimate} and the bounded mixed derivative
$\partial_s\partial_\mu M$ to obtain, uniformly on $I_\mu$,
\begin{align*}
 \partial_\mu M_\kappa(\mu)
 &=\partial_\mu M(\mu,s)
    +2\omega\partial_\mu\omega\,\partial_s M(\mu,s)\\
 &=\partial_\mu M_0(\mu)+O(\kappa^2),\qquad s=\omega^2.
\end{align*}
At $\mu=\mu_*^\kappa$, the mean value theorem and
$\partial_\mu^2M_0(\mu_*)<0$ imply
$|\mu_*^\kappa-\mu_*|\le C\kappa^2$.
The tangent convergence follows from
\eqref{eq:fixed-J-tangent-chain-rule}, the $C(\overline D)$ parameter
continuity, and Proposition~\ref{prop:slow-vacuum-tangents}.
\end{proof}

The condition $\partial_\mu^2M_0(\mu_*)<0$ is a scalar nondegeneracy
assumption. Low- and high-density polytropic asymptotics alone do not
exclude a flat first maximum.

\subsection{Asymptotically polytropic examples}

The scalar crossing condition \eqref{eq:nondegenerate-maximum} is the
remaining input for the turning-point and strict-comparison theorems.  The
following equations of state contain the low-density exponent $5/3$ and
illustrate two high-density regimes in which mass extrema occur.

\begin{example}
\label{ex:asymptotic-polytropes}
Let $K,\rho_c,a>0$ and
\begin{equation}
 P(\rho)=K\rho^{\gamma_0}
 \left(1+(\rho/\rho_c)^a\right)^{(\gamma_\infty-\gamma_0)/a},
 \quad \frac43<\gamma_0<2,\qquad
 1<\gamma_\infty<\frac43,\qquad \gamma_\infty\ne\frac65. \label{eq:example-EOS-family}
\end{equation}
Writing $x=(\rho/\rho_c)^a$, one has
\begin{equation}
 \frac{\rho P'(\rho)}{P(\rho)}
 =\gamma_0+(\gamma_\infty-\gamma_0)\frac{x}{1+x}
 =\frac{\gamma_0+\gamma_\infty x}{1+x}>1.
\end{equation}
Moreover,
\begin{align}
 P(\rho)&=K\rho^{\gamma_0}(1+O(\rho^a))
       &&(\rho\downarrow0),\notag\\
 P(\rho)&=K\rho_c^{\gamma_0-\gamma_\infty}
          \rho^{\gamma_\infty}(1+O(\rho^{-a}))
       &&(\rho\to\infty),
\end{align}
with the corresponding differentiated asymptotics.  Thus the pressure-law
part of Hypothesis~\ref{hyp:EOS} holds.  For instance,
\begin{align}
 P_1(\rho)&=K\rho^{5/3}(1+\rho/\rho_c)^{-5/12},
 &\gamma_\infty&=\frac54,                                  \label{eq:example-P1}\\
 P_2(\rho)&=K\rho^{5/3}(1+\rho/\rho_c)^{-1/2},
 &\gamma_\infty&=\frac76.                                  \label{eq:example-P2}
\end{align}
For $P_1$, $\gamma_\infty=5/4$ and the high-density polytropic index is
$n_\infty=4$; the asymptotics give
$M_0(\mu)=C\mu^{-1/8}(1+o(1))$, whereas
$\partial_\mu M_0(\mu)>0$ for small $\mu$.
Hence the mass curve has a finite sign-changing maximum.  For $P_2$,
$\gamma_\infty=7/6$ and $n_\infty=6>5$; the high-density mass--radius
curve spirals and consequently has infinitely many mass extrema.  See
\cite[Section~3.6, especially (3.87)--(3.89) and
Corollary~3.3]{LinZeng2022}; compare
\cite[Example~1, p.~1753]{LinWang2023}.

For either law, assume that the first critical point is a nondegenerate
maximum, namely \eqref{eq:nondegenerate-maximum} holds.  This is a
one-dimensional nondegeneracy condition on the Lane--Emden solution; it may
be established analytically or by validated numerical integration, but is
not implied by the two asymptotic exponents.  Under this condition,
\cite[Theorem~1.2 and Lemmas~3.10 and 3.14]{LinZeng2022} and the results
below verify the remaining inputs.  Hence, on $I_\mu$, the viscous unstable
count is zero on the increasing side of $M_\kappa$ and one on its decreasing
side.  For every sufficiently small $\kappa\ne0$, the corresponding
inviscid star remains spectrally stable on a nonempty interval beyond the
fixed-$J$ mass maximum.  The nonhomology condition is automatic because
$\gamma_0=5/3>4/3$.
\end{example}

Pure polytropes with $\gamma_0=5/3$ satisfy the analytic and slow-branch
parts of the theory, but they are not turning-point examples:
$M_0(\mu)=C\mu^{(3\gamma_0-4)/2}$ is monotone.  Likewise, the
differentially rotating pure-polytropic example in
\cite[Example~2, p.~1754]{LinWang2023} fixes a shellwise angular-momentum
distribution and belongs to the inviscid problem, not to the rigidly
rotating fixed-$J$ family considered here.

\subsection{Second variation at fixed \texorpdfstring{$M$ and $J$}{M and J}}

Fix any regular rigidly rotating equilibrium satisfying
Hypothesis~\ref{hyp:EOS}, and write \(I=I(\bar\rho)\).
The gravitational convention $V_q=-\abs{x}^{-1}*q$ gives
\begin{equation}
 Q_0[q]=\int h'(\bar\rho)|q|^2dx+\int V_q\overline q\,dx,
 \qquad Q_J[q]=Q_0[q]+\frac{\omega^2}{I}
                    \left|\int r^2qdx\right|^2.       \label{eq:Q0-QJ}
\end{equation}

\begin{lemma}
\label{lem:second-variation}
Let $\mathscr F_{c,\omega}=E+cM-\omega J$.  At
$(\bar\rho,\omega r e_\theta)$,
\begin{equation}
 n^-\!\left(D^2\mathscr F_{c,\omega}
      \big|_{\{\delta M=\delta J=0\}_{\rm ax}}\right)
 =n^-\!\left(Q_J\big|_{X_0}\right).                 \label{eq:hessian-index}
\end{equation}
\end{lemma}
\begin{proof}
Recall that
\begin{equation}
 E(\rho,v)=\int_{\mathbb R^3}
 \left\{\frac12\rho|v|^2+\Phi(\rho)
                 +\frac12\rho V_\rho\right\}dx,
 \qquad \Phi''(\rho)=h'(\rho).                          \label{eq:energy-for-Hessian}
\end{equation}
At $(\bar\rho,\omega r e_\theta)$, the mixed terms in $D^2E$ and
$-\omega D^2J$ cancel.  Therefore, for a real perturbation $(q,w)$,
\begin{equation}
 D^2\mathscr F_{c,\omega}[(q,w)]
 =Q_0[q]+\|w\|_H^2.                                      \label{eq:unreduced-Hessian}
\end{equation}
Indeed, the mixed kinetic term is
$2\omega\int qrw_\theta\,dx$, while
$D^2J[(q,w)]=2\int qrw_\theta\,dx$.

The linearized constraints are
\begin{equation}
 \int q\,dx=0,
 \qquad \omega\int r^2q\,dx+(w,q_{\rm rot})_H=0.                         \label{eq:linearized-MJ-constraints}
\end{equation}
For $q\in X_0$, set
\begin{equation}
 \beta(q)=\omega\int r^2q\,dx,
 \qquad w_q=-\frac{\beta(q)}I q_{\rm rot}.
\end{equation}
Every velocity satisfying the second constraint in
\eqref{eq:linearized-MJ-constraints} has a unique decomposition
\begin{equation}
 w=w_q+w_\perp,
 \qquad (w_\perp,q_{\rm rot})_H=0.                       \label{eq:Hessian-square-completion}
\end{equation}
Since $I=\|q_{\rm rot}\|_H^2$ and $w_q\perp_Hw_\perp$,
\begin{align}
 D^2\mathscr F_{c,\omega}[(q,w)]
 &=Q_0[q]+\frac{\omega^2}{I}
       \left|\int r^2q\,dx\right|^2+\|w_\perp\|_H^2 \notag\\
 &=Q_J[q]+\|w_\perp\|_H^2.                              \label{eq:exact-Hessian-splitting}
\end{align}
The map $(q,w)\mapsto(q,w_\perp)$ is a bounded linear isomorphism from the
axisymmetric space $\{\delta M=\delta J=0\}_{\rm ax}$ onto
$X_0\times\{q_{\rm rot}\}^{\perp_H}$; boundedness follows from
Cauchy--Schwarz in $X$ for the functional $q\mapsto\int r^2q\,dx$.
Invariance of inertia under bounded linear isomorphisms and
\eqref{eq:exact-Hessian-splitting} prove \eqref{eq:hessian-index}.  The
complex statement is obtained by Hermitian complexification.
\end{proof}

\subsection{The branch tangent}
\begin{lemma}
\label{lem:branch-tangent}
Let
\begin{equation}
 f_{\mu,\kappa}:=\partial_\mu\widetilde\rho_{\mu,\kappa},
                                                               \label{eq:fixed-J-Eulerian-tangent}
\end{equation}
where the derivative of the Eulerian zero extension is taken at fixed
physical position.  Differentiation along the branch at fixed $J=\kappa$
gives
\begin{equation}
 L_Jf_{\mu,\kappa}
 =-(\partial_\mu c_{\mu,\kappa})\,\mathbf 1,
 \qquad
 Q_J[f_{\mu,\kappa}]
 =-(\partial_\mu c_{\mu,\kappa})
      \partial_\mu M_\kappa(\mu),                         \label{eq:tangent-identity}
\end{equation}
where $X^*$ is the anti-dual and $\mathbf1\in X^*$ is defined by
$\langle\mathbf1,q\rangle=\int\overline q\,dx$.
\end{lemma}

\begin{proof}
At fixed $\kappa$, the equilibrium equation is
\begin{equation}
 h(\rho_{\mu,\kappa})+V_{\rho_{\mu,\kappa}}
 -\frac12\omega(\mu,\kappa)^2r^2+c_{\mu,\kappa}=0.
                                                               \label{eq:equilibrium-fixedJ}
\end{equation}
The zero-extension regularity in
Proposition~\ref{prop:slow-branch-full-range} permits
differentiation at fixed $x$; there is no boundary measure because the
density vanishes continuously on the moving boundary.  Thus, with
$f=f_{\mu,\kappa}$,
\begin{equation}
 h'(\bar\rho)f+V_f-\omega(\partial_\mu\omega)r^2
      +\partial_\mu c_{\mu,\kappa}=0.                    \label{eq:equilibrium-differentiated}
\end{equation}
Differentiating the fixed-angular-momentum identity
$\omega I(\rho_{\mu,\kappa})=\kappa$ gives
\begin{equation}
 I\,\partial_\mu\omega
 +\omega\int r^2f\,dx=0,
 \qquad
 \partial_\mu\omega=-\frac{\omega}{I}\int r^2f\,dx.      \label{eq:omega-mu}
\end{equation}
Substitution into \eqref{eq:equilibrium-differentiated} yields the first
identity in \eqref{eq:tangent-identity}.  Pairing it with $f$ gives the
second, since the zero-extension derivative and the boundary behavior imply
\begin{equation}
 \partial_\mu M_\kappa(\mu)=\int_{\mathbb R^3}f_{\mu,\kappa}\,dx.
\end{equation}
\end{proof}

\begin{lemma}
\label{lem:c-positive}
On every compact interval $I_\mu\Subset(0,\mu_e)$ there are
$c_I>0$ and $\kappa_I>0$ such that
\begin{equation}
  \partial_\mu c_{\mu,\kappa}\ge c_I
  \qquad(\mu\in I_\mu,\ |\kappa|<\kappa_I).
  \label{eq:c-positive}
\end{equation}
Combined with \eqref{eq:LJ-one-negative} and
Lemma~\ref{lem:one-constraint-inertia}, this positivity fixes the
orientation of the constrained-index crossing.
\end{lemma}

\begin{proof}
The low-density orientation \eqref{eq:low-density-c-scaling} and the
definition of $\mu_e$ give
$\partial_\mu c_{\mu,0}>0$ on $(0,\mu_e)$.  Compactness therefore
gives $\partial_\mu c_{\mu,0}\ge2c_I>0$ on $I_\mu$.
At fixed $\kappa$, the chain rule, the $C^2$ scalar dependence from
Proposition~\ref{prop:slow-branch-full-range}, and
\eqref{eq:omega-mu-estimate} imply
\begin{align*}
 \partial_\mu c_{\mu,\kappa}
 &=\partial_\mu c(\mu,s)
   +2\omega\partial_\mu\omega\,\partial_s c(\mu,s)\\
 &=\partial_\mu c(\mu,0)+O(\kappa^2),\qquad s=\omega^2,
\end{align*}
uniformly on $I_\mu$.  For sufficiently small $|\kappa|$, the error is
at most $c_I$, proving \eqref{eq:c-positive}.
\end{proof}

For the slowly rotating interval before the first radial Hessian degeneracy,
define
\[
 X_{\rm ev}=\{q\in X_{\bar\rho}:q(x,y,-z)=q(x,y,z)\},\qquad
 X_{\rm odd}=\{q\in X_{\bar\rho}:q(x,y,-z)=-q(x,y,z)\}.
\]
The spherical spectral theorem and form-norm perturbation give
\begin{equation}
 \begin{aligned}
 n^-(\widehat L_J|_{X_{\rm ev}})&=1,
 &\Ker(\widehat L_J|_{X_{\rm ev}})&=\{0\},\\
 \widehat L_J|_{X_{\rm odd}}&\ge0,
 &\Ker(\widehat L_J|_{X_{\rm odd}})&=\Span\{\tau\},
 \qquad \tau=-\partial_z\bar\rho.
 \end{aligned}                                             \label{eq:LJ-one-negative}
\end{equation}
and $\partial_\mu c_{\mu,\kappa}>0$ by Lemma~\ref{lem:c-positive}.  The even and odd
$z$-parity spaces reduce $L_J$;
$\tau$ is odd while $f_{\mu,\kappa}$ is even.  A proof of the perturbative
spectral assertion is included in Proposition~\ref{prop:density-spectrum}.

\begin{lemma}
\label{lem:one-constraint-inertia}
Under \eqref{eq:LJ-one-negative} and
$\partial_\mu c_{\mu,\kappa}>0$,
\begin{equation}
  n^-(Q_J|_{X_0})=
  \begin{cases}
    0,&\partial_\mu M_\kappa(\mu)\ge0,\\
    1,&\partial_\mu M_\kappa(\mu)<0.
  \end{cases}
  \label{eq:constraint-index}
\end{equation}
If $\partial_\mu M_\kappa(\mu)\ne0$, then
$\ker(Q_J|_{X_0})=\operatorname{Span}\{\tau\}$; at a turning point,
\begin{equation}
  \Ker(Q_J|_{X_0})
  =\operatorname{span}\{\tau,f_{\mu,\kappa}\}.
  \label{eq:turning-kernel}
\end{equation}
\end{lemma}

\begin{proof}
Let $a(u,v)=\langle L_Ju,v\rangle$, linear in $u$ and conjugate-linear
in $v$. The even and odd subspaces are invariant because $\bar\rho$ is
even in $z$. By \eqref{eq:LJ-one-negative}, the odd restriction is nonnegative
with kernel $\Span\{\tau\}$, while the even restriction is invertible
and has one negative direction.

Let
\[
  f=f_{\mu,\kappa}=\partial_\mu\rho_{\mu,\kappa}.
\]
Since the equilibrium is real and even in $z$, $f$ is real and belongs to
$X_{\mathrm{ev}}$; center normalization gives $f(0)=1$, so $f\ne0$.
By Lemma~\ref{lem:branch-tangent},
\[
  L_Jf=-(\partial_\mu c_{\mu,\kappa})\,\mathbf1.
\]
Pairing with any $q\in X_{\mathrm{ev}}$ gives
\[
  a(f,q)
  =-\partial_\mu c_{\mu,\kappa}\int\overline q\,dx.
\]
In particular, for $q\in X_{0,\mathrm{ev}}:=X_0\cap X_{\mathrm{ev}}$,
we have $\int q\,dx=0$, so
\begin{equation}
  a(f,q)=0\qquad\text{for all }q\in X_{0,\mathrm{ev}}.
  \label{eq:f-orthogonal-to-X0}
\end{equation}
Also,
\begin{equation}
  a(f,f)
  =-\partial_\mu c_{\mu,\kappa}
   \int f\,dx
  =-\partial_\mu c_{\mu,\kappa}
   \partial_\mu M_\kappa(\mu).
  \label{eq:a-ff}
\end{equation}

Assume first that $\partial_\mu M_\kappa(\mu)\ne0$.  Then $\int f\,dx\ne0$, so
$f\notin X_{0,\mathrm{ev}}$.  By \eqref{eq:f-orthogonal-to-X0},
$a(f,f)\ne0$, and we have the exact $a$-orthogonal decomposition
\begin{equation}
 X_{\mathrm{ev}}=\Span\{f\}\oplus_a X_{0,\mathrm{ev}}.     \label{eq:even-a-decomposition}
\end{equation}
Indeed, for $q\in X_{\mathrm{ev}}$ subtract
$(\int q/\int f)f$ to obtain its unique mass-zero component.
Sylvester additivity in \eqref{eq:even-a-decomposition} gives
\begin{equation}
 1=n^-(a|_{X_{\mathrm{ev}}})
  =n^-(a|_{\Span\{f\}})+n^-(a|_{X_{0,\mathrm{ev}}}).       \label{eq:Sylvester-one-constraint}
\end{equation}
By \eqref{eq:a-ff}, $a(f,f)>0$ when
$\partial_\mu M_\kappa<0$ and $a(f,f)<0$ when
$\partial_\mu M_\kappa>0$.  Equation
\eqref{eq:Sylvester-one-constraint}, together with nonnegativity in the odd
sector, proves \eqref{eq:constraint-index} in the noncritical case.

For the kernel, let $q\in X_{0,\mathrm{ev}}$ satisfy
$a(q,y)=0$ for all $y\in X_{0,\mathrm{ev}}$.  The mass functional has
codimension one, so there is an $\alpha\in\C$ with
$L_Jq=\alpha\mathbf1$.  Since $L_J$ is invertible on $X_{\mathrm{ev}}$ and
$L_Jf=-(\partial_\mu c_{\mu,\kappa})\mathbf1$,
\begin{equation}
 q=-\frac{\alpha}{\partial_\mu c_{\mu,\kappa}}f.
\end{equation}
Integrating and using $\int q=0$ and
$\int f=\partial_\mu M_\kappa\ne0$ gives $\alpha=0$, hence $q=0$.
The odd kernel is $\Span\{\tau\}$, so
$\Ker(Q_J|_{X_0})=\Span\{\tau\}$.

Now suppose $\partial_\mu M_\kappa(\mu)=0$.  Then $\int f\,dx=0$, so
$f\in X_{0,\mathrm{ev}}$.  By \eqref{eq:f-orthogonal-to-X0}, we have
\[
  a(f,q)=0\qquad\text{for all }q\in X_{0,\mathrm{ev}}.
\]
In particular $a(f,f)=0$ by \eqref{eq:a-ff}.  We claim that
$X_{0,\mathrm{ev}}$ has no negative direction.  Suppose, to the contrary,
that there exists $q\in X_{0,\mathrm{ev}}$ with $a(q,q)<0$.  Since the
form $a$ is nondegenerate on $X_{\mathrm{ev}}$ and $f\neq0$, there exists
$g\in X_{\mathrm{ev}}$ such that $a(f,g)\ne0$.  Define
\[
  h=g-\frac{a(g,q)}{a(q,q)}\,q.
\]
Then
\[
  a(f,h)=a(f,g)-\frac{\overline{a(g,q)}}{a(q,q)}\,a(f,q)
  =a(f,g)\ne0,
\]
and
\[
  a(h,q)=a(g,q)-\frac{a(g,q)}{a(q,q)}\,a(q,q)=0.
\]
On the two-dimensional subspace $\operatorname{span}\{f,h\}$, the matrix
of $a$ is
\[
  \begin{pmatrix}
    a(f,f) & a(f,h)\\
    a(h,f) & a(h,h)
  \end{pmatrix}
  =
  \begin{pmatrix}
    0 & a(f,g)\\
    a(g,f) & a(h,h)
  \end{pmatrix}.
\]
Since $a(f,h)=a(f,g)\ne0$, this matrix has determinant
$-|a(f,g)|^2<0$.  Hence, $\operatorname{span}\{f,h\}$ contains exactly one
negative direction.  Moreover, $q$ is $a$-orthogonal to both $f$ and $h$,
and $a(q,q)<0$, so $q$ is a second, independent negative direction in
$X_{\mathrm{ev}}$.  This contradicts the assumption that $X_{\mathrm{ev}}$
has only one negative direction.  Therefore $X_{0,\mathrm{ev}}$ is
nonnegative.

It remains to describe the kernel at the turning point.  Let
$q\in X_{0,\mathrm{ev}}$ belong to the kernel of $Q_J|_{X_0}$.  Since the
form is nonnegative on $X_{0,\mathrm{ev}}$, the Cauchy--Schwarz inequality
implies that $q$ is in the kernel if and only if
\[
  a(q,y)=0\qquad\text{for all }y\in X_{0,\mathrm{ev}}.
\]
The anti-dual annihilator of the mass-zero subspace is the span of
$\mathbf1$.  Hence there is $\alpha\in\C$ such that
\[
  L_J q=\alpha\,\mathbf1.
\]
But Lemma~\ref{lem:branch-tangent} gives
\[
  L_J f=-\partial_\mu c_{\mu,\kappa}\,\mathbf1.
\]
Since $\partial_\mu c_{\mu,\kappa}>0$ and $L_J$ is
invertible on $X_{\mathrm{ev}}$, we deduce
\[
  q=-\frac{\alpha}{\partial_\mu c_{\mu,\kappa}}\,f.
\]
Thus the even part of the kernel is one-dimensional and spanned by
$f=f_{\mu,\kappa}$.  Adding the odd axial translation $\tau$ gives
\[
  \Ker(Q_J|_{X_0})
  =\operatorname{span}\{\tau,f_{\mu,\kappa}\}.
\]
This proves \eqref{eq:turning-kernel} and completes the lemma.
\end{proof}

\section{Turning point on a fixed-\texorpdfstring{$J$}{J} branch}
\label{sec:turning-proof}

We now verify the density spectral input and prove
Theorem~\ref{thm:turning-point}.  The interval is chosen before the first
radial zero mode of the spherical density Hessian; on this interval its
negative index is one.

\begin{proposition}
\label{prop:density-spectrum}
Assume Hypothesis~\ref{hyp:EOS}, let
$I_\mu\Subset(0,\mu_e)$, and let $\rho_{\mu,\kappa}$ be the
fixed-$J$ branch of Proposition~\ref{prop:fixed-J-branch}.  There are
$\kappa_I>0$ and $\delta_I>0$ such that, for
$\mu\in I_\mu$ and $|\kappa|<\kappa_I$,
\begin{equation}
 n^-(\widehat L_J(\mu,\kappa))=1,\qquad
 \Ker\widehat L_J(\mu,\kappa)
   =\Span\{\partial_z\rho_{\mu,\kappa}\}.                  \label{eq:slow-density-spectrum}
\end{equation}
The negative spectral space is one-dimensional and even in $z$, while the
kernel is odd.  If $E^-_{\mu,\kappa}$ is the negative spectral space, then
\begin{equation}
 \langle L_J(\mu,\kappa)q,q\rangle
 \ge\delta_I\|q\|_{X_{\mu,\kappa}}^2                       \label{eq:slow-density-gap}
\end{equation}
whenever
$q\perp_{X_{\mu,\kappa}}
 (E^-_{\mu,\kappa}\oplus
  \Span\{\partial_z\rho_{\mu,\kappa}\})$.
All constants are uniform on $I_\mu$.
\end{proposition}

\begin{proof}
Choose a ball $D$ containing all supports and set
$\mathcal H=L^2_{\rm ax}(D)$.  Put
\[
 b_{\mu,\kappa}
 =\bigl[\mathscr G'(q_{\mu,\omega(\mu,\kappa)^2})\bigr]^{1/2}
 =\mathbf1_{\Omega_{\mu,\kappa}}
      h'(\rho_{\mu,\kappa})^{-1/2},
 \qquad j_{\mu,\kappa}=r^2b_{\mu,\kappa}.
\]
The isometry $\sigma\mapsto h'(\rho_{\mu,\kappa})^{1/2}\sigma$, followed
by zero extension, identifies $X_{\mu,\kappa}$ with a closed subspace of
$\mathcal H$.  Extending by the identity on its orthogonal complement, the
Riesz Hessian becomes
\[
 \mathbb L_{J,\mu,\kappa}
 =I+\mathbb G_{\mu,\kappa}
  +\frac{\omega(\mu,\kappa)^2}{I(\rho_{\mu,\kappa})}
       (j_{\mu,\kappa}\otimes j_{\mu,\kappa}),
\]
where $\mathbb G_{\mu,\kappa}$ has kernel
$-b_{\mu,\kappa}(x)b_{\mu,\kappa}(y)/|x-y|$.
The fixed-domain branch convergence and $\mathscr G'(0)=0$ give
$b_{\mu,\kappa}\to b_{\mu,0}$ uniformly in $D$, uniformly for
$\mu\in I_\mu$.  Since $|x-y|^{-2}\in L^1(D\times D)$, the gravitational
kernels converge in Hilbert--Schmidt norm.  Also
$j_{\mu,\kappa}\to j_{\mu,0}$ in $\mathcal H$, while
$\omega^2/I=O(\kappa^2)$.  Consequently
\begin{equation}
 \sup_{\mu\in I_\mu}
 \big\|\mathbb L_{J,\mu,\kappa}-\mathbb L_{0,\mu,0}\big\|
 \longrightarrow0.                                        \label{eq:LJ-slow-norm-convergence}
\end{equation}
Choose uniform contours around the spherical negative eigenvalue and zero
using Lemma~\ref{lem:spherical-density-spectrum}.  Norm stability of their
Riesz projections and \eqref{eq:LJ-slow-norm-convergence} preserve the
rank one of both clusters and the uniform complementary gap.  Parity
commutes with the branch operators, so the negative cluster remains even.
Axial translation is an exact zero mode: differentiating the equilibrium
equation in $z$ gives
\[
 L_0(\mu,\kappa)\partial_z\rho_{\mu,\kappa}=0,\qquad
 \int r^2\partial_z\rho_{\mu,\kappa}\,dx=0,
\]
and hence
$L_J(\mu,\kappa)\partial_z\rho_{\mu,\kappa}=0$.
This nonzero odd vector belongs to the one-dimensional zero cluster and
therefore spans it.  The projection estimates yield
\eqref{eq:slow-density-gap}.
\end{proof}

\begin{proof}[Proof of Theorem~\ref{thm:turning-point}]
Proposition~\ref{prop:mass-maximum} constructs the fixed-$J$ mass maximum
$\mu_*^\kappa$ in $I_\mu$.  Proposition~\ref{prop:density-spectrum} supplies
\eqref{eq:LJ-one-negative}, while Lemma~\ref{lem:branch-tangent} supplies
the exact crossing identity
\begin{equation}
 Q_J[\partial_\mu\rho_{\mu,\kappa}]
 =-(\partial_\mu c_{\mu,\kappa})\,\partial_\mu M_\kappa(\mu).
\end{equation}
The one-constraint calculation in
Lemma~\ref{lem:one-constraint-inertia} therefore gives
\eqref{eq:constraint-index}.  Applying
Theorem~\ref{thm:main-index} proves \eqref{eq:turning-count} for the full
axisymmetric generator, with no restriction placed on its phase space.

At $\mu=\mu_*^\kappa$, the constrained density form is nonnegative and has
kernel \eqref{eq:turning-kernel}.  The abstract KTC theorem together with
Propositions~\ref{prop:root-chain-equivalence} and
\ref{prop:full-spectral-transfer} excludes all open-right-half-plane
spectrum and gives $N_+=0$ at the turning point.  The
branch tangent is the neutral direction that crosses there; axial
translation remains the exact symmetry kernel.  Since
$\partial_\mu^2 M_\kappa(\mu_*^\kappa)<0$,
$\partial_\mu M_\kappa$ changes from positive to negative,
so the unstable count changes from zero to one and no earlier transition is
possible on the stated interval.
\end{proof}

\begin{remark}
At the turning point the full generator need not have bounded center
dynamics: translation already produces the chain
\eqref{eq:translation-chain}, and the branch tangent may create an
additional generalized zero structure.  The theorem asserts absence of right-half-plane spectrum at the critical
parameter and gives its algebraic count on either side.
\end{remark}

\section{Viscosity versus inviscid rotation}
\label{sec:comparison}

We first record the support geometry needed for the inviscid Rayleigh form.

\begin{proposition}
\label{prop:slow-support-geometry}
Let $I_\mu\Subset(0,\mu_e)$ and consider the branch of
Proposition~\ref{prop:slow-branch-full-range}; put
$\alpha=1/(\gamma_0-1)$.  Uniformly for
$\mu\in I_\mu$ and sufficiently small $s\ge0$, the support
$\Omega_{\mu,s}$ is axisymmetric, $z$-even, strictly convex, and of class
$C^{3,\eta}$ for every $0<\eta<\min\{1,\alpha-1\}$ (only $C^2$ is used
below).  Its meridional section intersects the rotation axis in an
interval and can be written
\[
 \Omega_{\mu,s}
 =\{(r,\theta,z):0\le r<R_{\mu,s},\ |z|<Z_{\mu,s}(r)\}.
\]
There are $z_0,c_0,C_0>0$, uniform on the compact branch, such that near
the equatorial endpoint the boundary can be written
$r=\varphi_{\mu,s}(z)$, where
\[
 \varphi_{\mu,s}(0)=R_{\mu,s},\qquad
 \varphi'_{\mu,s}(0)=0,\qquad
 c_0\le-\varphi''_{\mu,s}(z)\le C_0\quad(|z|<z_0).
\]
If $Z_{\mu,s}(r)>0$ is determined by
$\varphi_{\mu,s}(Z_{\mu,s}(r))=r$, then, uniformly on the branch,
\begin{equation}
 Z_{\mu,s}(r)\asymp(R_{\mu,s}-r)^{1/2},\qquad
 d_{\mu,s}(r,z)\asymp
 Z_{\mu,s}(r)^2-z^2,\quad |z|<Z_{\mu,s}(r),\quad
 r\uparrow R_{\mu,s}.                                    \label{eq:slow-equator-geometry}
\end{equation}
Consequently all geometric assumptions in the rigid-rotation specialization
of the inviscid index theorem \cite{LinWang2023} are satisfied.
\end{proposition}

\begin{proof}
Proposition~\ref{prop:slow-vacuum-tangents} gives $C^{3,\eta}$ convergence
of the enthalpy and a uniformly transverse zero level.  Hence the free
surfaces form a $C^{3,\eta}$-small deformation of the spherical boundaries.
Their second fundamental forms remain uniformly positive, proving strict
convexity.  Axisymmetry and reflection symmetry are inherited from the
construction.  In a uniform equatorial chart write the meridional boundary
as $r=\varphi_{\mu,s}(z)$.  Integrating the uniform curvature bounds once
and twice gives
\[
 c_1t\le-\varphi'_{\mu,s}(t)\le c_2t,\qquad
 \frac{c_1}{2}t^2\le R_{\mu,s}-\varphi_{\mu,s}(t)
 \le\frac{c_2}{2}t^2
\]
for $0\le t\le z_0$.  Thus
$Z_{\mu,s}(r)^2\asymp R_{\mu,s}-r$.  Moreover, for
$0\le z\le Z:=Z_{\mu,s}(r)$,
\[
 \varphi_{\mu,s}(z)-r
 =\int_z^Z-\varphi'_{\mu,s}(t)\,dt\asymp Z^2-z^2,
\]
and the same holds for negative $z$ by reflection symmetry.  Since
$r-\varphi_{\mu,s}(z)$ is a uniformly regular defining function, its
absolute value is comparable with Euclidean distance to the boundary.
This proves \eqref{eq:slow-equator-geometry}.
\end{proof}

\subsection{Comparison of the constrained forms}

For an axisymmetric density perturbation $\sigma$, extended by zero outside
the star, define
\begin{equation}
 \Sigma_\sigma(r)=\int_{-\infty}^{\infty}\sigma(r,z)dz,\qquad
 m_\sigma(r)=\int_0^r s\Sigma_\sigma(s)ds.                              \label{eq:column-mass}
\end{equation}
Then $m_{\bar\rho}'(r)=r\Sigma_{\bar\rho}(r)$ and, for $\sigma\in X_0$,
$m_\sigma(R)=0$, where $R$ is the equatorial support radius.

For rigid rotation the reduced inviscid quadratic form \cite{LinWang2023} is
\begin{equation}
 \widetilde Q_{\rm EP}[\sigma]
 =Q_0[\sigma]+8\pi\omega^2
       \int_0^R\frac{|m_\sigma(r)|^2}{m_{\bar\rho}'(r)}dr.                  \label{eq:EP-reduced-form}
\end{equation}
This is exactly the rigid-rotation specialization of
\cite[Theorem~1.1 and Corollary~1.1]{LinWang2023}: their Rayleigh
discriminant is $4\omega^2$, which produces the coefficient $8\pi\omega^2$
in \eqref{eq:EP-reduced-form}.  Propositions~\ref{prop:slow-vacuum-tangents}
and \ref{prop:slow-support-geometry} supply their support and
equatorial-curvature
conditions; for
$\kappa\ne0$ the strict Rayleigh condition is automatic.  Their reduced
density space and mass constraint identify with the present $X_0$ after the
Riesz normalization \eqref{eq:LJ-Riesz}.  Here and in
Theorem~\ref{thm:strict-separation}, inviscid spectral stability means
absence of spectrum in $\{\Ree\lambda>0\}$; no boundedness of center dynamics
is asserted.
The extra term represents preservation of the material distribution of
specific angular momentum.  Viscosity removes those infinitely many constraints
and leaves only the total-$J$ rank-one term in $Q_J$.

\begin{lemma}
\label{lem:column-endpoints}
Let $\alpha=1/(\gamma_0-1)>1$ and set
\begin{equation}
 H_{\bar\rho}(r)=\int_{\{z:(r,z)\in\Omega_{\bar\rho}\}}
                  \frac{dz}{h'(\bar\rho(r,z))}.             \label{eq:column-H}
\end{equation}
Uniformly on a compact branch interval,
\begin{align}
 m_{\bar\rho}'(r)&\asymp r &&(r\downarrow0),               \label{eq:axis-column}\\
 m_{\bar\rho}'(r)&\asymp(R-r)^{\alpha+1/2},
 &H_{\bar\rho}(r)&\asymp(R-r)^{\alpha-1/2}
 &&(r\uparrow R).                                           \label{eq:equator-columns}
\end{align}
Consequently the Rayleigh form
\begin{equation}
 \mathcal R_{\bar\rho}[\sigma]
 :=\int_0^R\frac{|m_\sigma(r)|^2}{m_{\bar\rho}'(r)}\,dr       \label{eq:Rayleigh-form}
\end{equation}
is bounded on $X_0$, and its Riesz operator on $X_0$ is compact.
\end{lemma}

\begin{proof}
By Proposition~\ref{prop:slow-support-geometry}, in a fixed equatorial
chart $d(r,z)\asymp Z(r)^2-z^2$ and
$Z(r)\asymp(R-r)^{1/2}$.  The factorization
\eqref{eq:uniform-vacuum-factor} and
$1/h'(\bar\rho)\asymp d^{\alpha-1}$ give
\begin{align*}
 \Sigma_{\bar\rho}(r)&\asymp
 \int_{-Z(r)}^{Z(r)}(Z(r)^2-z^2)^\alpha dz
 =C_\alpha Z(r)^{2\alpha+1}
 \asymp(R-r)^{\alpha+1/2},\\
 H_{\bar\rho}(r)&\asymp
 \int_{-Z(r)}^{Z(r)}(Z(r)^2-z^2)^{\alpha-1}dz
 =C_{\alpha-1}Z(r)^{2\alpha-1}
 \asymp(R-r)^{\alpha-1/2}.
\end{align*}
Since $r\asymp R$ there, this proves \eqref{eq:equator-columns};
\eqref{eq:axis-column} follows from smoothness and positivity of the column
density through the axis.

Put $E_\sigma(r)=\int h'(\bar\rho)|\sigma|^2dz$.  Cauchy--Schwarz gives
\begin{equation}
 |\Sigma_\sigma(r)|^2\le E_\sigma(r)H_{\bar\rho}(r),\qquad
 2\pi\int_0^R rE_\sigma(r)dr=\|\sigma\|_X^2.              \label{eq:column-CS}
\end{equation}
For $r<\delta$, Cauchy--Schwarz in the radial variable gives
\begin{align}
 |m_\sigma(r)|^2
 &\le \left(\int_0^r sE_\sigma(s)\,ds\right)
       \left(\int_0^r sH_{\bar\rho}(s)\,ds\right) \notag\\
 &\le Cr^2\|\sigma\|_X^2.                                \label{eq:axis-column-bound}
\end{align}
Together with $m_{\bar\rho}'(r)\asymp r$, this gives
\begin{equation}
 \int_0^\delta\frac{|m_\sigma|^2}{m_{\bar\rho}'}dr
 \le C\delta^2\|\sigma\|_X^2.                             \label{eq:Rayleigh-axis-tail}
\end{equation}
For $r>R-\delta$, mass zero gives
$m_\sigma(r)=-\int_r^R s\Sigma_\sigma(s)ds$.  Equations
\eqref{eq:column-CS} and \eqref{eq:equator-columns} give
\begin{equation}
 |m_\sigma(r)|^2\le
 C(R-r)^{\alpha+1/2}\|\sigma\|_X^2,\qquad
 \int_{R-\delta}^R\frac{|m_\sigma|^2}{m_{\bar\rho}'}dr
 \le C\delta\|\sigma\|_X^2.                              \label{eq:Rayleigh-boundary-tail}
\end{equation}
On $[\delta,R-\delta]$ define
$T_\delta\sigma=m_\sigma|_{[\delta,R-\delta]}$.  Since
$m_\sigma'=r\Sigma_\sigma$, estimate \eqref{eq:column-CS} shows that
$T_\delta:X_0\to H^1([\delta,R-\delta])$ is bounded.  Rellich compactness
makes it compact into $L^2$.  Multiplication by
$(m_{\bar\rho}')^{-1/2}$ is bounded there, so the truncated Rayleigh
operator is $K_\delta^*K_\delta$, with $K_\delta$ compact.  The two tail
estimates are operator-norm bounds for the positive endpoint forms.  Hence
the full Riesz operator is the norm limit of compact interior truncations.
\end{proof}

\begin{lemma}
\label{lem:comparison}
For $\sigma\in X_0$,
\begin{align}
 \Delta Q[\sigma]
 &:={\widetilde Q}_{\rm EP}[\sigma]-Q_J[\sigma] \notag\\
 &=8\pi\omega^2\left[
 \int_0^R\frac{|m_\sigma(r)|^2}{m_{\bar\rho}'(r)}dr
 -\frac{\left|\int_0^R r m_\sigma(r)dr\right|^2}
        {\int_0^R r^2m_{\bar\rho}'(r)dr}\right]\ge0.                     \label{eq:DeltaQ}
\end{align}
If $\omega\ne0$, equality holds precisely when
\begin{equation}
 m_\sigma(r)=C r m_{\bar\rho}'(r)\quad\text{for a.e. }r\in(0,R).          \label{eq:Delta-equality}
\end{equation}
Equivalently,
\begin{equation}
 \Sigma_\sigma=C\bigl(2\Sigma_{\bar\rho}
                         +r\partial_r\Sigma_{\bar\rho}\bigr)
 \quad\hbox{in }\mathcal D'(0,R).                         \label{eq:column-equality}
\end{equation}
For complex perturbations $C\in\C$; on the real space $C\in\mathbb R$.
\end{lemma}

\begin{proof}
Integration by parts, using $m_\sigma(R)=0$, gives
\begin{equation}
 \int r^2\sigma dx
 =2\pi\int_0^Rr^2m_\sigma'(r)dr
 =-4\pi\int_0^Rr m_\sigma(r)dr,                                         \label{eq:Q-column}
\end{equation}
whereas
\begin{equation}
 I=2\pi\int_0^Rr^2m_{\bar\rho}'(r)dr.                                   \label{eq:I-column}
\end{equation}
Substitution into the rank-one term of $Q_J$ yields
\eqref{eq:DeltaQ}.  Weighted Cauchy--Schwarz applied to
$m_\sigma/\sqrt{m_{\bar\rho}'}$ and
$r\sqrt{m_{\bar\rho}'}$ proves nonnegativity.  Equality in
Cauchy--Schwarz is exactly \eqref{eq:Delta-equality}; distributional
differentiation gives \eqref{eq:column-equality}.
\end{proof}

Thus
\begin{equation}
 \widetilde Q_{\rm EP}\ge Q_J.                                          \label{eq:EP-ge-viscous}
\end{equation}
In particular,
\[
 n^-\!\left(\widetilde Q_{\rm EP}|_{X_0}\right)
 \le n^-\!\left(Q_J|_{X_0}\right).
\]
By Theorem~\ref{thm:main-index} and the Euler--Poisson index theorem of
\cite{LinWang2023}, the viscous unstable index is therefore no smaller than
the inviscid one for the slowly rotating branch considered here.
Variationally, this reflects the additional angular-momentum constraints of
the inviscid dynamics.  The strict inequality near the viscous turning point
is proved next.

\begin{lemma}
\label{lem:Abel-injectivity}
Let $w=w(|x|)$ be supported in $[0,R]$ and suppose
\begin{equation}
 \int_0^R|w(s)|s\,ds<\infty.                              \label{eq:Abel-integrability}
\end{equation}
If
\begin{equation}
 \Sigma_w(r)=2\int_r^R\frac{w(s)s}{\sqrt{s^2-r^2}}\,ds=0
 \quad\hbox{for a.e. }r\in(0,R),                          \label{eq:Abel-transform}
\end{equation}
then $w=0$ almost everywhere.
\end{lemma}

\begin{proof}
Set $F(u)=w(\sqrt u)$ and $A(t)=\Sigma_w(\sqrt t)$.  Then
\begin{equation}
 A(t)=\int_t^{R^2}\frac{F(v)}{\sqrt{v-t}}\,dv.             \label{eq:Abel-u}
\end{equation}
The integrability in \eqref{eq:Abel-integrability} permits absolute use of
Fubini's theorem, since for every $u<R^2$,
\begin{equation}
 \int_u^{R^2}\int_t^{R^2}
 \frac{|F(v)|}{\sqrt{(v-t)(t-u)}}\,dv\,dt
 =\pi\int_u^{R^2}|F(v)|\,dv<\infty.                    \label{eq:Abel-absolute-Fubini}
\end{equation}
For almost every $u\in(0,R^2)$,
\begin{align}
 \int_u^{R^2}\frac{A(t)}{\sqrt{t-u}}\,dt
 &=\int_u^{R^2}F(v)
   \left[\int_u^v\frac{dt}{\sqrt{(v-t)(t-u)}}\right]dv \notag\\
 &=\pi\int_u^{R^2}F(v)\,dv.                               \label{eq:Abel-double-transform}
\end{align}
If $A=0$ almost everywhere, the last integral vanishes for almost every
$u$.  Distributional differentiation gives $F=0$, and hence $w=0$ almost
everywhere.  The branch tangent and homology tangent used below satisfy
\eqref{eq:Abel-integrability}, because they are
$O(d^{\alpha-1})$ and $\alpha>1$.
\end{proof}

\subsection{Strictness on the turning tangent}

By Hypothesis~\ref{hyp:EOS},
\begin{equation}
 3P'(s)-h(s)\quad\hbox{is not constant on }(0,\mu_*].       \label{eq:nonhomology}
\end{equation}
Indeed, its leading term at the vacuum is a nonzero multiple of
$s^{\gamma_0-1}$ because $\gamma_0>4/3$.

\begin{proposition}
\label{prop:strict-comparison}
Assume Hypothesis~\ref{hyp:EOS} and the nondegenerate first-mass-maximum
condition \eqref{eq:nondegenerate-maximum}, and let
$\rho_{\mu,\kappa}$ be the fixed-$J$ branch on $I_\mu$ constructed in
Proposition~\ref{prop:fixed-J-branch}.  For every
sufficiently small nonzero $\kappa$,
\begin{equation}
 \Delta Q_{\mu_*^\kappa,\kappa}[f_\kappa]>0.                             \label{eq:strict-on-tangent}
\end{equation}
\end{proposition}

\begin{proof}
Suppose not.  There are $\kappa_n\to0$, $\kappa_n\ne0$, for which equality
holds.  Set
\begin{equation}
 \rho_*:=\rho_{\mu_*},\qquad R_*:=R_0(\mu_*),\qquad
 c_*:=c_{\mu_*,0}.
\end{equation}
Write
\begin{equation}
 \rho_n=\rho_{\mu_*^{\kappa_n},\kappa_n},\qquad
 f_n=\left.\partial_\mu\widetilde\rho_{\mu,\kappa_n}
                  \right|_{\mu=\mu_*^{\kappa_n}}.
\end{equation}
Let $R_n$ be the equatorial radius of $\{\rho_n>0\}$.
Since $\omega(\mu_*^{\kappa_n},\kappa_n)\ne0$, the equality case gives
\begin{equation}
 m_{f_n}(r)=C_n r m_{\rho_n}'(r).                                        \label{eq:equality-sequence}
\end{equation}
By \eqref{eq:branch-parameter-regularity},
\eqref{eq:fixed-J-tangent-chain-rule}, and
\eqref{eq:tangent-convergence}, the zero extensions satisfy
\[
 \|\rho_n-\rho_*\|_{C(\overline D)}
 +\|f_n-f_*\|_{C(\overline D)}\longrightarrow0
\]
on one fixed ball $D$.  Choose $L$ so that $D\subset[-L,L]^3$.
Integrating every column over the same interval gives
\[
 \sup_r|\Sigma_{f_n}(r)-\Sigma_{f_*}(r)|
 \le 2L\|f_n-f_*\|_{C(\overline D)},
\]
and the same estimate holds for the densities.  Multiplication by $r$
and integration from $0$ to $r$ therefore give
\begin{equation}
 m_{f_n}\longrightarrow m_{f_*},\qquad
 m_{\rho_n}'\longrightarrow m_{\rho_*}'
 \quad\hbox{uniformly on }[0,L].                         \label{eq:column-convergence}
\end{equation}
The columns have continuous representatives, so
\eqref{eq:equality-sequence} holds pointwise in the interior.
Choose $r_0\in(0,R_*)$ with $m_{\rho_*}'(r_0)>0$.  Then
\[
 C_n=\frac{m_{f_n}(r_0)}{r_0m_{\rho_n}'(r_0)}
\]
is bounded.  Passing to a subsequence, write $C_n\to C_*$.
Using $C_n\to C_*$ and exhausting $(0,R_*)$ therefore gives
\begin{equation}
 m_{f_*}(r)=C_*r m_{\rho_*}'(r).                                         \label{eq:limit-column-equality}
\end{equation}
If $C_*=0$, then $\Sigma_{f_*}=0$; Lemma~\ref{lem:Abel-injectivity} gives
$f_*=0$, contradicting the normalization
$f_*(0)=1$.

Consider the mass-preserving homology
\begin{equation}
 \rho_a(x)=a^3\rho_*(ax),\qquad
 g=\left.\partial_a\rho_a\right|_{a=1}
   =3\rho_*+x\cdot\nabla\rho_*.                                         \label{eq:homology-tangent}
\end{equation}
For a radial function $w$, its column density is the Abel transform in
\eqref{eq:Abel-transform}, so Lemma~\ref{lem:Abel-injectivity} applies.
Direct scaling gives
\begin{equation}
 m_g(r)=r m_{\rho_*}'(r),\qquad
 \Sigma_g=2\Sigma_{\rho_*}+r\partial_r\Sigma_{\rho_*}.                 \label{eq:homology-column}
\end{equation}
Therefore \eqref{eq:limit-column-equality}, differentiated in
distributions and combined with Lemma~\ref{lem:Abel-injectivity}, implies
\begin{equation}
 f_*=C_*g,\qquad C_*\ne0.                                                \label{eq:f-homology}
\end{equation}

Let $L_*=h'(\rho_*)+V_{(\cdot)}$ be the nonrotating density Hessian.
The branch tangent identity gives
\begin{equation}
 L_*f_*=-\left.\partial_\mu c_{\mu,0}\right|_{\mu=\mu_*}\,\mathbf1,
                                                               \label{eq:limiting-tangent-equation}
\end{equation}
a spatial constant.
Newtonian potential scales by
\begin{equation}
 V_{\rho_a}(x)=aV_{\rho_*}(ax),\qquad
 V_g=V_{\rho_*}+x\cdot\nabla V_{\rho_*}.                                 \label{eq:potential-scaling}
\end{equation}
Using the equilibrium equation and its spatial gradient,
\begin{align}
 L_*g
 &=h'(\rho_*)\bigl(3\rho_*+x\cdot\nabla\rho_*\bigr)
   +V_{\rho_*}+x\cdot\nabla V_{\rho_*}\notag\\
 &=3P'(\rho_*)-h(\rho_*)-c_*.                                           \label{eq:L-homology}
\end{align}
Equations \eqref{eq:f-homology} and \eqref{eq:L-homology} force
$3P'(\rho_*)-h(\rho_*)$ to be spatially constant.  For the spherical star,
\begin{equation}
 \frac{d}{dr}h(\rho_*(r))=-\frac{m_*(r)}{r^2}<0,
 \qquad 0<r<R_*,                                         \label{eq:spherical-enthalpy-monotonicity}
\end{equation}
where $m_*(r)=4\pi\int_0^r\rho_*(s)s^2\,ds$.  Thus $\rho_*$ decreases
continuously from $\mu_*$ to zero and assumes every value in
$(0,\mu_*]$.  It follows that $3P'(s)-h(s)$ is constant on that whole
interval, contradicting \eqref{eq:nonhomology}.

The vacuum limit of that constant is zero.  Hence the equality case would
give $3P'=h$ and, because $h'=P'/\rho$,
\begin{equation}
 3(P')'=\frac{P'}{\rho},\qquad
 P'(\rho)=A\rho^{1/3},\qquad P(\rho)=K\rho^{4/3}.                         \label{eq:four-thirds}
\end{equation}
Thus the $4/3$ homology law is the exact exceptional case.
\end{proof}

\subsection{A delayed inviscid transition}

\begin{proof}[Proof of Theorem~\ref{thm:strict-separation}]
At $\mu=\mu_*^\kappa$, Lemma~\ref{lem:one-constraint-inertia} gives
\begin{equation}
 Q_J\ge0\quad\text{on }X_0,\qquad
 \Ker(Q_J|_{X_0})=\Span\{f_\kappa,\tau_\kappa\},\qquad
 \tau_\kappa=-\partial_z\rho_{\mu_*^\kappa,\kappa}.                      \label{eq:viscous-kernel-at-turn}
\end{equation}
The column density of $\tau_\kappa$ is zero because, on each vertical
section,
\begin{equation}
 \Sigma_{\tau_\kappa}(r)
 =-\int\partial_z\rho_{\mu_*^\kappa,\kappa}(r,z)\,dz=0,     \label{eq:translation-column-zero}
\end{equation}
the density vanishing at both endpoints.  Hence
$m_{\tau_\kappa}=0$.  Denote the Hermitian sesquilinear polarization of
$\Delta Q$ by $\Delta Q[\cdot,\cdot]$; then
$\Delta Q[\tau_\kappa,\sigma]=0$ for every
$\sigma\in X_0$.  If
$\sigma\in\Ker(\widetilde Q_{\rm EP}|_{X_0})$, the nonnegativity of both
$Q_J$ and $\Delta Q$ first gives
$\sigma=af_\kappa+b\tau_\kappa$, and then
\begin{equation}
 0=\Delta Q[\sigma]=|a|^2\Delta Q[f_\kappa].
\end{equation}
Proposition~\ref{prop:strict-comparison} implies $a=0$.  Therefore
\begin{equation}
 \Ker(\widetilde Q_{\rm EP}|_{X_0})=\Span\{\tau_\kappa\}.                 \label{eq:EP-kernel-at-turn}
\end{equation}

Under the fixed-ambient realization in Lemma~\ref{lem:form-continuity}, the inviscid
Riesz operator is $I+K_{\mu,\kappa}$ on a fixed Hilbert space, where
$K_{\mu,\kappa}$ is compact and norm-continuous.  Write $p=(\mu,\kappa)$ and use $W_p,t_p$ from that lemma.  The
enlarged ambient spaces
\[
 Z_p=\{g\in L^2_{\rm ax}(D):(g,W_p)=(g,t_p)=0\}
\]
have norm-continuous orthogonal projections
$\Pi_p-t_p\otimes t_p$, since $t_p\perp W_p$ and $\|t_p\|=1$.
The added exterior directions have eigenvalue $1$.  Kato's near-identity
unitary identifies these enlarged spaces with $Z_{(\mu_*^\kappa,\kappa)}$;
no norm continuity of the moving support projection is used.  On this fixed space,
\eqref{eq:EP-kernel-at-turn} supplies a positive gap.  Operator-norm
continuity preserves half of this gap on
$[\mu_*^\kappa,\mu_*^\kappa+\delta_\kappa]$ for some
$\delta_\kappa>0$.  Translation remains an exact zero mode.  The sharp
Euler--Poisson index theorem of \cite{LinWang2023} then gives zero inviscid
unstable index throughout this interval, proving
Theorem~\ref{thm:strict-separation}.
\end{proof}

\begin{remark}
The inviscid form is coercive modulo axial translation, which remains
a neutral symmetry on the mass-zero space.  The rigid-rotation velocity
$q_{\rm rot}$ lies in the kernel of the viscous form, but for $\omega\ne0$
it is not by itself a zero mode of the full generator.  Its coefficient is
fixed by the total-$J$ constraint and removed by the Routh map; it therefore
does not add a density-kernel vector to \eqref{eq:EP-kernel-at-turn}.
\end{remark}

\begin{remark}
When $\kappa=0$, one has $\omega=0$ and $\Delta Q\equiv0$; the viscous and
inviscid thresholds coincide, as in the nonrotating theory.  For each fixed
small $\kappa\ne0$, the number $\delta_\kappa$ is obtained from the positive
inviscid spectral gap at $\mu_*^\kappa$ and the form-norm modulus in
Lemma~\ref{lem:form-continuity}.  No lower bound uniform as
$\kappa\to0$ is asserted; in particular $\delta_\kappa$ may tend to zero.
\end{remark}

\appendix

\section{Rotating-frame Lagrangian derivation}
\label{app:lagrangian}

\subsection{The rotating-frame equations}

Let $R(t)$ be rotation through angle $\omega t$ about the $z$ axis.  From an
inertial solution $(\rho^{\rm in},v^{\rm in},V^{\rm in})$, define
\begin{align}
 \rho(t,y)&=\rho^{\rm in}(t,R(t)y),\notag\\
 w(t,y)&=R(t)^{-1}v^{\rm in}(t,R(t)y)-\omega e_z\times y,\notag\\
 V(t,y)&=V^{\rm in}(t,R(t)y).                                           \label{eq:rotating-change}
\end{align}
A direct differentiation gives
\begin{align}
 \rho_t+\nabla\cdot(\rho w)&=0,                                         \label{eq:rot-Euler-mass}\\
 \rho\left(w_t+w\cdot\nabla w+2\omega e_z\times w
 +\omega e_z\times(\omega e_z\times y)\right)+\nabla P(\rho)
 &=\nabla\cdot\mathbb T(w)-\rho\nabla V,                                \label{eq:rot-Euler-momentum}\\
 \Delta V&=4\pi\rho.                                                     \label{eq:rot-Euler-Poisson}
\end{align}
Because a rigid inertial rotation has zero relative velocity, the equilibrium
is $(\bar\rho,w)=(\bar\rho,0)$ and
\begin{equation}
 h(\bar\rho)+V_{\bar\rho}-\half\omega^2r^2+c=0.
\end{equation}
The base Lagrangian map is therefore the identity in rotating coordinates.
In inertial coordinates it would be $a\mapsto R(t)a$, not the identity.

\subsection{Exact pullback to the reference star}
\label{pullback}
For each nearby motion choose an initial label map $Y_0$ that pushes the
fixed reference mass $\bar\rho(a)da$ to the perturbed initial density.  Let
$Y(t,a)$ be the ensuing flow of $w$ on the fixed reference domain
$\Omega=\Omega_{\bar\rho}$:
\begin{equation}
 Y_t(t,a)=w(t,Y(t,a)),\qquad Y(0,a)=Y_0(a).                             \label{eq:flow-map}
\end{equation}
The equilibrium member has $Y_0(a)=a$.  Allowing $Y_0$ to vary is essential:
if one imposed $Y(0)=\mathrm{Id}$ on every perturbed motion, the exact mass
on the right side below would be the perturbed initial density, not
$\bar\rho$.
Set
\begin{equation}
 F=\nabla_aY,\qquad J=\det F,\qquad \Psi(t,a)=V(t,Y(t,a)).
\end{equation}
Mass conservation is the exact identity
\begin{equation}
 \rho(t,Y(t,a))J(t,a)=\bar\rho(a).                                      \label{eq:exact-mass-pullback}
\end{equation}
This fixed-reference identity parametrizes the component with the same total
mass as the equilibrium.  If one instead allows a perturbed reference
density $\rho_0^\varepsilon(a)$, the right side is
$\rho_0^\varepsilon$ and the Eulerian first variation below acquires the
time-independent term $\delta\rho_0$.  On every nonzero exponential mode or
root chain it vanishes because
$\sigma+\nabla\cdot(\bar\rho\xi)$ is time independent, whereas a function
in $e^{\lambda t}\C[t]$ with $\lambda\ne0$ has no nonzero time-independent
component.  Conservation of mass alone would show only that the integral of
this defect is zero.  Thus the fixed-reference formulation loses no
nonzero spectral information used in this paper.
The pulled-back rate of strain is
\begin{equation}
 e_Y(Y_t)=\half\left(\nabla_aY_tF^{-1}
                  +F^{-T}(\nabla_aY_t)^T\right),
\end{equation}
and
\begin{equation}
 \mathbb T_Y
 =2\nu_s\dev e_Y(Y_t)
  +\nu_b\operatorname{tr}(\nabla_aY_tF^{-1})I.                          \label{eq:Lagrangian-viscosity}
\end{equation}
Using the Piola identity, the exact Lagrangian momentum and Poisson
equations are
\begin{align}
 &\bar\rho\left(Y_{tt}+2\omega e_z\times Y_t
       +\omega e_z\times(\omega e_z\times Y)\right)\notag\\
 &\qquad
 +\operatorname{Div}_a\left[
     JP(\bar\rho/J)F^{-T}\right]
 =\operatorname{Div}_a\left(J\mathbb T_YF^{-T}\right)
   -\bar\rho F^{-T}\nabla_a\Psi,                                        \label{eq:exact-Lagrange-momentum}\\
 &\operatorname{Div}_a\left(JF^{-1}F^{-T}\nabla_a\Psi\right)
   =4\pi\bar\rho.                                                        \label{eq:exact-Lagrange-Poisson}
\end{align}
On the reference boundary, the exact traction condition is
\begin{equation}
 J\left(P(\bar\rho/J)I-\mathbb T_Y\right)F^{-T}N=0,                      \label{eq:exact-Lagrange-traction}
\end{equation}
and $\Psi$ is coupled to the exterior harmonic potential by continuity of
the potential and normal flux.  Equations
\eqref{eq:exact-mass-pullback}--\eqref{eq:exact-Lagrange-traction} are posed
on a fixed domain and are the appropriate starting point for a rigorous
physical-vacuum linearization.

\subsection{First variation}

Write
\begin{equation}
 Y(t,a)=a+\varepsilon\xi(t,a)+O(\varepsilon^2).
\end{equation}
The standard determinant and inverse expansions are
\begin{equation}
 J=1+\varepsilon\nabla\cdot\xi+O(\varepsilon^2),\qquad
 F^{-1}=I-\varepsilon\nabla\xi+O(\varepsilon^2).                         \label{eq:FJ-expansions}
\end{equation}
The Lagrangian density variation and the Eulerian density variation are,
respectively,
\begin{equation}
 \delta\rho_L=-\bar\rho\nabla\cdot\xi,\qquad
 \sigma=\delta\rho_E
 =\delta\rho_L-\xi\cdot\nabla\bar\rho
 =-\nabla\cdot(\bar\rho\xi).                                             \label{eq:Lagrange-Euler-density}
\end{equation}
The relative velocity perturbation is
\begin{equation}
 u=\xi_t.                                                               \label{eq:u-xit}
\end{equation}
We give the Eulerian conversion of the Poisson variation explicitly.  Put
\begin{equation}
 B_\xi=(\nabla\cdot\xi)I-\nabla\xi-(\nabla\xi)^T,
 \qquad \psi_L=\delta\Psi,
 \qquad \varphi=\psi_L-\xi\cdot\nabla V_{\bar\rho}.       \label{eq:Poisson-variation-fields}
\end{equation}
Since $\delta(JF^{-1}F^{-T})=B_\xi$, the first variation of
\eqref{eq:exact-Lagrange-Poisson} is
\begin{equation}
 \nabla\cdot\bigl(\nabla\psi_L+B_\xi\nabla V_{\bar\rho}\bigr)=0.
                                                               \label{eq:Lagrangian-Poisson-first-variation}
\end{equation}
For every smooth $V$,
\begin{equation}
 \nabla\cdot(B_\xi\nabla V)+\Delta(\xi\cdot\nabla V)
 =\nabla\cdot(\xi\Delta V).                                \label{eq:Poisson-conversion-identity}
\end{equation}
Using $\Delta V_{\bar\rho}=4\pi\bar\rho$ and
$\sigma=-\nabla\cdot(\bar\rho\xi)$, equations
\eqref{eq:Lagrangian-Poisson-first-variation}--
\eqref{eq:Poisson-conversion-identity} give
\begin{equation}
 \Delta\varphi=-\nabla\cdot(4\pi\bar\rho\xi)=4\pi\sigma.
\end{equation}
The decay condition identifies $\varphi$ with
\begin{equation}
 V_\sigma(x)=-\int_{\R^3}\frac{\sigma(y)}{\abs{x-y}}dy.                  \label{eq:linear-Poisson-appendix}
\end{equation}

We record the cancellation in the momentum equation.  The material
variation of the centrifugal acceleration contains
$\omega e_z\times(\omega e_z\times\xi)$.  The Lagrangian variations of
$\nabla h(\bar\rho)$ and $\nabla V_{\bar\rho}$ contain the Hessian of the
equilibrium potentials applied to $\xi$.  Differentiating
\begin{equation}
 \nabla h(\bar\rho)+\nabla V_{\bar\rho}
 +\omega e_z\times(\omega e_z\times a)=0
\end{equation}
gives the explicit identity
\[
 D^2\bigl(h(\bar\rho)+V_{\bar\rho}\bigr)\xi
 +\omega e_z\times(\omega e_z\times\xi)=0.
\]
The pressure and gravitational acceleration contribute
$\nabla(h'(\bar\rho)\sigma+V_\sigma)
+D^2(h(\bar\rho)+V_{\bar\rho})\xi$.
Thus the material terms cancel.  The base viscous stress vanishes, so its
variation has no density-prefactor correction.  The remaining equation is
\begin{equation}
 u_t+2\omega e_z\times u
 +\nabla\bigl(h'(\bar\rho)\sigma+V_\sigma\bigr)
 =\bar\rho^{-1}\nabla\cdot\mathbb T(u),                                 \label{eq:linear-momentum-appendix}
\end{equation}
whereas differentiating \eqref{eq:Lagrange-Euler-density} in time gives
\begin{equation}
 \sigma_t=-\nabla\cdot(\bar\rho u).                                     \label{eq:linear-mass-appendix}
\end{equation}
These are exactly \eqref{eq:linear-mass}--\eqref{eq:linear-momentum}.

At the reference boundary $\bar\rho=0$, so
$P(\bar\rho/J)=P(0)=0$ identically for every admissible deformation.  The
base viscous stress is also zero.  Hence the first variation of
$J\mathbb T_YF^{-T}$ is exactly $\mathbb T(\xi_t)$, while the pressure
traction has zero first variation.  The remaining natural condition is
\begin{equation}
 \mathbb T(u)\mathbf n=0                                      \label{eq:linear-traction}
\end{equation}
for sufficiently regular perturbations.  For general energy states, the
equations and boundary condition are realized together by the form
identity; no separate trace of an individual momentum residual is
required.  No boundary condition on the Eulerian density perturbation
$\sigma$ is added.

\subsection{Equivalence of Lagrangian and Eulerian modes}

For a mode with $\lambda\ne0$, \eqref{eq:u-xit} gives
\begin{equation}
 \xi=\lambda^{-1}u,\qquad
 \sigma=-\lambda^{-1}\nabla\cdot(\bar\rho u).                            \label{eq:mode-conversion}
\end{equation}
For a root chain, the time derivative is invertible on
$e^{\lambda t}\C[t]$, giving the same triangular formulas at every chain
level.  Hence the Lagrangian and Eulerian formulations have identical
nonzero characteristic values and algebraic multiplicities.  Static
weighted-divergence-free displacements include relabeling directions, while
time-dependent fields in $\Ker\cC$ may represent physical solenoidal
velocities.

\section{Regularity of the slow-rotation branch}
\label{app:slow-branch}

The standard slow-rotation construction is most naturally carried out for
the enthalpy rather than for the zero-extended density.  We prove here the
parameter regularity and vacuum geometry in
Propositions~\ref{prop:slow-branch-full-range} and
\ref{prop:slow-vacuum-tangents}.  Second density derivatives may be
singular at the moving vacuum interface, but they remain integrable
throughout the range $\gamma_0<2$.

\begin{lemma}
\label{lem:transverse-positive-part}
Let $D\Subset\mathbb R^3$ be a bounded $C^2$ domain and let
$\mathscr U\subset C^1(\overline D;\mathbb R)$ be an open neighborhood for which
there are $c_0,c_b,\delta_0>0$ such that, uniformly for
$q\in\mathscr U$,
\[
 q\le-c_b<0\quad\hbox{on }\partial D,\qquad
 |\nabla q|\ge c_0\quad\hbox{on }\{|q|<2\delta_0\}.
\]
Suppose $\mathscr G(t)=0$ for $t\le0$, $\mathscr G\in C^2((0,\infty))\cap C^1(\mathbb R)$,
and, for some $\alpha>1$,
\[
 |\mathscr G^{(j)}(t)|\le Ct^{\alpha-j},\qquad
 0<t<2\delta_0,\quad j=0,1,2.
\]
Then
\[
 q\longmapsto \mathscr G(q):\mathscr U\to L^1(D),\qquad
 q\longmapsto V_{\mathscr G(q)}:\mathscr U\to C^1(\overline D)
\]
are $C^2$.  Their derivatives are
\begin{align}
 D(\mathscr G(q))v&=\mathscr G'(q)v,&
 D^2(\mathscr G(q))[v_1,v_2]
   &=\mathbf1_{\{q>0\}}\mathscr G''(q)v_1v_2,                       \label{eq:positive-part-derivatives}\\
 D(V_{\mathscr G(q)})v&=V_{\mathscr G'(q)v},&
 D^2(V_{\mathscr G(q)})[v_1,v_2]
   &=V_{\mathbf1_{\{q>0\}}\mathscr G''(q)v_1v_2}.                  \label{eq:potential-derivatives}
\end{align}
\end{lemma}

\begin{proof}
The assertion is local in $q$.  Fix $q_0\in\mathscr U$ and restrict to a
sufficiently small $C^1$-neighborhood of $q_0$.  Its zero set is compactly
contained in $D$.  Choose finitely many graph charts covering that set,
with each chart compactly contained in a larger chart on which a fixed
directional derivative of $q_0$ is bounded away from zero.  The same
derivative bounds hold for every nearby $q$.  The inverse-function theorem
then gives coordinates $(y',t)$ with $t=q(y)$, uniform bi-Lipschitz
constants, and Jacobians bounded above and below.  No normal exponential
coordinates, or second spatial derivatives of $q$, are needed.
Only this collar requires an estimate.  For $0<\varepsilon<\delta_0$,
\begin{equation}
 \int_{\{0<q<\varepsilon\}}q^{\alpha-2}\,dy
 \le C\int_0^\varepsilon t^{\alpha-2}\,dt
 \le C\varepsilon^{\alpha-1}.                             \label{eq:interface-L1}
\end{equation}
After one spatial derivative of the Newtonian kernel, the required estimate
is
\begin{equation}
 \sup_{x\in D}\int_{\{0<q<\varepsilon\}}
 \frac{q(y)^{\alpha-2}}{|x-y|^2}\,dy
 \le C\varepsilon^{\alpha-1}(1+|\log\varepsilon|).        \label{eq:interface-C1}
\end{equation}
To verify it, use one of these graph charts.  If $x$ belongs to the larger
chart, write its coordinates as $(x',s)$.  Bi-Lipschitz equivalence bounds
$|x-y|$ below by a constant times
$(|x'-y'|^2+|s-t|^2)^{1/2}$.  After translating $y'$, the local contribution
is bounded by
\[
 C\sup_{s\in\mathbb R}\int_0^\varepsilon t^{\alpha-2}
 \left(\int_{|y'|<C}\frac{dy'}{|y'|^2+|t-s|^2}\right)dt
 \le C\sup_s\int_0^\varepsilon t^{\alpha-2}
       \log\!\left(2+\frac1{|t-s|}\right)dt,
\]
For $|s|\le2\varepsilon$, put $t=\varepsilon u$ and
$a=s/\varepsilon$.  The last expression is bounded by
\[
 C\varepsilon^{\alpha-1}
 \left[(1+|\log\varepsilon|)\int_0^1u^{\alpha-2}\,du
 +\sup_{|a|\le2}\int_0^1u^{\alpha-2}|\log|u-a||\,du\right].
\]
To bound the last supremum, put $\beta=\alpha-1>0$.  For $0<a<1/2$,
split the integral at $2a$.  Substitution $u=av$ on $(0,2a)$ bounds that
part by
\[
 C a^\beta\left(1+|\log a|
       +\int_0^2v^{\beta-1}|\log|v-1||\,dv\right),
\]
which is uniformly bounded.  On $(2a,1)$ one has
$u/2\le|u-a|\le u$, so the contribution is bounded by
$C\int_0^1u^{\beta-1}(1+|\log u|)\,du<\infty$.
The same latter bound applies for $a\le0$.  For $1/2\le a\le2$, split
off $(0,1/4)$; on the rest the weight is bounded and the logarithmic
singularity has a uniformly bounded integral.  Thus the supremum is
finite for every $\alpha>1$.  If
$|s|>2\varepsilon$, then $|t-s|$ stays separated from zero relative to
$|s|$, and the same bound is simpler.  This proves
\eqref{eq:interface-C1}.  If $x$ lies outside the larger chart, it is
separated from the smaller chart and \eqref{eq:interface-L1} suffices.
Summing the finite cover proves the estimate uniformly in $x$.
The undifferentiated potential has the same bound, since $D$ is bounded.
Both right-hand sides tend to zero because $\alpha>1$.

In particular the second potential in \eqref{eq:potential-derivatives}
belongs to $C^1$.  Indeed, truncate its source to $\{q\ge\varepsilon\}$.
The truncated source is bounded and its Newtonian potential is $C^1$.
Estimate \eqref{eq:interface-C1}, multiplied by
$\|v_1\|_\infty\|v_2\|_\infty$, makes these potentials Cauchy in $C^1$
as $\varepsilon\downarrow0$ and identifies their limit with the stated
potential and gradient.

For continuity of the bilinear second derivatives, let $q_n\to q$ in
$C^1$.  On $\{|q|\ge2\varepsilon\}$ the signs agree for large $n$, and
uniform continuity of $\mathscr G''$ away from zero gives uniform
convergence of the coefficients.  On $\{|q|<2\varepsilon\}$, any positive
$q_n$ is smaller than $3\varepsilon$ once
$\|q_n-q\|_\infty<\varepsilon$.  Apply the uniform collar estimates
separately to $q$ and $q_n$.  They show that the second derivatives converge
in operator norm from $C^1\times C^1$ to $L^1$, and, after applying the
Newtonian potential, to $C^1$.

Finally, $\mathscr G'$ is locally absolutely continuous across zero because
$\alpha-2>-1$.  For a sufficiently small increment $v$, the segment
$q+tv$ stays in the chosen neighborhood.  The fundamental theorem of
calculus gives, in the indicated target spaces,
\[
 [D\mathscr G(q+v)-D\mathscr G(q)]w
 =\int_0^1\mathbf1_{\{q+tv>0\}}\mathscr G''(q+tv)vw\,dt.
\]
The continuity just proved identifies its Fr\'echet derivative.  The same
argument after the Newtonian potential, and the first-order identity for
$\mathscr G(q+v)-\mathscr G(q)$, prove both $C^2$ assertions and all
formulas in \eqref{eq:positive-part-derivatives}--\eqref{eq:potential-derivatives}.
\end{proof}

\begin{proof}[Proof of Proposition~\ref{prop:slow-branch-full-range}]
The equation-of-state asymptotics give, as $t\downarrow0$,
\begin{align}
 h(t)&=A_0t^{\gamma_0-1}\bigl(1+O(t^{a_0})\bigr),\notag\\
 \mathscr G(y)&=A_0^{-\alpha}y^\alpha
       \bigl(1+O(y^{a_0\alpha})\bigr),\qquad
 A_0=\frac{K_0\gamma_0}{\gamma_0-1}.                       \label{eq:enthalpy-inverse-asymptotics}
\end{align}
Moreover, the differentiated remainder assumptions in
\eqref{eq:EOS-asymptotic} and the inverse identities
\begin{equation}
 \mathscr G'(y)=\frac1{h'(\mathscr G(y))},\qquad
 \mathscr G''(y)=-\frac{h''(\mathscr G(y))}{h'(\mathscr G(y))^3}
\end{equation}
give, for $0<y<y_0$,
\begin{equation}
 c y^{\alpha-1}\le \mathscr G'(y)\le C y^{\alpha-1},\qquad
 |\mathscr G''(y)|\le C y^{\alpha-2}.                            \label{eq:G-derivative-bounds}
\end{equation}
In particular the zero extension satisfies $\mathscr G\in C^1(\mathbb R)$ and
$\mathscr G'(0)=0$.
Choose a ball $D$ that contains all spherical supports and on whose boundary
$q_{\mu,0}\le-c_b<0$, uniformly for $\mu\in I_\mu$.  Take
$\mathscr U$ to be an open $C^1$-neighborhood of this compact spherical
family on which the same boundary gap and transversality bounds hold.  Put
\[
 Y=\{w\in C^1(\overline D):w\hbox{ is axisymmetric and even in }z,\
                              w(0)=0\}.
\]
For $q=h(\mu)+w$, define
\[
 \mathcal F(\mu,s,w)(x)
 =w(x)-V_{\mathscr G(h(\mu)+w)}(0)+V_{\mathscr G(h(\mu)+w)}(x)-\frac12sr^2.
\]
On every spherical free surface,
\[
 \partial_nq_{\mu,0}=-\frac{M_0(\mu)}{R_0(\mu)^2}<0.
\]
This is uniform on $I_\mu$, so the functions $q_{\mu,0}$ have uniformly
transverse zero sets.  Lemma~\ref{lem:transverse-positive-part} shows that
$\mathcal F$ is $C^2$ from a neighborhood in
$I_\mu\times\mathbb R\times Y$ to $Y$.  Its $w$-derivative is
\eqref{eq:normalized-enthalpy-linearization}, whose inverses are uniformly
bounded by Lemma~\ref{lem:spherical-density-spectrum}.  The uniform
implicit-function theorem therefore yields
\[
 (\mu,s)\longmapsto q_{\mu,s}
       \in C^2(I_\mu\times(-s_0,s_0);C^1(\overline D)).
\]
Uniqueness in the even subspace gives equatorial symmetry.  This is the
center-normalized Lichtenstein construction.  It is compatible with the
local branches of Heilig and Jang--Makino (the 2017 paper in its stated
exponent range), and with the general-equation-of-state constructions in
\cite{JangMakino2019,LinWang2023}; see also
\cite{Heilig1994,JangMakino2017}.

Equations \eqref{eq:positive-part-derivatives} and
\eqref{eq:potential-derivatives} give
\begin{align}
 \partial_p\widetilde\rho
 &=\mathbf1_{\{q>0\}}\mathscr G'(q)q_p,\notag\\
 \partial_{pp'}\widetilde\rho
 &=\mathbf1_{\{q>0\}}
   \{\mathscr G''(q)q_pq_{p'}+\mathscr G'(q)q_{pp'}\}
       \quad\hbox{in }L^1(D),                              \label{eq:branch-second-density}
\end{align}
for $p,p'\in\{\mu,s\}$.  In particular, there is no surface measure.
Since the zero extension of $\mathscr G$ belongs to $C^1(\mathbb R)$ and
$\mathscr G'(0)=0$, the standard $C^0$ Nemytskii theorem also gives
$\widetilde\rho_{\mu,s}\in C^1_{\mu,s}(C(\overline D))$.
The $C^2$ regularity of $M$, $I$, and $c$ follows from
\eqref{eq:branch-second-density}, differentiation under the integral, and
the $C^1$-valued potential estimate.  The spatial and support properties
proved next show that these physical solutions are regular equilibria,
and complete the proof of Proposition~\ref{prop:slow-branch-full-range}.
\end{proof}

\begin{proof}[Proof of Proposition~\ref{prop:slow-vacuum-tangents}]
All estimates are uniform for $\mu$ in the stated compact interval and
$s$ sufficiently small.  Boundedness of the zero-extended density first
gives $q\in C^{1,\theta}$ for every $\theta<1$.  Since
$\mathscr G'(0)=0$ and \eqref{eq:G-derivative-bounds} makes $\mathscr G'$ H\"older with
exponent $\min\{1,\alpha-1\}$, the zero-extended density belongs to
$C^{1,\eta}$ for every $0<\eta<\min\{1,\alpha-1\}$.  Global Schauder
estimates for $\Delta V=4\pi\rho$ then give uniform bounds
\[
 q_{\mu,s}\in C^{3,\eta}(\overline D),
 \qquad 0<\eta<\min\{1,\alpha-1\}.
\]
For any prescribed $\eta$ in this range, first take uniform bounds at a
slightly larger exponent.  Parameter convergence in $C^1$ and interpolation
then give convergence in the spatial $C^{3,\eta}$ norm, in particular in
$C^2$.
The normal derivative remains bounded away from zero at $q_{\mu,s}=0$.
The level-set theorem therefore gives a $C^{3,\eta}$ free surface.  Its
second fundamental form is close to that of the corresponding sphere and
is uniformly positive for small $s$.

Write the surface as the radial graph
$|x|=\mathcal R_{\mu,s}(x/|x|)$, where $\mathcal R_{\mu,s}$ is positive, axisymmetric,
even in $z$, and continuous in $C^1(\mathbb S^2)$ as a function of the
parameters.  Explicit support maps are
\begin{equation}
 \Psi_{\mu,s}(a)=\mathcal R_{\mu,s}(a/|a|)a\quad(a\ne0),\qquad
 \Psi_{\mu,s}(0)=0.                                      \label{eq:radial-support-map}
\end{equation}
Their inverses have the same radial form.  Upper and lower bounds on
$\mathcal R_{\mu,s}$ and its $C^1$ norm give uniform bi-Lipschitz constants, and
parameter continuity holds in $W^{1,\infty}$.  Moreover,
\[
 \det D\Psi_{\mu,s}(a)=\mathcal R_{\mu,s}(a/|a|)^3
 \quad\hbox{for almost every }a\in B.
\]
The angular derivatives contribute a tangential shear, so this determinant
is positive and bounded above and below.  Differentiability at the single
point $a=0$ is not required.

On each radial segment, the fundamental theorem of calculus and the
uniformly negative radial derivative at the surface give, in a fixed
boundary collar,
\[
 q_{\mu,s}\circ\Psi_{\mu,s}=A_{\mu,s}\delta_B,\qquad
 0<A_*\le A_{\mu,s}\le A^*.
\]
Together with \eqref{eq:enthalpy-inverse-asymptotics}, this yields
\[
 \rho_{\mu,s}\asymp d_{\mu,s}^{\alpha},\qquad
 \mathscr G'(q_{\mu,s})=\frac1{h'(\rho_{\mu,s})}
       \asymp d_{\mu,s}^{\alpha-1}.
\]
Finally, \eqref{eq:branch-second-density} gives the tangent formula, and
\begin{equation}
 \|f^p_{\mu,s}\|_{X_{\mu,s}}^2
 =\int_{\Omega_{\mu,s}}\mathscr G'(q_{\mu,s})|q_p|^2\,dx,\qquad
 h'(\rho_{\mu,s})|f^p_{\mu,s}|^2
 \le Cd_{\mu,s}^{\alpha-1}.                               \label{eq:tangent-dominating-weight}
\end{equation}
The right-hand side is integrable in a boundary collar.  Under the unitary pullback,
\begin{equation}
 U_{\mu,s}f^p_{\mu,s}
 =|\det D\Psi_{\mu,s}|^{1/2}
   [\mathscr G'(q_{\mu,s}\circ\Psi_{\mu,s})]^{1/2}
   (\partial_pq_{\mu,s})\circ\Psi_{\mu,s}.                \label{eq:weighted-tangent-explicit}
\end{equation}
Every factor is continuous in the parameters in the uniform norm (the
Jacobian in the essential-supremum norm).  Since $\mathscr G'(0)=0$, the right-hand side is continuous in
$L^2(B)$.  The already proved $C^1$ parameter dependence of the
zero-extended density in $C(\overline D)$ also gives uniform convergence
of its first Eulerian tangents on the fixed ball.  No parameter derivative
of the support map or of the singular weight $h'(\rho)$ is used.
\end{proof}

\begin{remark}
The exponent $\gamma_0=5/3$ is of particular physical interest.  It is
the nonrelativistic-degenerate low-density asymptotic of the standard
Chandrasekhar equation of state for white-dwarf matter (and is also the
adiabatic exponent of a monatomic ideal gas).  For $\gamma_0=5/3$, one
has $\alpha=3/2$ and
$\mathscr G''(q)=O(q^{-1/2})$ at the vacuum interface.  Although second
parameter derivatives of the zero-extended density can be unbounded
there, the proof uses the integrable estimate
\[
 \|\mathscr G''(q)\|_{L^1(\{0<q<\varepsilon\})}
 =O(\varepsilon^{1/2})
\]
and the $C^1$ potential estimate
\[
 O\!\left(\varepsilon^{1/2}(1+|\log\varepsilon|)\right).
\]
By \eqref{eq:interface-L1}--\eqref{eq:interface-C1}, these give the
required $C^2$ dependence in $L^1$ for the density and in $C^1$ for its
potential.
\end{remark}

\begin{remark}
A factorization
$\rho_{\mu,s}\circ\Psi_{\mu,s}=\beta_{\mu,s}\delta_B^\alpha$ with
$\beta_{\mu,s}\in C^1$ need not follow from
\eqref{eq:EOS-asymptotic} when $a_0\alpha<1$, and is not used.  The
enthalpy factorization in \eqref{eq:uniform-vacuum-factor}, together with
the two displayed weight comparisons, is the precise regularity required.
\end{remark}

\section{Analytic Fredholm pencils and algebraic multiplicity}
\label{app:Gohberg-Sigal}

We recall only the facts used in Section~\ref{sec:KTC}; see
\cite{GohbergSigal1971}.  Let
$F:\mathcal U\to\mathcal L(X)$ be analytic on a plane domain and suppose
that every $F(\lambda)$ is Fredholm of index zero.  A point $\lambda_0$ is
a \emph{characteristic value} if $F(\lambda_0)$ is not invertible.  If
$F$ is invertible at one point of the connected domain $\mathcal U$, the
analytic Fredholm theorem states that $F^{-1}$ is finitely meromorphic:
the characteristic values are isolated and the principal part of
$F^{-1}$ at each of them has finite rank.

At an isolated characteristic value, a root chain
$u_0,\ldots,u_{m-1}$, with $u_0\ne0$, satisfies
\[
 \sum_{j=0}^{\ell}\frac1{j!}
 F^{(j)}(\lambda_0)u_{\ell-j}=0,
 \qquad 0\le\ell<m.
\]
A canonical system of root functions has partial multiplicities
$\kappa_1,\ldots,\kappa_N$.  The Gohberg--Sigal algebraic multiplicity is
\[
 \operatorname{algmult}(\lambda_0;F)
 =\sum_{j=1}^N\kappa_j
 =\operatorname{tr}\frac1{2\pi i}
   \int_{|\lambda-\lambda_0|=\delta}
   F'(\lambda)F(\lambda)^{-1}\,d\lambda.
\]
The trace is well defined because the contour integral is finite rank.
The algebraic multiplicity is the sum of the partial multiplicities, not
the maximal root-chain length (the latter is the pole order of
$F^{-1}$).  Multiplication on either side by analytic
invertible operator families preserves all partial multiplicities.

We also use the following homotopy form of the Gohberg--Sigal argument
principle.  Let $\mathscr O$ be a bounded domain with piecewise smooth
boundary.  Suppose $F_t(\lambda)$ is jointly norm-continuous in
$(t,\lambda)$ and analytic Fredholm of index zero in $\lambda$ on one
common neighborhood of $\overline{\mathscr O}$, for $t\in[0,1]$.  If $F_t(\lambda)$ is invertible for every
$(t,\lambda)\in[0,1]\times\partial\mathscr O$, then
\[
 \sum_{\lambda\in\mathscr O}
   \operatorname{algmult}(\lambda;F_t)
\]
is independent of $t$.  Thus characteristic values may split or collide
during the deformation, but the total algebraic multiplicity inside the
fixed contour cannot change unless a characteristic value crosses the
boundary.

\section{Functional-analytic lemmas}
\label{sec:analytic-lemmas}

\subsection{Weights and conserved covectors}

\begin{lemma}
\label{lem:bounded-covectors}
For every regular rigidly rotating equilibrium whose pressure law satisfies
Hypothesis~\ref{hyp:EOS},
\begin{equation}
 \frac1{h'(\bar\rho)},\quad
 \frac{z^2}{h'(\bar\rho)},\quad
 \frac{r^4}{h'(\bar\rho)}\in L^1(\Omega_{\bar\rho}).                    \label{eq:covector-integrability}
\end{equation}
Moreover $\tau=-\partial_z\bar\rho\in X$ and all four functionals
\eqref{eq:mass-covector}--\eqref{eq:J-covector} are bounded on the energy
phase space.
\end{lemma}

\begin{proof}
The pressure asymptotics and $\gamma_0<2$ imply
$\inf_{0<s\le\|\bar\rho\|_\infty}h'(s)>0$.  Since the support is bounded,
$1/h'(\bar\rho)$ and its products with $z^2$ and $r^4$ are integrable.
This proves \eqref{eq:covector-integrability}.  Put
$W=h(\bar\rho)$.  The equilibrium equation gives
$\nabla W=-\nabla V_{\bar\rho}+\omega^2r e_r\in L^\infty$, and continuity
of $\bar\rho$ gives zero trace at the free boundary.  Thus the zero
extension of $W$ lies in $W^{1,\infty}$, and the chain rule gives
\begin{equation}
 h'(\bar\rho)|\nabla\bar\rho|^2
 =\frac{|\nabla W|^2}{h'(\bar\rho)}\in L^1.               \label{eq:translation-X-estimate}
\end{equation}
Thus $\tau=-\partial_z\bar\rho\in X$.  Cauchy--Schwarz in $X$ proves
boundedness of the density moments, and the velocity moments are bounded
directly in $H$.
\end{proof}

\subsection{Newtonian compactness}

\begin{lemma}
\label{lem:gravity-compact}
The form
\begin{equation}
 (\sigma,\eta)\longmapsto\int_{\R^3}\sigma V_{\bar\eta}\,dx
\end{equation}
is bounded on $X\times X$ and its Riesz operator is compact and
self-adjoint on $X$.  Hence the Riesz realizations of $L_0$ and $L_J$ are the identity plus
compact self-adjoint operators; in particular, they are Fredholm and
have finite Morse index.
\end{lemma}

\begin{proof}
If $\rho_*=\|\bar\rho\|_\infty$, continuity and positivity of $h'$ on
$(0,\rho_*]$, together with $h'(s)\to+\infty$ as $s\downarrow0$, give
$\inf_{0<s\le\rho_*}h'(s)>0$.  The unitary map
\begin{equation}
 U:X\to L^2_{\rm ax}(\Omega),\qquad U\sigma=h'(\bar\rho)^{1/2}\sigma
\end{equation}
conjugates the Riesz operator of the gravitational form to the integral
operator with kernel
\begin{equation}
 K(x,y)=-\frac1{\sqrt{h'(\bar\rho(x))h'(\bar\rho(y))}}\frac1{|x-y|}.
                                                               \label{eq:gravity-HS-kernel}
\end{equation}
The weight factor is bounded, and
$|x-y|^{-2}\in L^1(\Omega\times\Omega)$ in three dimensions.  Hence
$K\in L^2(\Omega\times\Omega)$: the conjugated operator is
Hilbert--Schmidt and therefore compact; symmetry of the kernel gives
self-adjointness.  Its Hilbert--Schmidt norm also gives boundedness of the
form.  Finally,
$\xi_r=r^2/h'(\bar\rho)$ belongs to $X$ and Cauchy--Schwarz gives
$|\int r^2\sigma|\le\|\xi_r\|_X\|\sigma\|_X$.  Thus
\[
 \widehat L_J=I+K_g+\frac{\omega^2}{I}\,\xi_r\otimes\xi_r
\]
is bounded, self-adjoint, and Fredholm, with essential spectrum $\{1\}$,
finite-dimensional kernel, and finite Morse index.
\end{proof}

\subsection{Closed form realization}

\begin{lemma}
\label{lem:full-generation}
The form \eqref{eq:full-sectorial-form} is densely defined and sectorial
after addition of a sufficiently large multiple of the
$\mathscr H_{\rm full}$ inner product.  The unshifted form represents
$-\cA_{\rm NSP}$ with domain \eqref{eq:full-generator-domain}.  Hence
$\cA_{\rm NSP}$ is closed, has nonempty resolvent, and generates a strongly
continuous analytic semigroup on $\mathscr H_{\rm full}$.  The same
statements hold for \eqref{eq:reduced-generator} on
$\mathscr H_{\rm red}$.
\end{lemma}

\begin{proof}
Smooth compactly supported axisymmetric fields are dense in $H$: first
approximate in the weighted $L^2$ norm and then average the approximants
under rotations about the $z$-axis.  Hence $H^1_{\rm ax}$ is dense in $H$.
Moreover, $P_\diamond$ preserves $H^1_{\rm ax}$, so applying it to such
approximants shows that $V_\diamond$ is dense in $H_\diamond$.

Lemma~\ref{lem:C-form-bound} makes $\cC:H^1_{\rm ax}\to X$ bounded;
$L_0:X\to X^*$ and $G_0:H\to H$ are bounded.  Thus all off-diagonal terms in
\eqref{eq:full-sectorial-form} are bounded on
$\mathscr V_{\rm full}$.  Young's inequality and the generalized Korn
estimate \eqref{eq:Korn-shifted} give, for a sufficiently large shift
$\beta_0$,
\begin{equation}
 \Ree\{\mathfrak a_{\rm full}[U,U]+\beta_0\|U\|_{\mathscr H_{\rm full}}^2\}
 \ge c\bigl(\|\sigma\|_X^2+\|u\|_{H^1}^2\bigr),
 \qquad U=(\sigma,u).                                      \label{eq:full-form-coercive}
\end{equation}
The same boundedness estimates give
\begin{equation}
 |\Imm\mathfrak a_{\rm full}[U,U]|
 \le C\bigl(\|\sigma\|_X^2+\|u\|_{H^1}^2\bigr)
 \le C'\Ree\{\mathfrak a_{\rm full}[U,U]
              +\beta_0\|U\|_{\mathscr H_{\rm full}}^2\}.   \label{eq:full-form-sector}
\end{equation}
Thus the numerical range lies in a fixed sector.  Therefore
the shifted form is closed and sectorial.  The first representation theorem
\cite[Chapter~1]{Ouhabaz2005} represents the unshifted form by an
operator $B$ for which $B+\beta_0 I$ is $m$-sectorial.
Testing separately against $(\eta,0)$ and $(0,w)$ shows
\begin{equation}
 B(\sigma,u)=(-\cC u,\mathsf Du+G_0u+\cC^*L_0\sigma),
\end{equation}
and that its operator domain is exactly
\eqref{eq:full-generator-domain}.  Hence $B=-\cA_{\rm NSP}$.
The standard sectorial-form theorem yields the claimed analytic semigroup;
the traction condition \eqref{eq:linear-traction} is the natural boundary
condition encoded by $d$.

On $X_{00}\times V_\diamond$ replace $\cC,L_0,G_0$ by
$\cC_\diamond,L_J,G$.  Coercivity follows directly from
\eqref{eq:Korn-coercive}, and the identical argument gives
\eqref{eq:reduced-generator-domain}.  This is a direct construction of the
density--velocity generator; it does not identify $X_{00}\times H_\diamond$
with the second-order companion space.
\end{proof}

\subsection{Fixed-ambient form continuity}

\begin{lemma}
\label{lem:form-continuity}
For the branch of Proposition~\ref{prop:slow-branch-full-range}, the Riesz
realizations of $Q_J|_{X_0}$ and
$\widetilde Q_{\rm EP}|_{X_0}$, as well as the translation projection,
admit extensions to a fixed ambient Hilbert space that depend continuously
in operator norm on $(\mu,\kappa)$.  The added directions have eigenvalue $1$.
The Newtonian and Rayleigh contributions are compact.  Consequently their
isolated finite-dimensional spectral projections vary continuously.
\end{lemma}

\begin{proof}
Choose a ball $D$ containing all nearby supports and set
$\mathcal H=L^2_{\rm ax}(D,dx)$.  Write $p=(\mu,\kappa)$,
$q_p=q_{\mu,\omega(\mu,\kappa)^2}$, $\rho_p=\mathscr G(q_p)$,
$\omega_p=\omega(\mu,\kappa)$, and $I_p=I(\rho_p)$.  Put
\begin{equation}
 W_p(x)=\bigl[\mathscr G'(q_p(x))\bigr]^{1/2}
 =\mathbf1_{\Omega_p}(x)h'(\rho_p(x))^{-1/2},\qquad
 j_p=r^2W_p.                                             \label{eq:ambient-weight}
\end{equation}
Here $\mathscr G'$ is extended by zero on $(-\infty,0]$.  The physical-vacuum
estimate and fixed-domain convergence of $q_p$ imply
$W_p\to W_{p_0}$ uniformly on $D$.  Under the isometry
$g=h'(\rho_p)^{1/2}\sigma$, followed by zero extension, the mass constraint
is $(g,W_p)_{\mathcal H}=0$.  Introduce the norm-continuous orthogonal
projection
\begin{equation}
 \Pi_pg=g-\frac{(g,W_p)_{\mathcal H}}{\|W_p\|_{\mathcal H}^2}W_p.          \label{eq:ambient-mass-projection}
\end{equation}
The norms of $W_p$ are bounded away from zero on the compact parameter set.

Let $\mathsf G_p$ be the integral operator on $\mathcal H$ with kernel
\begin{equation}
 -\frac{W_p(x)W_p(y)}{|x-y|}.
\end{equation}
It is Hilbert--Schmidt, and it depends continuously on $p$ in
Hilbert--Schmidt norm because $|x-y|^{-2}\in L^1(D\times D)$ and $W_p$
converges uniformly.  The constrained viscous Riesz operator, extended to
$\mathcal H$, is
\begin{equation}
 \mathbb L_{J,p}=I+\Pi_p\mathsf G_p\Pi_p
  +\frac{\omega_p^2}{I_p}\,
       \Pi_p(j_p\otimes j_p)\Pi_p.                       \label{eq:ambient-viscous-Riesz}
\end{equation}
On the physical mass-zero subspace this is the Riesz realization of $Q_J$;
on the added mass direction and functions supported outside $\Omega_p$ it
is the identity.  The added directions therefore have eigenvalue $1$.
All nonidentity terms in \eqref{eq:ambient-viscous-Riesz} are compact and
operator-norm continuous.

For the Rayleigh term define, with values in $L^2(0,R^*;dr)$,
\begin{equation}
 (\mathcal K_pg)(r)=
 \frac{\mathbf1_{(0,R_p)}(r)}{\sqrt{m_{\rho_p}'(r)}}
 \int_0^r\xi\int_{\mathbb R}
       W_p(\xi,z)(\Pi_pg)(\xi,z)\,dz\,d\xi .              \label{eq:Rayleigh-fixed-ambient}
\end{equation}
Here $R_p$ is the equatorial radius and $R^*$ is a common upper bound.
For $g\in\Ran\Pi_p$, this is precisely the square root of the Rayleigh
form.  Precomposition by $\Pi_p$ gives $m_{W_p\Pi_pg}(R_p)=0$, which is needed
for the endpoint estimate below.

The endpoint estimates \eqref{eq:Rayleigh-axis-tail} and
\eqref{eq:Rayleigh-boundary-tail}, applied to
$\sigma=W_p\Pi_pg$, give uniformly in $p$
\begin{equation}
 \|\mathbf1_{(0,\delta)}\mathcal K_p\|^2\le C\delta^2,
 \qquad
 \|\mathbf1_{(R_p-\delta,R_p)}\mathcal K_p\|^2\le C\delta.               \label{eq:ambient-Rayleigh-tails}
\end{equation}
Fix $p_0$.  On every common interior interval
$[\delta,R_{p_0}-2\delta]$, for $p$ close to $p_0$, uniform convergence of
$W_p$ and $m_{\rho_p}'$, the positive lower bound for $m_{\rho_p}'$, and
$\Pi_p\to\Pi_{p_0}$ in norm imply operator-norm convergence of the
truncated column--Volterra operators.  These truncated operators are
compact, either directly from their integral kernels or from the bounded
map into $H^1$ followed by Rellich compactness.  Since
$R_p\to R_{p_0}$, the part outside the common interval is contained in
endpoint strips of width at most $3\delta$ for both $p$ and $p_0$.
Equation \eqref{eq:ambient-Rayleigh-tails} therefore gives
\[
 \limsup_{p\to p_0}\|\mathcal K_p-\mathcal K_{p_0}\|
 \le C\delta^{1/2}.
\]
Letting $\delta\downarrow0$ proves norm continuity.  It also proves
compactness, since $\mathcal K_p$ is a norm limit of compact interior
truncations.

The constrained inviscid Riesz operator is represented on $\mathcal H$ by
\begin{equation}
 \mathbb L_{{\rm EP},p}
 =I+\Pi_p\mathsf G_p\Pi_p
   +8\pi\omega_p^2\mathcal K_p^*\mathcal K_p.             \label{eq:ambient-inviscid-Riesz}
\end{equation}
All nonidentity terms are compact and norm-continuous.  As for
\eqref{eq:ambient-viscous-Riesz}, the artificial mass direction and the
exterior complement contribute only the eigenvalue \(1\).  The normalized translation vector
$t_p=-W_p\partial_zq_p/\|W_p\partial_zq_p\|_{\mathcal H}$ is well defined:
the translation mode is nonzero, and on each compact branch neighborhood its
denominator is bounded away from zero by continuity.  Hence $t_p$ varies
continuously, so its orthogonal projection does as well.  Kato's theorem on
Riesz projections \cite[Chapter~IV]{Kato1995} proves the last assertion.
\end{proof}

\noindent \textbf{Funding}
Ming Cheng is supported by the Natural Science Foundation of Jilin Province
(Grant No.~20260102253JC) and the National Natural Science Foundation of China
(Grant No.~12371191).  Zhiwu Lin was supported in part by the National Natural
Science Foundation of China (No.~12494544).  Yucong Wang was supported in part
by the National Natural Science Foundation of China (No.~12501290), the Furong
Young Talents Program for Science and Technology Innovation of Hunan Province
(No.~2026RC3168), the Natural Science Foundation of Hunan Province
(No.~2025JJ60068), the Scientific Research Fund of the Hunan Provincial
Education Department (No.~25B0150), the 111 Project (No.~D23017), and the
Program for Science and Technology Innovative Research Teams in Higher
Educational Institutions of Hunan Province, China.

\noindent \textbf{Generative AI statement.}
OpenAI's ChatGPT was used during revision to assist with language editing,
LATEX reorganization, literature checks, and mathematical consistency checks. The
authors assume responsibility for all content.

\noindent \textbf{Data Availability Statement:} This manuscript has no associated data.
\bigskip

\noindent \textbf{Conflict-of-Interest Statement:}
The authors declare no conflict of interest.

\end{document}